\documentclass[oneside]{book}
\usepackage{uiucthesis-amo}

\renewcommand{\bibname}{Bibliography}

\usepackage{amsmath,amsthm,amsfonts,amssymb}

\newcommand{\mynewtheorem}[2]{\newtheorem{#1}[equation]{#2}}

\theoremstyle{plain}
\mynewtheorem{theorem}{Theorem}
\mynewtheorem{lemma}{Lemma}
\mynewtheorem{cor}{Corollary}
\mynewtheorem{corollary}{Corollary}
\mynewtheorem{proposition}{Proposition}
\mynewtheorem{result}{Result}

\theoremstyle{definition}
\mynewtheorem{definition}{Definition}
\mynewtheorem{convention}{Convention}

\theoremstyle{remark}
\mynewtheorem{comment}{Comment}
\mynewtheorem{remark}{Remark}
\mynewtheorem{example}{Example}

\newcommand{\lb}{\ensuremath{\langle}}
\newcommand{\rb}{\ensuremath{\rangle}}
\newcommand{\gen}[1]{\ensuremath{\lb #1 \rb}}

\newcommand{\suchthat}{\ensuremath{\thinspace{}|\thinspace{}}}

\newcommand{\isomorphic}{\ensuremath{\cong}}
\newcommand{\isom}{\isomorphic}

\newcommand{\overbar}[1]{\overline{#1}}

\newcommand{\tensor}{\ensuremath{\otimes}}

\newcommand{\basept}{\ast}

\newcommand{\abs}[1]{\lvert{#1}\rvert}

\DeclareMathOperator*{\hocolim}{hocolim}
\DeclareMathOperator*{\colim}{colim}

\DeclareMathOperator*{\holim}{holim}

\DeclareMathOperator{\Map}{Map}
\DeclareMathOperator{\Hom}{Hom}

\newcommand{\cat}[1]{\mathcal{#1}}

\usepackage[matrix,arrow]{xy}

\usepackage{hyperref}

\DeclareMathOperator{\Obj}{Obj}
\DeclareMathOperator{\Skel}{Skel}
\DeclareMathOperator{\diag}{diag}
\DeclareMathOperator*{\hofib}{hofib}
\DeclareMathOperator*{\fib}{fib}
\DeclareMathOperator{\Surj}{Surj}
\DeclareMathOperator{\Perp}{\bot}

\newcommand{\LoopInfty}{\Omega^{\infty}}
\newcommand{\SigmaInfty}{\Sigma^{\infty}}
\newcommand{\Set}[1]{\{{#1}\}}
\newcommand{\strictrealization}[1]{\abs{#1}}
\newcommand{\realization}[1]{{\lvert\lvert{#1}\rvert\rvert}}
\newcommand{\Power}{\mathcal{P}}
\newcommand{\cube}[1]{{\mathcal{#1}}}
\newcommand{\res}{\rvert}

\newcommand{\Ord}[1]{[{#1}]}

\newcommand{\Pnd}{P_n^{d}}
\theoremstyle{definition}
\mynewtheorem{assumption}{Assumption}
\mynewtheorem{notation}{Notation}

\theoremstyle{plain}
\mynewtheorem{condition}{Condition}

\DeclareMathOperator{\Image}{Im}       
\DeclareMathOperator{\coker}{coker}    

\theoremstyle{definition}              
\mynewtheorem{hypothesis}{Hypothesis}  

\newcommand{\PerpStrict}{\Perp^{\text{strict}}} 

\newcommand{\undercat}{\mathbin{\setminus}}
\newcommand{\overcat}{\mathbin{/}}

\usepackage{amsbsy}
\newcommand{\BoldPerp}{\boldsymbol{\Perp}}

\DeclareMathOperator*{\cofib}{cofib}

\DeclareMathOperator{\Sing}{Sing}

\DeclareMathOperator{\Free}{Free}

\DeclareMathOperator{\Tot}{Tot}

\begin{document}


%
%
\title{Goodwillie Calculi}
\author{Andrew John Mauer-Oats}
\department{Mathematics}
\schools{B.A., Williams College, 1993 \\
         M.S., University of Illinois, 1996}
\phdthesis
\degreeyear{2002}

\maketitle

\frontmatter


%
%
\leavevmode\vfill
\begin{center}
For 309, 917, and 5825.
\end{center}
\vfill
\clearpage
%
%
\chapter*{Acknowledgements}

There are many people come and gone, at the University of Illinois and
in the world at large, that have made it possible for me to write this
thesis and stay sane.

First, a celebration of the newest: Olivia has brightened my life in
ways I could never have imagined when I started this journey. 

Second: to my companion of ten years now, Jennifer, a hearty laugh for
what we have gotten ourselves into. May the years ahead be ever more
fun. 

Next comes my advisor, Randy McCarthy, without whose positive outlook
this project would have been abandoned long before its time, and who
should know that he would be at the top of this list, but there are
things more important in life than a Ph.D.

Also, a thanks to my thesis committee, and Brenda Johnson, for the
otherwise thankless task of reading this document. A special thanks to
Dan Grayson for the time and care he put into his reading and
subsequent suggestions.

In the jumble of the rest, let me cite thanks to Kristine Baxter, Ben
Richert, and a sun-filled corner office for motivation to consider
deeply questions as the story of topology began to unfold for me.
Further thanks to Vahagn Minasian for always having something
interesting to talk about when I was tired of my own work.
Even more, a heartfelt thanks to my whole family for their support,
especially my parents, and Carl, Dan, and Jean, all of whose energy and
interests have sustained me.

\cleardoublepage


\tableofcontents






\mainmatter

%
%
\chapter{Introduction}

The Goodwillie tower $\Set{P_n F}$ of a homotopy functor $F$ 
gives information about $F(X)$ only for $X$ within the ``radius of
convergence'', where $F(X) = \holim P_n F(X)$.  On the boundary of the
radius of convergence (where the connectivity of $X$ is one below that
required to guarantee convergence of the tower), one generally finds
difficult problems. For instance, the Goodwillie tower of the identity
functor from spaces to spaces (and other functors such as $Q(X)$ and $A(X)$)
converges for simply connected spaces; 
questions on the `edge of the radius of convergence' involve spaces
whose first homotopy group is nontrivial. In general, if $X$ is
a nilpotent space, one may prove the same theorems as if $\pi_1 X
=0$, but little can be said if $X$ is not nilpotent. This
indicates that the maximum possible set of convergence for
these functors should be ``nilpotent spaces'', not 
``simply connected spaces'', the answer produced by ordinary
Goodwillie calculus. 

The Goodwillie tower is based on the idea of approximating a functor
$F$ by a series of functors $P_n F$ satisfying the very strong
property of ``$n$-excision''. One might hope that by weakening this
condition, one might obtain a larger radius of convergence. We begin
our study by reviewing the weaker property of ``$n$-additive'' and
showing that for some functors, the weaker approximations give a
larger radius of convergence.

These new constructions feature the left Kan extension in a prominent
way. Briefly, given a full subcategory $\cat{C}$ of the category of
spaces $\cat{T}$, the left Kan extension $L_{\cat{C}}$ gives a way of
constructing the adjoint to the restriction map for functors from $\cat{T}$. 
That is, there is an adjoint isomorphism:
$$
\Hom_{\cat{C}} ( F, G \res_{\cat{C}} ) 
\cong 
\Hom_{\cat{T}} ( L_{\cat{C}} F, G)
.
$$
The importance to us of the left Kan extension is that it defines a
new functor $L_{\cat{C}} F(X)$ using only the behavior of $F$ on
$\cat{C}$ and something about the relationship of $\cat{C}$ to $X$.
Specifically, $L_{\cat{C}} F(X)$ depends only on $F(\cat{C})$ 
and $\Map(C,X)$ for objects $C$ of $\cat{C}$. In particular,
no objects resembling $F(\Map(C,X))$ appear.

It turns out that the $n$-additive approximation of $F$ can be
expressed as Goodwillie's $n$-excisive approximation applied to an
associated functor $L^0 F_X$, which is the left Kan extension (from
finite sets = coproducts of $S^0$) of the functor $F_X(-) = F(X \wedge -)$. 
This suggests that the first thing one should investigate is the left
Kan extensions of functors. We begin by investigating left Kan
extensions of functors from spaces to spectra, since that case is
generally much simpler to understand than the case of functors from
spaces to spaces. In this case, the left Kan
extensions along the full subcategory $\cat{C}_n$ generated by
$\Set{\bigvee^k S^0 \suchthat k\le n }$
turns out to classify all degree $n$ functors from spaces to spectra,
so we obtain a complete understanding of all functors of finite degree
from spaces to spectra in this way.

Functors from spaces to spaces are much more complicated, 
but we can
use our results on functors from spaces to spectra to understand the
Goodwillie derivatives $D_n F$ of any functor, since these functors
factor through the category of spectra as 
$D_n F(X) = \LoopInfty \left( \mathbf{C_n} \wedge X^{\wedge n} \right)$, 
for $\mathbf{C}_n$ a spectrum. In particular,
we show that if $\mathbf{C_n}$ is connective, then $D_n F$ commutes
with realizations. Using that result, we give a sufficient condition for
an analytic functor to commute with realizations. 

Once we understand functors from spaces to spectra as left Kan
extensions, we can ask to what extent left Kan extensions of functors
from spaces to spaces are interesting. The additive Goodwillie tower
arises from $P_n (L^0 F_X)$, and the (ordinary) excisive Goodwillie
tower arises from applying $P_n$ to the left Kan extension  $L^{\infty}$
over all
finite coproducts of spheres of the same dimension
$$\Set{\bigvee^k S^m \suchthat 0 \le m\le \infty \text{ and } k \ge 0}.$$
Between $L^{0} F$ and $L^{\infty} F$ lies an
infinite sequence of Kan extensions $L^a F$, arising from using
coproducts of spheres $S^m$, with $m\le a$,
equipped with natural
transformations $L^{a} F \rightarrow L^{a+1}F$.
This sequence
could give rise to an entire family of ``theories'' $P_n^{(a)}$
between the additive and excisive approximations. 
The first step toward showing that this tower is interesting is to show
that the approximations $P_n^{(a)} F$ can be distinct.
We produce a family of examples, one for each $a$, such that
$P_n^{(a)} F \not\simeq P_n^{(a+1)} F$. We then go on to show that if
$F$ is an analytic functor, the tower 
$$
P_n (L^0 F)  \rightarrow \cdots \rightarrow P_n(L^a F) \rightarrow
\cdots \rightarrow P_n (L^{\infty} F) 
$$
stabilizes at a finite stage, so that Goodwillie's $P_n F$ can
actually be computed by examining a left Kan extension of finite
dimension. This is interesting because the left Kan extension $L^{a}
F$ requires ``less'' information than $F$ to compute, since it depends
only on the subcategory $\cat{C}$ and maps from objects $C \in \cat{C}$
to $X$, and not arbitrarily high suspensions of $X$ as $P_n F(X)$
requires. 
Fundamentally, even $P_n^{(0)}$ is an interesting functor, and
understanding it is a necessary prerequisite to understanding the
filtration of theories $P_n^{(a)}$. The $n^{\text{th}}$ cross effect
functor of $F$ at $X$, denoted $\Perp_n F(X)$, measures how 
much $F(\bigvee^n X)$ fails to be determined by the value of $F$ on
smaller coproducts of $X$. Since this is essentially exactly the
``information'' available to compute $P_n^{(0)} F(X)$, there should be
a very close relationship between the two.
One of the main results in this thesis is
Theorem~\ref{thm:main-theorem}, which establishes that if $F$ is
reasonably good, there is a fibration sequence:
$$ \realization{\Perp_{n+1}^{*+1} F(X)}
 \rightarrow F(X) \rightarrow P_n^{(0)} F(X).$$
As a consequence of this theorem, we derive a spectral sequence with
$E^1_{p,q} = \pi_p \Perp_{n+1}^q F(X)$ converging to $\pi_{p+q}
P_n^{(0)} F(X)$.
Also, this theorem gives us a way of relating the Goodwillie
tower of the identity functor of (simplicial) groups to a derived
functor of the lower central series.

This thesis is organized as follows. In Chapter~\ref{chap:categories},
we review the basic categories and constructions used throughout. In
Chapter~\ref{chap:goodwillie-calculus}, we give a brief exposition of
some of Goodwillie's calculus of functors, including $n$-cubes, the
Blakers-Massey theorem, and some basic examples. In
Chapter~\ref{chap:left-kan}, we explain the left Kan extension and its
homotopy invariant counterpart. In Chapter~\ref{chap:finite-degree},
we show that all degree $n$ functors from spaces to spectra are
left Kan extensions over the full subcategory of spaces containing
$\bigvee^k S^0$, for $k=0,\ldots,n$. In
Chapter~\ref{chap:realization}, we show several results.
Section~\ref{sec:analytic-functors-connective} shows that analytic
functors have connective coefficient spectra.
Section~\ref{sec:loopinfty-commutes-with-certain-realizations} shows
that $\LoopInfty$ commutes with realizations of simplicial connective
(\emph{i.e.}, bounded at $\pi_0$) spectra.
Section~\ref{sec:analytic-realizations} combines these results to show
that (reduced) analytic functors from spaces to spaces commute with
realizations of simplicial $k$-connected spaces, where $k$ is the
larger of the radius of convergence or $-c$ for the universal
analyticity constant $c$ (see \S\ref{sec:analytic-functors}).
Chapter~\ref{chap:cotriples} gives background on cotriples.
Chapter~\ref{chap:properties-of-Pnd} establishes basic properties of
the $P_n^{(0)}$ and $\Perp_n$ constructions.
In Chapter~\ref{chap:main-theorem}, we prove the main theorem
(\ref{thm:main-theorem}), which
establishes the relationship between 
the functor $P_n^{(0)}F$ and the $(n+1)^{\text{st}}$ cross effect.
Chapter~\ref{chap:main-consequences} elucidates some of the
consequences of the main theorem, including the existence of a
spectral sequence to calculate $P_n^{(0)} F(X)$ and the relationship
with the work of Curtis on the lower central series of a simplicial group.
Chapter~\ref{chap:different-theories} shows that there is a whole
family of different theories interpolating between additive and
excisive calculus, and all are distinct. 


%
%
\chapter{Categories And Homotopy Invariance}
\label{chap:categories}
There are two main categories we will study: pointed spaces and
spectra. We will also be interested in simplicial objects in both of
these categories. In this section, we explain exactly what we mean by
these categories, and give a brief synopsis of the properties that we
use.


\section{Spaces}
\label{sec:categories-spaces}

By ``spaces'' or ``topological spaces'', we mean the 
topological category of 
compactly generated Hausdorff topological spaces with nondegenerate basepoint.
In this category $\Hom$ is itself a topological space using the
(compactly generated) compact-open topology. When we want to emphasize
its nature as a space, we will write $\Map$.
Many convenient properties (such the continuity of the evaluation map
from $X \times \Map(X,Y)$ to $Y$) always hold in this category.
The formation of various categorical constructions, such as product,
requires a ``compactification'' of the topology on the product for
arbitrary spaces. This is to be done implicitly wherever necessary.
See \cite{Steenrod:a-convenient-category-of-topological-spaces} for
more information about this. Henceforth, unless otherwise stated, the
term ``space'' or ``topological space'' will mean an object of this
category.

The pointed category has an object that is both initial and final;
we use both $\basept{}$ and $0$ to denote this object, depending on
context. 

\section{Spectra}

Spectra are the ``stable category'' associated to spaces. The
so-called na\"{\i}ve spectra will be sufficient for our purposes.
References for this material include Adams
\cite{Adams:stable-homotopy-and-generalised-homology}
and Kochman
\cite[Chapter~3.3]{Kochman:bordism-stable-homotopy-and-adams-spectral-sequences},
who also follows Adams' treatment. 
In this
category, a spectrum $\mathbf{X}$ consists of a sequence of
topological spaces $\Set{X_i \suchthat i\ge 0}$ and 
structure maps $\Sigma X_{i} \rightarrow X_{i+1}$. 
By adjunction, the structure maps may also be
specified by a map $X_{i} \rightarrow \Omega X_{i+1}$; this is
sometimes more convenient --- see the example of $\mathbf{H}G$ below.
A morphism $f:\mathbf{X}\rightarrow\mathbf{Y}$ in this category is a
(cofinal) sequence of maps $f_i: X_i \rightarrow Y_i$ that commute with the
structure maps. (Cofinal means that the maps need not be defined for
all $i$; just on a cofinal subset of indices.)
The homotopy groups here
are $\pi_n \mathbf{X} = \colim_i \pi_{n+i} X_i$; a spectrum may have
negative homotopy groups.

A nontrivial example of a spectrum is the Eilenberg-MacLane spectrum
$\mathbf{H}G$. The spectrum $\mathbf{H}G$ has $\mathbf{H}G_n =
K(G,n)$, and the structure map is given by the canonical equivalence
$K(G,n) \rightarrow \Omega K(G,n+1)$. Its only nonzero homotopy group
is $\pi_0 = G$. This is an example of an ``omega spectrum'':
$\mathbf{X}$ is called an omega spectrum if the adjoint structure map
$X_n \rightarrow \Omega X_{n+1}$ is an equivalence. Every spectrum is
equivalent to an omega spectrum.
Another example is the sphere spectrum, frequently
denoted $\mathbf{S}$ or $\mathbf{S^0}$, given by $(\mathbf{S})_n = S^n$,
the $n$-sphere, and structure maps 
$\Sigma S^n \xrightarrow{\cong} S^{n+1}$.
Its homotopy is stable homotopy, $\pi_n \mathbf{S} = \pi_n^S S^0$.
This is a ``suspension spectrum''--- one in which the structure map
$\Sigma X_n \rightarrow X_{n+1}$ is an isomorphism.

The homotopy category of spectra is a triangulated category,
much like the homotopy category of chain complexes. 
Just as in the case of chain complexes, it
is sometimes desirable to distinguish between arbitrary spectra and
``bounded below'' spectra, whose homotopy $\pi_n$ vanishes for all
$n\le N$. The main important trait of bounded below spectra is that
suspension increases their connectivity. We use the word
``\emph{connective}'' to mean a spectrum that has no negative homotopy
groups. In the literature, the word connective sometimes means bounded below. 

The categories of spectra and spaces are related by a pair of adjoint
functors, $\LoopInfty$ and $\SigmaInfty$. The functor $\SigmaInfty:
\text{Spaces} \rightarrow \text{Spectra}$ creates a spectrum from a
space $X$ by putting $(\SigmaInfty X)_n = \Sigma^n X$, with the
structure maps $\Sigma (\SigmaInfty X)_n \xrightarrow{=} (\SigmaInfty
X)_{n+1}$. The functor $\LoopInfty: \text{Spectra} \rightarrow
\text{Spaces}$ sends $\mathbf{X}$ to $\colim_n \Omega^n X_n$. Note
that $\pi_n (\LoopInfty \mathbf{X}) = \pi_n \mathbf{X}$ for all $n\ge
0$. In particular, $\pi_0$ and $\pi_1$ of an ``infinite loop space''
are abelian groups. 
The adjunction that arises from the familiar
suspension-loop adjunction is
\begin{equation*}
\Hom_{\text{Spectra}} (\SigmaInfty X, \mathbf{Y}) 
\cong
\Hom_{\text{Spaces}} ( X, \LoopInfty \mathbf{Y})  .
\end{equation*}
When working with unbased spaces, the appropriate functor to use is
$\Sigma^{\infty}_{+}(X)$, which is the suspension spectrum of $X$
taken after a disjoint basepoint is added.

The category of spectra has the very useful property that fibration
sequences and cofibration sequences are equivalent.
The proof of this
uses the Blakers-Massey theorem (\ref{thm:blakers-massey}), so it
appears later, as Corollary~\ref{cor:spectra-cofibration=fibration} on
page~\pageref{cor:spectra-cofibration=fibration}. The equivalence of
fibration and cofibration sequences implies that the fiber of a
map is naturally equivalent to the loop spectrum of the cofiber of the map. 


When we have groups acting on spectra, we will always be in a
situation where it is appropriate to use na\"{\i}ve $G$-spectra. These
are spectra $\mathbf{X}$ in which $G$ acts on each $X_n$, and the
structure maps $S^1 \wedge \mathbf{X}_n \rightarrow \mathbf{X}_{n+1}$
are $G$-equivariant, with $G$ acting trivially on the suspension
coordinate. The simplest $G$-spectra are those which are suspension
spectra of spaces with a free $G$-action.


\section{Functors}
\label{sec:categories-functors}

The category of topological spaces (or spectra) is enriched over
topological spaces (respectively, spectra), meaning that the $\Hom$
sets can be given the structure of a topological space (respectively,
spectrum). We require that our functors respect this additional
structure. 


Let $\cat{C}$ and $\cat{D}$ be topological categories.  Let $F:
\cat{C} \rightarrow\cat{D}$ be a functor. In the standard terminology, 
$F$ is called \emph{continuous} if the map $f \mapsto F(f)$ induces a
continuous map  $\Hom_{\cat{C}} (A,B) \rightarrow \Hom_{\cat{D}} (FA, FB)$.


We require that all functors be continuous.


\section{Simplicial Objects}

We will also be interested in the simplicial objects in the 
categories of spaces and spectra: simplicial spaces and simplicial spectra. 
The standard
reference for all facts about simplicial objects is May's book
\cite{May:simplicial-objects-in-algebraic-topology}, but Curtis's
award-winning exposition \cite{Curtis:simplicial-homotopy-theory} is a more accessible
place to start. Weibel~\cite[Chapter 8]{Weibel:homological-algebra} is
a concise but valuable reference.  The recent publication by
Goerss and Jardine \cite{Goerss-Jardine:simplicial-homotopy-theory} is
another resource for facts about simplicial homotopy theory. The
reader completely unfamiliar with the subject is advised to
consult one of these references; this is just a very brief review of
some relevant facts.

Before delving into the definitions (which are notoriously opaque),
let us consider why simplicial objects are so important. In general,
simplicial objects add another ``dimension'' to a category; for
instance, simplicial abelian groups are equivalent to chain complexes
(bounded $\geq 0$) of abelian groups (this is known as the Dold-Kan
correspondence). Adding this dimension provides a setting for
homological algebra by providing a category in which projective
resolutions of abelian groups can live. In this case, there is no way
to ``reduce'' a projective resolution back down to an ordinary abelian
group without losing the information it provides.
In the case of
simplicial spaces, however, the base category (spaces) already has
enough structure that it is possible to reduce
a simplicial space back down the an ordinary space without losing
information.\footnote{This
  actually might be said to occur because spaces are equivalent to
  simplicial sets, so they already ``contain one simplicial
  dimension'', and the Eilenberg-Zilber theorem
  (\ref{thm:eilenberg-zilber}) shows that nothing more is gained by
  adding more simplicial dimensions.}
This process is called ``realization'', and plays a
central role in the work in this thesis.

Let $\Delta$ denote the category of ordered finite sets
whose objects are $\Set{[n] \suchthat n\ge 0}$, with $[n] = \Set{0 <
  \cdots < n}$, and whose morphisms are nondecreasing set maps. A
simplicial object in any given category is a functor from the opposite
category of $\Delta$, denoted
$\Delta^{\text{op}}$, to the given category. The behavior of a functor
$\Delta^{\text{op}} \rightarrow \cat{C}$ is determined by its values
on the objects and on certain morphisms called ``face'' and
``degeneracy'' maps. In $\Delta$, there are $n+1$ face maps $\delta_i:
[n]\rightarrow{} [n-1]$:
\begin{equation*}
\delta_i(j) = 
\begin{cases}
j & \text{if $j\le i$} \\
j-1 & \text{if $j > i$}
\end{cases},
\end{equation*}
and $n+1$ degeneracy maps $\sigma_i: [n]\rightarrow{} [n+1]$:
\begin{equation*}
\sigma_i(j) = 
\begin{cases}
j & \text{if $j< i$} \\
j+1 & \text{if $j\ge i$}
\end{cases}
.
\end{equation*}
Generally a simplicial object is denoted by $X_{\cdot}$, and its value
on objects is $X_n = X([n])$. The image of the face maps are the $d_i =
X(\delta_i)$, and the image of the degeneracy maps are the $s_j = X(\sigma_j)$.

Let $\Delta^n$ denote the
standard $n$-simplex.
The realization $\realization{X_\cdot}$ of a simplicial space
$X_{\cdot}$ is taken to 
be the colimit of the following process. 
Let $R_0 = X_0$, and proceed by
induction to let $R_n$ be the pushout of the following diagram:
$$
\xymatrix{
R_{n-1}
&
X_n \times \partial \Delta^n
\ar[r]
\ar[l]
&
X_n \times \Delta^n
}
,
$$
where the left map is given by $(x,p) \mapsto (d_i x)$ when $p$ is
an element of the $i$-th face of $\Delta^n$.  
The definition of
realization given here is called the ``fat'' realization; when we need
to refer to the usual definition (which uses the quotient of $X_n$ by
the degeneracies where we have used $X_n$), we will say ``strict
realization'' and denote it $\strictrealization{X_\cdot}$. 
As discussed in Section~\ref{sec:homotopy-functors} below, if the
degeneracy maps $s_j: X_{n-1} \rightarrow X_n$ are not cofibrations,
the quotient may not be a ``homotopy invariant''. However, 
in the case of simplicial sets, all injections are cofibrations, so
there is never an issue when working with simplicial sets.

One very important fact about realization is that it is homotopy
invariant in the following sense.
\begin{lemma}[Realization Lemma]
\label{lem:realization-lemma}
(\cite[Proposition~A.1, p.~308]{Segal:categories-and-cohomology-theories})
Let $X_{\cdot}$ and $Y_{\cdot}$ be simplicial spaces, and suppose
$f_{\cdot}:  X_{\cdot} \rightarrow  Y_{\cdot}$ is a simplicial map
with $f_n$ a weak equivalence for all $n$. Then
$\realization{f_{\cdot}}$ is a weak equivalence. $\qed$
\end{lemma}

We will sometimes want to use the strict realization and know that it
is a homotopy invariant. This happens if the simplicial space being
realized is ``good''.
\begin{definition}[Good simplicial space]
\label{def:good-simplicial-space}
A simplicial space is called \emph{good} if all of the degeneracy maps
$s_j: X_n \rightarrow X_{n+1}$ are closed cofibrations.
\end{definition}

If a  simplicial space is good, then both the ``fat'' and ``strict''
realizations are equivalent. 

\begin{theorem}
\label{thm:good-fat-strict-realization}
(\cite[Proposition~A.2, p.~308]{Segal:categories-and-cohomology-theories})
If $X$ is a good simplicial space, then the natural map 
$\realization{X} \rightarrow \strictrealization{X}$ is a (weak
homotopy) equivalence.
$\qed$
\end{theorem}

\begin{corollary}
\label{cor:good-multisimplicial}
Let $X_{\cdot\cdot}$ be a bisimplicial space. If each simplicial space
$[j] \mapsto X_{i,j}$ is good and $[i] \mapsto X_{i,j}$ is 
good, then the natural map 
$\realization{X} \rightarrow \strictrealization{X}$
between the realizations in one direction and another is a weak
homotopy equivalence. For
this reason, we call such bisimplicial spaces ``good'' as well.
\end{corollary}
\begin{proof}
Since the realization of levelwise cofibrations is a
cofibration, it suffices to show that the degeneracy maps in 
each simplicial direction of a multi-simplicial space are
cofibrations. (Our spaces are all Hausdorff by hypothesis, so
cofibration implies closed.) 
Taking the realizations in one direction at a time, this follows from
Theorem~\ref{thm:good-fat-strict-realization}.
\end{proof}

\begin{lemma}
\label{lem:preserve-goodness}
Many operations preserve ``goodness''. Let $X$ be a good
simplicial space (where each space has a nondegenerate
basepoint),  
and let $Z$ be a space (with a nondegenerate basepoint). 
Then the following simplicial spaces are good:
\begin{enumerate}
\item $( X_\cdot )_+$ (even if $X_\cdot$ does not have a nondegenerate
basepoint) 
\item $ Z \vee X_\cdot$
\item $ Z \times X_\cdot$
\item $ Z \wedge X_\cdot$
\item $\Map(C,X_\cdot)$, for any compact cofibrant $C$
\end{enumerate}
If each $X(k)_\cdot$ is a good simplicial space, then:
\begin{itemize}
\item[6.] $diag \left( \bigvee_{i=1}^k X(i)_\cdot \right)$ is good
\end{itemize}
And finally, if $[i,j] \mapsto X(i)_j$ is a bisimplicial space with
each $X(i)$ good, then:
\begin{itemize}
\item[7.] the 
realization $\realization{[i] \mapsto
  X(i)_\cdot}$ is a good simplicial space
\end{itemize}
\end{lemma}
\begin{proof}
We only need to make these arguments in the category of spaces, so
when it is convenient, we can use a characterization of cofibrations
that is specific to that category.

Item $1$: obvious.
Item $2$: coproduct (colimit) of cofibrations is a cofibration.
Item $3$: follows from characterization of cofibrations via
neighborhood deformation retracts (as in 
\cite[Theorem~VII.1.5, p.~431]{Bredon:topology-and-geometry}).
Item $4$: the map is question is the pushout (colimit) of vertical
cofibrations in the following diagram:
$$\xymatrix{
\basept 
\ar@{>->}[d]
&
\ar@{>->}[d]
Z \vee X_n
\ar[l]
\ar[r]
&
\ar@{>->}[d]
Z \times X_n
\\
\basept &
Z \vee X_{n+1}
\ar[l]
\ar[r]
&
Z \times X_{n+1}
}$$
Item $5$: the neighborhood retraction for $X_n$ in $X_{n+1}$ induces a
neighborhood retraction of $\Map(C,X_n)$ in $\Map(C,X_{n+1})$ using
the height function $\phi(f) = \sup_{c\in C} \phi_{X}(f(c))$ 
derived from the height function $\phi_X$ for $s_n$.
Item $6$: For a finite coproduct, this follows from (2) along with the
fact that the composition of cofibrations is a cofibration (the
diagonal degeneracies come from composing the degeneracies in each
individual direction).
Item $7$: Each map $s_{i,j} : X(i)_j \rightarrow X(i)_{j+1}$ is a
cofibration, so the realization in the $i$ direction preserves the
cofibrations, producing a cofibration $\realization{X(\cdot)_j} \rightarrow
\realization{X(\cdot)_{j+1}}$. 
\end{proof}


In this paper we work with the category of
topological spaces because homotopy inverse limit constructions are
very important, and these require fibrant objects to be well-behaved.
When working with simplicial sets, it is more effort to maintain
fibrancy. However, some standard results that we use are proven for
bisimplicial spaces, so we need to establish that they also hold for
topological spaces. 

To this end, we recall some facts about simplicial and bisimplicial
sets. Given a simplicial space $X$, the singular set functor,
$\Sing(X)$ produces a simplicial set whose $k$-simplices are the set
(not space) of continuous maps of the standard topological $k$-simplex
into $X$; 
that is, $\Hom(\Delta^k,X)$. The functor $\Sing$ is right adjoint to
the strict realization functor, and the map
$\strictrealization{\Sing(X)}\rightarrow X$ is always a weak equivalence. 
These facts and more can be found in
\cite[Chapter~1]{Goerss-Jardine:simplicial-homotopy-theory}.
Given a bisimplicial set $X$, there is a
functor ``Tot'' that produces a simplicial set. Let $\Delta[m]$ be the
standard simplicial $n$-simplex, $\Hom_\Delta(-,[m])$, let $\Delta^m$
be standard $m$-simplex that is the strict realization of $\Delta[m]$,
 and let
$X_{m,*}$ denote the simplicial set $[k] \mapsto X_{m,k}$.
This functor ``Tot'' can be
described as the coequalizer of the diagram:
$$
\xymatrix{
{\displaystyle
\bigsqcup_{\alpha: [m]\rightarrow{} [n]} X_{m,*} \times \Delta[n]
}
\ar@<1.0ex>[r]
\ar@<0.0ex>[r]
&
{\displaystyle
\bigsqcup_{[m]} X_{m,*} \times \Delta[m]
}
},$$
where the first coproduct is taken over all morphisms in $\Delta$, and
the second is taken over all objects in $\Delta$. The first morphism
sends $(x,y)$ to $(\alpha^* x,y)$ and the second morphism sends
$(x,y)$ to $(x,\alpha_* y)$.

Applying the \emph{strict} geometric realization functor (which
commutes with coproducts and finite products) to this
diagram produces a diagram
$$
\xymatrix{
{\displaystyle
\bigsqcup_{\alpha: [m]\rightarrow{} [n]} 
\strictrealization{X_{m,\cdot}} 
\times 
\Delta^n
}
\ar@<1.0ex>[r]
\ar@<0.0ex>[r]
&
{\displaystyle
\bigsqcup_{[m]} 
\strictrealization{X_{m,\cdot}} 
\times 
\Delta^m
}
}.$$
The coequalizer of this
diagram is the realization $\strictrealization{[i] \mapsto
  \strictrealization{[j]\mapsto X_{i,j}}}$. But strict realization is a left
adjoint, and hence preserves coequalizers, so we have established:
\begin{lemma}
\label{lem:realization-tot}
  Let $X$ be a bisimplicial set. Then using strict realizations,
  $\strictrealization{\Tot(X_{\cdot\cdot})}$ is isomorphic (homeomorphic) to 
  $\strictrealization{[i] \mapsto \strictrealization{[j]\mapsto X_{i,j}}}$. 
  $\qed$
\end{lemma}

\begin{corollary}
\label{cor:tot-realization}
  Let $X_\cdot$ be a simplicial space and $\Sing_\cdot X_\cdot$ be the
  bisimplicial set formed by applying the singularization functor to
  each $X_i$. Then we have:
  $$\realization{\Tot(\Sing_\cdot X_\cdot)} \simeq \realization{X_\cdot}.$$
\end{corollary}
\begin{proof}
  Let $Y$ be the simplicial set $Y_{i,j} = \Sing_j X_i$. 
  Since $\Tot(Y)$ is a simplicial set, the (fat) realization and
  strict realizations are equivalent, so we can work with the strict
  realization. Lemma~\ref{lem:realization-tot} then gives: 
  $$\strictrealization{\Tot(Y_{\cdot\cdot})}
  \cong 
  \strictrealization{[i] \mapsto \strictrealization{[j]\mapsto Y_{i,j}}}.$$ The
  inner realization on the right is $\strictrealization{[j]\mapsto Y_{i,j}}
  = \strictrealization{\Sing(X_i)} \simeq X_i$. The functor $\Sing$ takes
  inclusions to cofibrations, so the simplicial space
  $\strictrealization{\Sing(X_i)}$ is good; hence the strict realization
  that appears here is equivalent to the fat realization. We can then
  use the fact that the fat realization is a homotopy functor, so the
  weak equivalences   $\strictrealization{\Sing(X_i)}\simeq X_i$ induce an
  equivalence of (fat) realizations: 
  $$\realization{[i]\mapsto \strictrealization{\Sing(X_i)}}
  \xrightarrow{\simeq}
  \realization{[i] \mapsto X_i }
  $$
  Chaining the equivalences together produces the desired result.
\end{proof}

The degeneracy maps encode ``redundant'' information that is necessary
for the proper homotopical behavior of the object. One important
consequence of the existence of the degeneracy maps is the
Eilenberg-Zilber theorem. 
The Eilenberg-Zilber theorem for bisimplicial sets 
relates the $\Tot$ of a bisimplicial space to its
diagonal. The diagonal of a bisimplicial object is
$\diag(X_{\cdot\cdot})_n = X_{n,n}$.

\begin{theorem}[Eilenberg-Zilber] 
(\cite[Proposition~B.1, p.~119]{Bousfield-Friedlander:Gamma-Spaces})
\label{thm:eilenberg-zilber}
Let $X$ be a bisimplicial set. There is a natural
isomorphism of simplicial sets $\Tot(X) \xrightarrow{\cong} \diag(X)$.
$\qed$
\end{theorem}

We actually want to use the following statement for bisimplicial
spaces:
\begin{corollary}
\label{cor:useful-eilenberg-zilber}
Let $X_{\cdot\cdot}$ be a good  bisimplicial space. 
The realization in 
one direction and then the other,
$\lvert{X_{\cdot\cdot}}\rvert$, is naturally homotopy equivalent to
$\realization{\diag{X_{\cdot\cdot}}}$. Realization is a homotopy
colimit, and homotopy colimits commute up to
natural isomorphism, so the order in which the
realizations are taken does not matter.
\end{corollary}
\begin{proof}
Since $X$ is good, the (fat)
realization $\realization{X_{\cdot\cdot}}$ in the statement of the
theorem  is equivalent to 
the strict realization. Using strict realizations, we have a homeomorphism:
$$\strictrealization{\diag(X)} 
\xrightarrow{\cong} 
\strictrealization{\Tot(X)}\cong
\strictrealization{[i]\mapsto \strictrealization{[j]\mapsto X_{i,j}}} .$$
The construction of $\Tot$ for bisimplicial sets prior to
Lemma~\ref{lem:realization-tot}, and the maps in the Eilenberg-Zilber
theorem, have direct translations to bisimplicial spaces once we use
the strict realization. This translation gives the equivalence above;
it remains to check that $\diag(X)$ is a good space, so that its
strict realization agrees with the (fat) realization. The degeneracies
in $\diag(X)$ are compositions of ``horizontal'' and ``vertical''
degeneracies, both of which are cofibrations by our hypothesis that
$X$ is a good bisimplicial space, so $\diag(X)$ is a good simplicial
space. 
\end{proof}


The realization of a levelwise fibration of simplicial spaces need not
be a fibration, but with some conditions it is. The following lemma is
stated in \cite{Waldhausen:K-theory-of-free-products-1} for bisimplicial
sets; we will not argue that it is also true for simplicial spaces,
since it is also an easy corollary of
Theorem~\ref{thm:Bousfield-Friedlander}, below. 
\begin{lemma}
\label{lem:realization-of-levelwise-fibration}
\label{lem:Waldhausen-fibration-lemma}
(\cite[Lemma~5.2, p.~165]{Waldhausen:K-theory-of-free-products-1})
Let $X_\cdot \rightarrow Y_\cdot \rightarrow Z_{\cdot}$ be map of
simplicial spaces such that each $X_n \rightarrow Y_n \rightarrow Z_n$
is a fibration up to homotopy. If each $Z_n$ is connected, then
$\realization{X_{\cdot}} \rightarrow \realization{Y_{\cdot}}
\rightarrow \realization{Z_{\cdot}}$ is a fibration up to homotopy.
$\qed$ 
\end{lemma}

We use a generalization of this result to $2$-cubes, due to Bousfield
and Friedlander, heavily later in this work. They define a fibrancy
condition called the $\pi_*$-Kan condition.
\begin{definition}
\label{def:pi-star-Kan}
  A simplicial space $X_{\cdot}$ is said to satisfy the $\pi_*$-Kan
  condition if:
  \begin{itemize}
  \item for any $m\ge 1$ and any $t\ge 1$, and for any point $a\in
    X_m$, any coherent collection  (in the sense of the usual fibrancy
    condition: $\partial_i x_j = \partial_{j-1} x_i$ for $i>j$ with $i,j\not=k$) 
    of elements  $x_i \in \pi_t(X_{m-1},\partial_i a)$ (for
    $0\le i \le m$, and $i\not= k$), there exists a $y\in \pi_t(X_m,a)$
    with $\partial_i y = x_i$; 
    and
  \item the simplicial set $\pi_0(X_{\cdot})$ is fibrant.
  \end{itemize}
\end{definition}
For instance, a simplicial space $X_{\cdot}$ certainly satisfies the
$\pi_*$-Kan condition if each $X_i$ is connected. Also, simplicial
spaces arising from bisimplicial groups satisfy the $\pi_*$-Kan
condition. 
\begin{theorem}[Bousfield-Friedlander]
\label{thm:Bousfield-Friedlander}
(\cite[Theorem~B.4, p.~121]{Bousfield-Friedlander:Gamma-Spaces})
Let 
$$\xymatrix{
V_\cdot
\ar[r]
\ar[d]
&
X_\cdot
\ar[d]
\\
W_\cdot
\ar[r]
&
Y_\cdot
}$$
be a commutative square of simplicial spaces such that for each $n$,
the square consisting of $V_n$, $W_n$, $X_n$, and $Y_n$ is a homotopy
pullback square. If $X$ and $Y$ satisfy the $\pi_*$-Kan condition and
if $\pi_0 X_\cdot \rightarrow \pi_0 Y_\cdot$ is a fibration of
simplicial sets, then after realization we have a homotopy
pullback square:
$$\xymatrix{
\realization{V_\cdot}
\ar[r]
\ar[d]
&
\realization{X_\cdot}
\ar[d]
\\
\realization{W_\cdot}
\ar[r]
&
\realization{Y_\cdot}
}$$
\end{theorem}
Bousfield and Friedlander actually prove their theorem for
bisimplicial sets, so some comments are in order to apply it to
simplicial spaces. The functor $\Sing$ is a right adjoint, so it
preserves inverse limits; in particular, if $\cube{X}$ is a homotopy
pullback cube of spaces, then $\Sing \cube{X}$ is a homotopy pullback
cube of simplicial sets. The bisimplicial set $\Sing Y_\cdot$
satisfies the $\pi_*$-Kan condition if and only if the simplicial
space $Y_\cdot$  does, since it is a condition on homotopy groups, and
the homotopy groups of $\Sing Y$ and $Y$ are isomorphic for any space
$Y$. Now starting with a commutative square of simplicial spaces
satisfying the hypotheses stated above, we apply $\Sing$ to produce
a commutative square of bisimplicial sets satisfying the analogous
hypotheses used by Bousfield and Friedlander. Their result is then
that the square of simplicial sets formed by taking the diagonal is
Cartesian. Then $\diag (\Sing_\cdot X_\cdot) \cong \Tot(\Sing_\cdot
X_\cdot)$ and $\realization{\Tot(\Sing_\cdot
X_\cdot)} \simeq \realization{X_\cdot}$
(Corollary~\ref{cor:tot-realization}), producing the result as we 
state it.
$\qed$

The realization of a simplicial spectrum requires that we define both
the spaces in the realization and structure maps.
Begin with a simplicial spectrum $[m] \mapsto \mathbf{X}_m$ with the
spectrum $\mathbf{X}_m$ consisting of spaces $X_{m,n}$ and structure
maps $S^1 \wedge X_{m,n} \rightarrow X_{m,n+1}$.
Define the realization of this simplicial spectrum
$\realization{[m]\mapsto \mathbf{X}_m}$ to have $n^{\text{th}}$ space 
$\realization{[m] \mapsto X_{m,n}}$.  Recall that the
suspension of $X$ is homeomorphic to $\hocolim\left( \basept
  \leftarrow X \rightarrow \basept \right)$ in the category of pointed
spaces. The structure maps are
 given by commuting the realization (which is a homotopy colimit) with
 the suspension (which is also a homotopy colimit) and using the
 structure map of each $\mathbf{X}_m$ in the following manner:
\begin{align*}
  S^1 \wedge \realization{[m]\mapsto \mathbf{X}_m}_n
&=
  S^1 \wedge \realization{[m]\mapsto {X}_{m,n}}
\\
&\simeq
\realization{[m] \mapsto S^1 \wedge X_{m,n}}
\\
&\rightarrow
\realization{[m] \mapsto X_{m,n+1}}
\\
&=
\realization{[m] \mapsto \mathbf{X}_m}_{n+1}
\end{align*}


\section{The Nerve Of A Category}

A category $\cat{C}$ determines a simplicial set called the nerve
of $\cat{C}$, denoted $N_{\cdot} \cat{C}$. The $n$-simplices of this
object consist of $n$ composable morphisms in the category; for $n=0$,
we define $N_0\cat{C} = \Obj(\cat{C})$ (or alternatively, consider
only the identity morphisms). The face maps are given by composing two
adjacent morphisms, or deleting them at the extrema, and the
degeneracy maps are given by inserting identity morphisms. Explicitly,
let $\alpha\in N_n \cat{C}$ be a sequence of $n$ composable morphisms:
$$\alpha = ( C_n \xrightarrow{\alpha_{n-1}} \cdots
  \xrightarrow{\alpha_0} C_0 ) .$$
Then the faces of $\alpha$ are given by:
$$
d_i \alpha = 
\begin{cases}
  C_{n} \xrightarrow{\alpha_{n-1}} \cdots
  \xrightarrow{\alpha_1} C_1
&
\text{if $i=0$}
\\
  C_{n} \xrightarrow{\alpha_{n-1}} \cdots
  C_{i+1} \xrightarrow{\alpha_{i-1}\alpha_{i}} C_{i-1} \rightarrow
  \cdots
  \xrightarrow{\alpha_0} C_0
&
\text{if $0<i<n$}
\\
  C_{n-1} \xrightarrow{\alpha_{n-2}} \cdots \xrightarrow{\alpha_0} C_0
& 
\text{if $i=n$}
\end{cases}
$$
The degeneracies of $\alpha$ are given by:
$$
s_j \alpha = 
\begin{cases}
C_n \rightarrow \cdots \rightarrow C_j \xrightarrow{=}
C_j \rightarrow \cdots \rightarrow C_0 
&
\text{for $0\le i \le n$}
\end{cases}
$$
For example, the category 
$$ C_2 \xrightarrow{\alpha_1} C_1 \xrightarrow{\alpha_0} C_0$$
has as its
nerve the following simplicial object:
\begin{itemize}
\item in dimension zero:  three nondegenerate simplices, $C_0$,
$C_1$, and $C_2$;
\item in dimension one: three nondegenerate simplices:
\begin{gather*}
C_1 \xrightarrow{\alpha_0} C_0
\\
C_2 \xrightarrow{\alpha_1}C_1 
\\
C_2 \xrightarrow{\alpha_0\alpha_1} C_0,
\end{gather*}
 plus three more (degenerate)
simplices that correspond to the identity maps of $C_0$, $C_1$, and
$C_2$;
\item  in dimension two: one nondegenerate simplex: $C_2
\xrightarrow{\alpha_1} C_1 \xrightarrow{\alpha_0} C_0$, and six degenerate
simplices: $C_2 \xrightarrow{\alpha_1} C_1 \xrightarrow{=} C_1$, \emph{etc}.
\item in higher dimensions: degenerate simplices only.
\end{itemize}
To illustrate the action of the face maps, consider their action  on
the 2-simplex $C_2\rightarrow C_1 \rightarrow C_0$:
\begin{align*}
 d_0 ( C_2 \xrightarrow{\alpha_1} C_1 \xrightarrow{\alpha_0} C_0 ) 
 &= C_2 \xrightarrow{\alpha_1} C_1 \\
 d_1 ( C_2 \xrightarrow{\alpha_1} C_1 \xrightarrow{\alpha_0} C_0 )
 &= C_2 \xrightarrow{\alpha_0\alpha_1} C_0 \\
 d_2 ( C_2 \xrightarrow{\alpha_1} C_1 \xrightarrow{\alpha_0} C_0 ) 
 &= C_1 \xrightarrow{\alpha_0} C_0 
\end{align*}


\section{Equivalences And Connectivity}

In any of these categories, a map is $k$-connected if it is an
isomorphism on $\pi_j$ for $j<k$ and surjective on $\pi_k$. An object
is $k$-connected if the map \emph{from} the initial object is
$k$-connected. Note that this means that $S^n$ an $(n-1)$-connected
space. A spectrum is called connective if all of its negative homotopy
groups are zero.

A map of spectra is an equivalence if it is an isomorphism on $\pi_*$.
A map of spaces is an equivalence if it induces a bijection on $\pi_0$
and an isomorphism on
$\pi_*$ for all compatible choices of basepoint (not just the
basepoint with which all pointed spaces are equipped).


\section{Homotopy Invariance}
\label{sec:homotopy-functors}

The basic object of study of diverse variations of Goodwillie calculus
is a homotopy functor. A homotopy functor is a functor that preserves
equivalences.  In our setting, we will mainly consider functors from
pointed spaces to pointed spaces. It turns out that the study of these
functors is intimately tied up with the study of functors from pointed
spaces to spectra, so we will also be interested in those. For various
examples, it is more convenient to consider algebraic settings, such
as functors from spaces to chain complexes of abelian groups.
Generally these embed into the category of spectra or spaces in some
manner that should be clear upon reflection. For instance, integral
homology $H_*(X;\mathbb{Z})$ is generally regarded as the homology of
a chain complex 
of abelian groups, but is also $\pi_*(\mathbf{H}\mathbb{Z}\wedge X)$ or $\pi_*
\LoopInfty(\mathbf{H}\mathbb{Z} \wedge X)$, which provides a sensible
way of considering homology as the homotopy of a spectrum or space. 


Although homotopy functors, such as $\pi_*$ itself, homology, and
loops on a space, are abundant, there are many familiar functors that
are \emph{not} homotopy functors. For example, the pushout is not a
homotopy functor because the diagram 
$$
\xymatrix{
\ast{}
&
S^0 
\ar[r]
\ar[l]
&
\ast{}
}
$$
has as its pushout one point, but there is an equivalence of diagrams
(an honest map of diagrams that is an equivalence on each vertex)
between this one and  
$$ 
\xymatrix{
D^1
&
S^0 
\ar[r]
\ar[l]
&
D^1
},
$$
whose pushout is $S^1$.  This is just the beginning of trouble;
there are simplicial spaces whose strict realization is not equivalent
to the ``fat'' realization used in this paper.

Since we are interested in studying only homotopy functors, and we do
not want to be constantly concerned whether various constructions are
homotopy invariant, we make the blanket assertion that \emph{all}
constructions will be made in a homotopy invariant way. In particular,
all colimits will be homotopy colimits, and all inverse limits will be
homotopy inverse limits. In order to remind the reader, we will use
the symbols $\hocolim$ and $\holim$ for these constructions. We will
point out explicitly other places where homotopy-invariant
constructions are necessary as they arise. Three situations are worth
mentioning in particular: 
\begin{itemize}
\label{page:homotopy-invariant-constructions}
\item If $X$ and $Y$ have nondegenerate basepoints, then the standard
coproduct $X \vee Y$ is a homotopy invariant, so there is no need to
think of a special coproduct occurring.
\item If $X$ and $Y$ are spaces, then $X\times Y$ is a homotopy
invariant. 
\item If $X$ and $Y$ are spaces with nondegenerate basepoints, then
  $X\wedge Y$ is homotopy invariant. In this case the inclusion
  $X \vee Y \rightarrow X \times Y$ is a cofibration, so the strict
  cofiber ($=X \wedge Y$) is a homotopy invariant.
\item If $X$ is a CW complex and $Y$ is any space, then $\Map(X,Y)$
is a homotopy invariant construction, so there is no need to take a
special $\Map$ as long as the source is a CW complex.
\end{itemize}

In his book \emph{Homotopical Algebra}
\cite{Quillen:homotopical-algebra}, Quillen developed a general
framework for understanding problems of homotopy invariance, called
``model categories''. Quillen's work codifies the general
understanding that one should make sure that colimit constructions
involve cofibrations (``cofibrant objects'') and inverse limit
constructions involve fibrations (``fibrant objects''). Dwyer and
Spalinksi \cite{Dwyer-Spalinski} have written an excellent
introduction to Quillen's work, full of examples familiar to the
working topologist or algebraist.  In the category of topological
spaces with the model structure where weak equivalences are
$\pi_*$-isomorphisms and fibrations are Serre fibrations, all objects
are fibrant, and CW complexes are cofibrant. 
In the category of simplicial sets with $\pi_*$-isomorphisms for weak
equivalences and Kan fibrations for fibrations, all objects
are cofibrant, but only Kan complexes are fibrant.




%
%
\chapter{Goodwillie Calculus}
\label{chap:goodwillie-calculus}


\section{$n$-cubes}


We must first lay out the notation and vocabulary we will use. Given a
set $T$, define the category $\Power(T)$ to have objects all subsets
of $T$ and morphisms the inclusions of subsets.
A $T$-cube $\cube{X}$ is a functor defined on $\Power(T)$.
In general, the functor $\cube{X}$ will take values in the category of
spaces or spectra. An $n$-cube is a $T$-cube with $\abs{T}=n$. When there
is only one $n$-cube being discussed, we may let $\mathbf{n}$
denote the set $\Set{1,\ldots,n}$, and  speak simply of an $\mathbf{n}$-cube.
A $\mathbf{2}$-cube $\cube{X}$ is a diagram like this:
$$\xymatrix{
\cube{X}(\emptyset)
\ar[r]
\ar[d]
&
\cube{X}(\Set{1})
\ar[d]
\\
\cube{X}(\Set{2})
\ar[r]
&
\cube{X}(\Set{1,2}) 
}$$
The ``initial'' object in the cube is
$\cube{X}(\emptyset)$ and the ``terminal'' object is $\cube{X}(\Set{1,2})
= \cube{X}(\mathbf{2})$. We will frequently refer to
those particular two objects in any cube. When we want to consider the
relationship between the initial object and the rest of the cube, we
will use the category $\Power_0(\mathbf{n}) =
\Power(\mathbf{n}) - \Set{\emptyset}$, and use the inverse limit over
this category to assemble the information about all of the objects
except $\cube{X}(\emptyset)$. Similarly, if we want to consider the
relationship between the final object and the rest of the cube, we
will use the category  $\Power_1(\mathbf{n}) = \Power(\mathbf{n}) -
\Set{\mathbf{n}}$.

\begin{definition}[Cartesian]
An $\mathbf{n}$-cube
$\cube{X}$ is ``Cartesian'' if the map $$\cube{X}(\emptyset)
\rightarrow \holim_{U\in\Power_0(\mathbf{n})} \cube{X}(U)$$ is an
equivalence.  
\end{definition}
Alternatively, $\cube{X}$ is Cartesian if it is 
a homotopy pullback cube; that is, $\cube{X}(\emptyset)$
is equivalent to the homotopy inverse limit of the rest
of the cube. 
A 2-cube that is a pullback cube is guaranteed to be a
homotopy pullback if one of the maps to the terminal object is a
fibration. 

\begin{definition}[co-Cartesian]
An $\mathbf{n}$-cube $\cube{X}$ is ``co-Cartesian'' if the map
$$\hocolim_{U\in\Power_1(\mathbf{n})} \cube{X}(U) \rightarrow
  \cube{X}(\mathbf{n})$$ is an equivalence. 
\end{definition}
Alternatively, the cube is
  co-Cartesian if it is a homotopy pushout.
$\cube{X}$ is said to be strongly co-Cartesian if every two
dimensional sub-cube is co-Cartesian. A 2-cube that is a pushout cube
is guaranteed to be a homotopy pushout if one of the maps from the
initial object is a cofibration.
 
A cube $\cube{X}$ is strongly co-Cartesian if every 2-cube contained
in $\cube{X}$ is co-Cartesian. 
A cube is said to be
$k$-Cartesian if the map $\cube{X}(\emptyset)
\rightarrow \holim_{U\in\Power_0(\mathbf{n})} \cube{X}(U)$ is
$k$-connected.

The following cube, which forms the suspension of $X$, is an example
of a co-Cartesian $2$-cube.
$$\xymatrix{
X 
\ar[r]
\ar[d]
&
CX
\ar[d]
\\
CX
\ar[r]
&
\Sigma X
}$$
By the Freudenthal suspension theorem, we see that when
$X=S^n$, this cube is also $(2n-1)$-Cartesian.
\begin{theorem}[Freudenthal]
\label{thm:freudenthal}
For $n\ge 1$, the map $S^n \rightarrow \Omega S^{n+1}$ is $(2n-1)$-connected.
$\qed$
\end{theorem}



\section{The Blakers-Massey Theorem And Its Consequences}

The Blakers-Massey theorem is closely related to the Freudenthal
suspension theorem and the Eilenberg-Zilber theorem for bisimplicial
sets. It gives a way to understand the homotopy of a pushout in a
range, in terms of the spaces used to construct it. The most often-used
statement of the theorem, as proven by Ellis and Steiner, follows.
\begin{theorem}[Ellis-Steiner]
\cite{Blakers-Massey:announce,Ellis-Steiner}
\nocite{Blakers-Massey:I,Blakers-Massey:II,Blakers-Massey:III}
\label{thm:blakers-massey}
Let $\cube{X}$ be a strongly co-Cartesian $\mathbf{n}$-cube of spaces
($n\ge 1$), with each map
$\cube{X}(\emptyset)\rightarrow\cube{X}(\Set{i})$ being
$k_i$-connected. Then $\cube{X}$ is $\left((1-n)+\sum k_i\right)$
Cartesian.
$\qed$
\end{theorem}
In particular, this immediately implies the Freudenthal theorem since
the map from $S^n$ to the cone on $S^n$ is $n$-connected, so the
$2$-cube computing $\Sigma S^n$ is $(1-2)+(n+n)=2n-1$ connected.



In at least one delicate calculation, we will have occasion to use the
full strength of the theorem that Goodwillie proves. 
\begin{theorem}[Goodwillie]
\label{thm:blakers-massey-goodwillie}
Let $\cube{X}$ be an $S$-cube, with $\abs{S}\ge 1$. Suppose that
\begin{enumerate}
\item for each nonempty $T\subset S$, the sub-$T$-cube of $\cube{X}$
induced by the inclusion of $T$ into $S$
is $k(T)$-co-Cartesian, and
\item $k(U) \le k(T)$ whenever $U\subset T$.
\end{enumerate}
Then $\cube{X}$ is $k$-Cartesian, where $k$ is the minimum of
$(1-\abs{S})+\sum_\alpha k(T_\alpha)$ over all partitions $\Set{T_\alpha}$
of $S$ by nonempty sets.
$\qed$
\end{theorem}

Using the Blakers-Massey theorem, we can now prove some important
properties of spectra. We will begin by proving the basic property of
spectra that $\Omega$ and $\Sigma$ are inverse operations on the
homotopy category. This is what is meant by spectra being a ``stable"
category. 

Let $\mathbf{X}$ be a spectrum. Since $\mathbf{X}$ is equivalent to an omega
spectrum, and we are only interested to behavior up to homotopy, we
may assume that $\mathbf{X}$ is an omega spectrum.  

First suppose $\mathbf{X}$ is a bounded below omega spectrum (so $\pi_j$ is
zero for all $j\ll 0$). $\mathbf{X}$ is an omega
spectrum, so $\pi_j X_n = \pi_{j+1} X_{n+1} = \cdots = \pi_{(j-n)} \mathbf{X}$.
Since $\mathbf{X}$ is bounded below, $\pi_{(j-n)} \mathbf{X} = 0$ for $j-n \ll 0$, which
shows that there is a (not necessarily positive) constant $c$ such
that $\pi_j X_n = 0$ for $j < n + c$. That is, $X_n$ is roughly
$(n+c)$-connected.

We will use this fact and the Blakers-Massey theorem to prove that $\mathbf{X}
\rightarrow \Omega\Sigma \mathbf{X}$ is an equivalence by showing that the map
is at least $m$-connected for arbitrary $m$. Let $N(m)$ be large
enough that $X_n$ is $m$-connected for all $n\ge N(m)$. By
Theorem~\ref{thm:blakers-massey}, the map $X_n \rightarrow
\Omega\Sigma X_n$ is $(2m-1)$-connected, which is certainly more than
$m$-connected. This holds for all $n\ge N(m)$, and the homotopy type
of a spectrum only depends on a cofinal subset of the spaces that
compose it, so this shows that $\mathbf{X}\rightarrow \Omega\Sigma \mathbf{X}$ is at
least $m$-connected, with $m$ arbitrary. Therefore, the map must be an
equivalence. 
If $\mathbf{X}$ is not bounded below, then write $\mathbf{X}$ as the homotopy colimit of
bounded below spectra $\mathbf{X}\gen{m}$ created by taking the $(m+n)$-th
connective cover $X_n \gen{m+n}$ of each $X_n$. (The spectrum
$\mathbf{X}\gen{m}$ has $\pi_j = 0$ for $j<m$.) Since the $m$-th 
connective cover of $\mathbf{X}$
comes equipped with a map to $\mathbf{X}$, 
this gives us a sequence $\mathbf{X} =
\hocolim_{m\rightarrow -\infty} \mathbf{X}\gen{m}$.  From the  bounded
below case, we have an equivalence on each object $\mathbf{X}\gen{m}
\rightarrow \Omega\Sigma (\mathbf{X}\gen{m})$. The homotopy colimit of
a map of diagrams that is a weak equivalence on each object is itself
a weak equivalence; this provides the required equivalence $\mathbf{X}
\simeq \hocolim_{m\rightarrow -\infty} \mathbf{X}\gen{m} \simeq
\hocolim_{m\rightarrow -\infty} \Omega\Sigma(\mathbf{X}\gen{m}) \simeq
\Omega\Sigma \mathbf{X}$.


A similar argument using the dual Blakers-Massey theorem shows that
$\Sigma\Omega \mathbf{X} \rightarrow \mathbf{X}$ is an equivalence as well. This
establishes:
\begin{lemma}
\label{lem:omega-sigma-inverses}
In the homotopy category of spectra, $\Omega$ and $\Sigma$ are inverse
operations. 
\end{lemma}

A very similar argument to the one in
Lemma~\ref{lem:omega-sigma-inverses} shows the
following:
\begin{lemma}
\label{lem:spectra-co-cartesian=cartesian}
If $\cube{X}$ is a co-Cartesian cube of spectra, then $\cube{X}$ is
also a Cartesian cube.
\end{lemma}
In particular, this implies:
\begin{corollary}
\label{cor:spectra-cofibration=fibration}
If $X\rightarrow Y\rightarrow Z$ is a cofibration sequence of spectra,
then it is equivalent to a fibration sequence.
\end{corollary}
\begin{proof}
A cofibration sequence is (up to homotopy), a pushout cube 
$$\xymatrix{
X 
\ar[r]
\ar[d]
&
Y
\ar[d]
\\
CX
\ar[r]
&
Z
}, $$
so the previous lemma applies.
\end{proof}

As a functional corollary of the fact that $\Omega$ and $\Sigma$ are
inverse operations, we find that $\Sigma$ commutes with homotopy
inverse limits, and $\Omega$ commutes with homotopy colimits.

\begin{corollary}
In the homotopy category of spectra, $\Omega$ commutes with
$\hocolim$, and $\Sigma$ commutes with $\holim$.
\end{corollary}
\begin{proof}
Let $X$ be a functor from an unspecified diagram category to spectra.
We have $\Omega \holim \Sigma X \simeq \holim \Omega \Sigma X$ since
homotopy inverse limits commute. Applying $\Sigma$ to both sides and
using the fact that $\Sigma\Omega$ and $\Omega\Sigma$ are the identity
up to homotopy in spectra, we have $\holim \Sigma X \simeq \Sigma
\holim \Omega\Sigma X \simeq \Sigma \holim X$. A similar proof shows
$\Omega$ commutes with $\hocolim$.
\end{proof}


\section{Excisive Functors}
\label{sec:excisive-functors}

One of the central notions in Goodwillie calculus is that of
``excision''. Generally speaking, an excisive functor takes
co-Cartesian cubes to Cartesian cubes. The most common example of an
excisive functor is a generalized homology theory. Unfortunately,
beginning from the usual axiom for excision, establishing this
involves a few details, so we delay it until
Section~\ref{sec:homology-1-excisive}. Excisive functors in Goodwillie
calculus are roughly analogous to polynomial functions; one weakness
of this analogy will be noted in
Section~\ref{sec:analytic-functors}. 

Technically speaking, an $n$-excisive functor takes strongly
co-Cartesian $(n+1)$-cubes to Cartesian cubes. Many functors occurring
in nature are not excisive; for example, in the category of spaces,
neither the identity functor nor $\Omega^k$ is $n$-excisive for any
$n$. However, all good functors satisfy a property known as ``stable
excision''. The condition of stable
  excision is a way of codifying to what extent a generalization of
  the Blakers-Massey theorem holds, so its definition is strongly
  reminiscent of that theorem.
Given an $(\mathbf{n+1})$-cube $\cube{X}$ in which
the map $\cube{X}(\emptyset) \rightarrow \cube{X}(\Set{i})$ is
$k_i$-connected, a functor $F$ is said to be stably $n$-excisive if
the cube $F\cube{X}$ is $(\sum k_i - c)$-Cartesian, for a
constant $c$ independent of $\cube{X}$.
The stable excision condition with constant $c$ is known as
``$E_n(c)$''. On occasion, one can only guarantee stable excision if
the maps are of sufficient connectivity (all $k_i \ge \kappa$); that is also
good enough  for the machinery of calculus to work, and this
condition is known as ``$E_n(c,\kappa)$''. 

In order to create an $n$-excisive functor out of a stably $n$-excisive
functor, Goodwillie introduces the functor $T_n$ defined on the
category of functors, which takes a functor $F$ satisfying $E_n(c)$
and produces a new functor $T_n F$ satisfying $E_n(c-1)$. This functor
is equipped with a natural transformation $F\rightarrow T_n F$, so one
can take the colimit to produce an excisive functor (which is what
$E_n(-\infty)$ means). 

Let us now define $T_n F$.  Let $[n]$ denote the set
$\Set{0,1,\ldots,n}$ with basepoint $0$
(this has the same cardinality as the set with the same notation used
in the simplicial category),
and let $A * B$ denote the topological join of two spaces (considered
as \emph{unpointed} spaces). Note that
$X * \emptyset = X$, and $X*[0] = CX$, and in general 
$X * [n]$ is equivalent to $\bigvee^{n} S X$ --- $n$ copies of the
unreduced suspension of $X$.
It is easy to see
that the $[n]$-cube $U \mapsto U$ is strongly co-Cartesian. Similarly,
for a fixed space $X$, the $[n]$-cube
\begin{equation}
\label{eq:Tn-cube}
\cube{X}(U) = X * U
\end{equation}
is strongly co-Cartesian. Since $X*\emptyset = X$, the initial object
of this cube is $X$. 
This cube without the initial object gives
rise to the functor $T_n F$:
\begin{equation}
\label{eq:def-Tn}
 T_n F(X) = \holim_{U \in \Power_0([n])} F(X * U) 
\end{equation}
The map from the initial object ($=F(X)$) to the homotopy inverse
limit of the punctured cube ($=T_n F(X)$) gives the natural
transformation $F \rightarrow T_n F$. A bound on connectivity of this
map can be deduced; if $X$ is an $r$-connected space, then the map $X
\mapsto X*[0]=CX\simeq \basept$ is $(r+1)$-connected. 
If $F$ satisfies $E_n(c)$, then
the cube is at least $((n+1) (r+1) - c)$-Cartesian. Furthermore, all
of the vertices in the cube used to define $T_n F$ have connectivity
at least $r+1$, so iterating the $T_n$ construction produces a map
$T_n F\rightarrow T_n T_n F$ that is even more highly connected. The
limit $P_n F(X) = \colim_{k} T_n^k F(X)$ is the universal $n$-excisive
approximation to $F$, and the map $F(X)\rightarrow P_n F(X)$ is
$((n+1)(r+1)-c)$-connected.  

Goodwillie~\cite{Cal3}
shows that if $F$ is $m$-excisive or analytic (see
Section~\ref{sec:analytic-functors}), then iterating $T_n F$ produces
an $n$-excisive functor. Actually, as he later established, the
functor $P_n F$ is always $n$-excisive. 

\begin{theorem} (\cite{Cal3})
If $F$ is $m$-excisive for some $m$, then $P_n F$ is the universal
$n$-excisive approximation to $F$.
\end{theorem}

The functor $T_n$ commutes with $\holim$ since homotopy inverse limits
commute. The functor $P_n$ commutes with finite homotopy inverse
limits, since filtered $\hocolim$ and finite $\holim$ commute (up to
equivalence).

The most important theorem about the structure of the Taylor tower
gives information about the fiber $D_n F$ of the map $P_n F\rightarrow
P_{n-1}F$: the space $D_n F(X)$ turns out to be an infinite loop space.

\begin{theorem} (\cite{Cal3})
If $F$ is an analytic functor from spaces to spaces that commutes up
to equivalence with filtered colimits of finite complexes
(\emph{i.e.}, satisfies the limit axiom (\ref{def:limit-axiom})),  then
the functor $D_n F$ is an $n$-homogeneous functor given by 
$$ D_n F(X) = \LoopInfty( C_n \wedge_{h \Sigma_n} X^{\wedge n} ),$$
where $C_n$ is some spectrum with a $\Sigma_n$ action, and the smashing
over $h \Sigma_n$ denotes taking homotopy orbits.
\end{theorem}

For any functor $F$, the $D_n F$ are called the ``layers'' of the
Taylor tower, and the associated $C_n$ are called the ``coefficient
spectra''. 

As part of working out a theory of Postnikov invariants for his Taylor
tower, Goodwillie shows that there is actually a functorial delooping
of the derivatives so the usual fibration sequence $D_n \rightarrow
P_n \rightarrow P_{n-1}$ can be delooped once:
\begin{theorem}
\label{thm:delooping-Dn}
(\cite{Cal3})  
If $F$ is a reduced, analytic functor from spaces to spaces, 
then the map $P_n F \rightarrow P_{n-1} F$ is part of a fibration sequence 
$$ P_n F \rightarrow P_{n-1} F \rightarrow \Omega^{-1} D_n F ,$$
where $\Omega^{-1} D_n F$ is a homogeneous $n$-excisive functor whose
loopspace is necessarily $D_n F$.
\end{theorem} 

\section{Example: $P_1 F(X)$}
\label{sec:P1F}

When $F$ is a reduced functor, the construction for $P_1 F(X)$ is the
stabilization of $F(X)$. First, let us compute $T_1 F(X)$ and the map
$F(X) \rightarrow T_1 F(X)$. The diagram to consider to construct $T_1
F(X)$ is:
$$\xymatrix{ F(X) 
\ar[r]
\ar[d]
&
F(X*[0])
\ar[d]
\\
F(X*[0])
\ar[r]
&
F(X*[1])
}$$
The homotopy inverse limit of the punctured cube is easy to understand
once we recall $X*[0] = CX \simeq \basept$ and $X*[1] \simeq \Sigma X$.
\begin{align*}
\holim ( F(X*[0]) \leftarrow F(X*[1]) \rightarrow F(X*[0]) )
&\simeq
\holim ( \basept \leftarrow F(\Sigma X) \rightarrow \basept )
\\
&\simeq \Omega F(\Sigma X)
\end{align*}
The map is then the stabilization map $F(X) \rightarrow \Omega F
(\Sigma X)$. 
Taking the limit as this process is iterated produces $P_1 F(X) \simeq
\colim \Omega^n F(\Sigma^n X)$.


\section{Example: Homology Theories Are $1$-excisive}
\label{sec:homology-1-excisive}

In this section we explain in detail how the excision axiom for
homology corresponds to $1$-excisiveness. One statement of the
excision axiom for homology theories is:
\begin{quote}
Given the pair $(X,A)$ and an open set $U \subset X$ such that
$\overbar{U} \subset \text{int}(A)$, the inclusion $(X-U, A-U)
\hookrightarrow (X,A)$ induces an isomorphism $H_*(X-U, A-U)
\cong H_*(X,A)$. \cite[IV.6, p.~183]{Bredon:topology-and-geometry}
\end{quote}
The data given here corresponds to the existence of a certain
(strongly) co-Cartesian cube:
$$\xymatrix{
A-U
\ar[r]
\ar[d]
&
A
\ar[d]
\\
X-U
\ar[r]
&
X
}$$
Applying the singular chains functor $C_*$, we have a cube of chain
complexes 
(whose homology may be thought of as the homotopy groups of the
functor $Y \mapsto \mathbf{HZ}\wedge Y$, as 
noted in Section~\ref{sec:homotopy-functors}). The relative
homology groups are the mapping cones of the maps $H_*(A-U)\rightarrow
H_*(X-U)$ and $H_*(A)\rightarrow H_*(X)$.  The assertion that they are
isomorphic is equivalent to asserting that the cube is co-Cartesian
(because a cube is co-Cartesian if and only if the iterated homotopy
cofiber (=mapping cone) of the cube is contractible). By
Lemma~\ref{lem:spectra-co-cartesian=cartesian}, a cube of spectra is 
co-Cartesian if and only if it is Cartesian, so the resulting cube is
Cartesian as well. Hence $H_*$ takes this co-Cartesian cube to a
Cartesian cube of spectra.
It remains to show that $H_*$ takes all (strongly) co-Cartesian
2-cubes to Cartesian cubes. By replacing an arbitrary co-Cartesian
cube with a weakly equivalent CW-cube (all spaces CW complexes, all
maps CW inclusions), we obtain nice inclusion maps. Let $\cube{X}$
denote our CW 2-cube. Putting $X=\cube{X}(\Set{1,2})$,
$A=\cube{X}(\Set{1})$, and $U = A - \cube{X}(\emptyset)$, it is easy
to check that $\cube{X}(\Set{2}) = X - U$. At this point, we must
recall that despite the statement of the axiom given above, it is
sufficient to require that the pair $(A, A-U)$ be an NDR pair, or that $A-U
\rightarrow A$ be a cofibration. Since all CW inclusions are
cofibrations, our cube satisfies this hypothesis, reducing the general
case to the one we have worked out previously.



\section{Analytic Functors}
\label{sec:analytic-functors}

If the excisive functors of Goodwillie calculus are analogous to
polynomial functions in ordinary calculus, analytic functors are
analogous to functions with whose Taylor series converge in some disk
about the origin.

An informal statement of analyticity is this: a functor is
$r$-analytic if the coefficient spectra $C_n$ that compose its layers
have a connectivity that tends to $-\infty$ with slope roughly $\ge
-rn$ for some $c$ independent of $n$. In some sense, this means that
for spaces of connectivity $\ge r$, the individual ``layers'' $C_n$
can be distinguished. We will give a rigorous definition of
analyticity and later, in
Section~\ref{sec:analytic-functors-connective} prove that it has the
properties of this informal definition.

\begin{definition}[analytic]
Formally, a functor $F$ is $r$-analytic if there exists a constant $c$
depending only on $F$ such that $F$ satisfies $E_n(rn-c)$ for each
$n$.\footnote{Goodwillie only requires the weaker condition
  $E_n(rn-c,r+1)$ be satisfied.}
\end{definition}

The immediate consequence of this definition is that, according
to the computation following \eqref{eq:def-Tn}, if $X$ is
$r$-connected, the map $F(X) \rightarrow P_n F(X)$ is
$(n+r+c+1)$-connected.  The critical trait is that in this case, the
connectivity of the map $F(X)\rightarrow P_n F(X)$ increases with $n$.
In this case, the Taylor tower of $F$ is said to converge, since in
the inverse limit, the map $F \rightarrow \holim P_n F$ is an
equivalence ($\infty$-connected).

Arguments about analytic functors frequently make use of asymptotic
estimates such as the preceding one. To make the essence of these
estimates clearer, we will omit the irrelevant constants and use the
phrase ``approximately $n$-connected'' to mean ``there exists a
constant $c$ independent of the variables appearing, such that the map
is at least $(n+c)$-connected''. The following lemma is a good example
of this.

\begin{lemma}
\label{lem:connectivity-F-PnF}
If $F$ is $r$-analytic and $X$ is $(m-1)$-connected (for example
$X=S^m$), then the map $F(X) \rightarrow P_n F(X)$ is
approximately $n(m-r)$-connected.
\end{lemma}
\begin{proof}
This is a direct computation. By hypothesis, $F$ satisfies $E_n(rn)$
--- note that we have omitted the $c$ --- and hence takes strongly
co-Cartesians squares of the form used in constructing $T_n$
\eqref{eq:Tn-cube} to $k$-Cartesian squares, where $k= -rn + m(n+1) =
nm -rn +m = n(m-r) +m \approx n(m-r)$. If it were possible for $m$ to
be negative, we would want to be more careful about ignoring the $+m$
to get a lower bound.
\end{proof}

In this work, the statement ``$F$ is analytic'' means ``$F$ is
$r$-analytic for some $r$''. This is a different usage from that of
Goodwillie's first paper on calculus \cite{Cal1}, where
``analytic'' means ``$1$-analytic'', but consistent with later usage
\cite{Cal3}.


%
%
\section{Technical Lemmas}
In this section, we record some technical lemmas about Cartesian cubes
and excisive functors that we will need later.


The first lemma is that given a Cartesian cube of cubes, the cube
resulting from taking fibers of the inner cubes is still Cartesian.
We begin by recalling Proposition~0.2 from \cite{Cal2}, which we will
use to perform our decomposition of the homotopy inverse limit.
\begin{proposition}[{\cite[Proposition~0.2]{Cal2}}]
\label{prop:Goodwillie-decompose-holim}
If $\cube{A}$ is covered by $\cube{A}_1$ and $\cube{A}_2$ in the sense
that the nerve of $\cube{A}$ is the union of the nerves of
$\cube{A}_1$ and $\cube{A}_2$, then for any functor $F$ from
$\cube{A}$ to unbased spaces, the diagram of fibrations
$$\xymatrix{
\holim (F)
\ar[r]
\ar[d]
&
\holim ( F \rvert_{\cube{A}_1} )
\ar[d]
\\
\holim ( F \rvert_{\cube{A}_2} )
\ar[r]
&
\holim ( F \rvert_{( \cube{A}_1 \cap \cube{A}_2 )} )
}$$
is a pullback square. $\qed$
\end{proposition}

\begin{lemma}
\label{lem:big-cube-Cartesian-fiber-cube-Cartesian}
  Let $\cube{X}$ be a $(S \amalg T)$-cube, regarded as an $S$-cube
  $\cube{Y}$ of $T$-cubes $\cube{Y}(U)$, for $U\subset S$. Let
  $\widetilde{\cube{Y}}$ be the $S$-cube of total fibers of the
  $T$-cubes $\cube{Y}(U)$.
  If $\cube{X}$ is Cartesian, then $\widetilde{\cube{Y}}$ is
  Cartesian.
\end{lemma}
\begin{proof}
  Recall that one may compute $\widetilde{\cube{Y}}(U)$ as the fiber
  of the map 
  $$ \cube{Y}(U)(\emptyset) \rightarrow \holim_{V \in \Power_0(T)}
  \cube{Y}(U)(V). $$
  The cube $\widetilde{\cube{Y}}$ is Cartesian if the map
  \begin{equation}
    \label{eq:ytilde-cartesian}
    \widetilde{\cube{Y}}(\emptyset)
    \rightarrow 
    \holim_{U\in \Power_0(S)} \widetilde{\cube{Y}}(U)
  \end{equation}
  is an equivalence.
  Consider the cube:
  \begin{equation}
    \label{eq:local-cube-1}
  \xymatrix{
    \cube{Y}(\emptyset)(\emptyset)
    \ar[d]
    \ar[r]
    &
    {\displaystyle \holim_{U\in\Power_0(S)} \cube{Y}(U)(\emptyset)}
    \ar[d]
    \\
    {\displaystyle \holim_{V\in\Power_0(T)} \cube{Y}(\emptyset)(V)}
    \ar[r]
    &
    {\displaystyle \holim_{
        \substack{
          U \in \Power_0(S) 
          \\
          V\in\Power_0(T)}}
      \cube{Y}(U)(V)}
    }
  \end{equation}
  The fibers of this cube in the vertical direction are the functors
  of $\widetilde{\cube{Y}}$ that we are interested in. We will show
  that if $\cube{X}$ is Cartesian, then this cube is Cartesian, and
  hence that \eqref{eq:ytilde-cartesian} is an equivalence.

    Let
    $\cube{A}_1$ be the full subcategory of $\Power_0(S \amalg T)$
    generated by $\Set{ (U \times V) \suchthat U \in \Power_0(S), V
      \subset T}$, 
    and similarly let $\cube{A}_2$ be generated by
    $\Set{ (U \times V) \suchthat U \subset S, V \in \Power_0(T)}$. 
    Using Proposition~\ref{prop:Goodwillie-decompose-holim}, the
    following cube is a homotopy pullback:
    $$\xymatrix{
      {\displaystyle
      \holim_{\Power_0(S \amalg T)} \cube{Y}
      }
      \ar[r]
      \ar[d]
      &
      {\displaystyle
      \holim_{\cube{A}_1} \cube{Y}
      }
      \ar[d]
      \\
      {\displaystyle
      \holim_{\cube{A}_2} \cube{Y}
      }
      \ar[r]
      &
      {\displaystyle
      \holim_{\cube{A}_1 \cap \cube{A}_2} \cube{Y}
      }
    }$$
    We can then recognize
    \begin{align*}
    \holim_{U\in\Power_0(S)} \cube{Y}(U)(\emptyset)      
    &\simeq
    \holim_{U\in\Power_0(S)} \holim_{V \in \Power(T)} \cube{Y}(U)(T)
    \\
    &\simeq
    \holim_{
      \substack{
        (U\times V)\in\Power_0(S\amalg T) 
        \\ 
        U \in \Power_0(S), V \subset T}
    } \cube{Y}(U)(T)
    \\
    &\simeq
    \holim_{\cube{A}_1}  \cube{Y},
    \end{align*}
    to identify the upper right hand corner of
    \eqref{eq:local-cube-1}. Similarly, the lower left corner is
    $\holim_{\cube{A}_2} \cube{Y}$ and the lower right corner is 
    $\holim_{\cube{A}_1\cap \cube{A}_2} \cube{Y}$.
    Hence $\holim_{\Power_0(S\amalg T)} \cube{Y}$ is actually equivalent
    to the homotopy pullback of the lower right hand corner of the
    cube in \eqref{eq:local-cube-1}. 

    If $\cube{X}$ is Cartesian,
    that is exactly the assertion that  the map
    $$\cube{Y}(\emptyset)(\emptyset) 
    \rightarrow
    \holim_{\Power_0(S \amalg T)} \cube{Y}$$
    is an equivalence, and $\holim_{\Power_0(S \amalg T)} \cube{Y}$ is
    the homotopy pullback of the lower right of
    \eqref{eq:local-cube-1}, so \eqref{eq:local-cube-1} is actually
    Cartesian, as desired.
\end{proof}


The next lemma is that given mild conditions, $n$-excisiveness is
preserved by extensions along fibrations.

\begin{lemma}
\label{lem:excisive-up-fibration}
Let $A \rightarrow B \rightarrow C$ be a fibration of homotopy
functors from spaces to spaces. If $A$ and $C$ are $n$-excisive, and
furthermore either:
\begin{enumerate}
\item $B$ takes connected values; or
\item $\pi_0 B$ and $\pi_0 C$ lift to functors to groups, and the
natural transformation $\pi_0 B\rightarrow \pi_0 C$ is a
\emph{surjective} group homomorphism, 
\end{enumerate}
then $B$ is $n$-excisive.
\end{lemma}
\begin{proof}
Let $\cube{Y}$ be a strongly co-Cartesian $(n+1)$-cube. 
By the hypothesis that $A$ and $C$ are $n$-excisive, the left and
right vertical maps are equivalences: 
$$\xymatrix{
A\cube{Y}(\emptyset)
\ar[d]^{\simeq}
\ar[r]
&
B\cube{Y}(\emptyset)
\ar[d]
\ar[r]
&
C\cube{Y}(\emptyset)
\ar[d]^{\simeq}
\\
{
\displaystyle
\holim_{\Power_0(\mathbf{n+1})} 
A\cube{Y}
}
\ar[r]
&
{
\displaystyle
\holim_{\Power_0(\mathbf{n+1})} 
B\cube{Y}
}
\ar[r]
&
{
\displaystyle
\holim_{\Power_0(\mathbf{n+1})} 
C\cube{Y}
}
}$$
We can then use the Five Lemma on the long exact sequences of the two
fibrations to conclude that the middle map is an isomorphism on
$\pi_m$, for $m\ge 1$. It remains to handle $\pi_0$. By hypothesis,
$\pi_0 B \cube{Y} \rightarrow \pi_0 C\cube{Y}$ is a surjection. A
diagram chase then shows that the middle vertical map is an
isomorphism on $\pi_0$. Therefore $B$ is $n$-excisive.
\end{proof}

%
%
\chapter{Left Kan Extensions}
\label{chap:left-kan}

The left Kan extension provides a natural way of extending a functor
defined on a category $\cat{C}$ to one defined on a larger category
$\cat{D}$. In complete generality, extensions of functors do not
always exist, but in our setting, where we are dealing with
subcategories of topological spaces or spectra, there is no problem.


\section{Strict Left Kan Extension}

To understand how the left Kan extension functions, first let us begin
by working in a simpler setting.  Suppose $V$ is a sub-vector space of
$W$, and we are given a function $f$ defined on $V$. The analogous
question is, ``Does there exist an $\tilde{f}$ defined on all of
$W$ that agrees with $f$ on $V$?''  In the case of vector spaces, the
answer is clearly yes; we can extend by zero (or anything else) on the
orthogonal complement to $V$. Evidently, this method is heavily
dependent on the existence of an inner product.

The case of categories and functors is not quite so simple. If a
morphism $\alpha: A \rightarrow B$ in $\cat{C}$ factors through an
object $D$ in $\cat{D}$ (after inclusion), then any functor that sends
all $D$ to zero must also send $\alpha$ to zero. But this may not be
what our original functor ``$f$'' does to $\alpha$. Because of this
complication, we need a more clever way of extending a functor from
$\cat{C}$ to one on $\cat{D}$.
 
First, consider the root of the problem: let $a, b \in \cat{C}$, and
let $d\in \cat{D}$, and let $F$ denote the functor on $\cat{C}$ that
we are trying to extend to $\cat{D}$. Suppose that there is a single
morphism from $a$ to $b$, and that it factors through $d$.
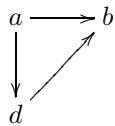
$$
\xymatrix{ 
a
\ar[r] 
\ar[d]
& 
b 
\\ 
d \ar[ur]
&
}
$$
One candidate for $\tilde{F}(d)$ that will make the diagram commute is
$\tilde{F}(d) = F(a)$. We can then collapse the morphism $a\rightarrow d$ to
the identity, and let $\tilde{F}(d\rightarrow b) = F(a\rightarrow b)$.

To extend this scheme to the situation where there is more than one
morphism from $a$ to $d$, we can let $\tilde{F}(d) = F(a) \times
\Hom(a,d)$. Then for each $j:a\rightarrow d$, we can let
$\tilde{F}(j)$ send $F(a)$ to $(F(a),j) \hookrightarrow F(a)\times\Hom(a,d)$.
To do this consistently for all objects of $\cat{D}$, take the
coproduct over all objects of $\cat{C}$. Then there is a small
matter that for objects in $\cat{C}$, the new function $\tilde{F}$ will be too
large, since it consists of at least $F(C)\times\Hom(C,C)$, so one
must divide out the morphisms of $\cat{C}$ by identifying
$(F(f)F(C),1)$ with $(F(C),f)$.
In general let $\tilde{F}$ be the functor given by taking the
coproduct over all objects in $C \in \cat{C}$ of $F(C) \times
\Hom(C,-)$, and then equalizing out the morphisms in $\cat{C}$; that
is, consider
$$ \bigvee_{C\in\cat{C}} F(C) \times \Hom(C,-) ,$$
and for every morphism $f: C\rightarrow C'$ in $\cat{C}$, identify
$(F(C),f)$ with $(F(C'),1)$. This has the effect of forcing
$\tilde{F}(C')$ to be equal to $F(C')$ since every element of
$f\in \Hom(C,C')$ gives rise to an identification $(F(C),f)$ with
$(F(C'),1)$, so every element of the coproduct is either empty (if
there are no morphisms $C\rightarrow C'$) or identified with $F(C')$.
Another way of writing this is 
$$\tilde{F} = F(\cat{C}) \tensor_{\cat{C}} \Hom(\cat{C},-) ,$$ 
where the coproduct over all $C\in \cat{C}$ is implicit in the meaning
of $\tensor_{\cat{C}}$.
This is the left Kan extension, also known as a ``coend''. It is
sometimes written $\int^{c} F(c) \times \Hom(c,-)$, which is useful
for understanding interchanges of limiting processes, but (in my
opinion) makes less
familiar the properties we are most interested in.

We will use the notation $L_I F$ to denote the left Kan extension of
$F$ along $I$. When $I$ is the inclusion of a subcategory $\cat{C}
\rightarrow \cat{D}$, we will use $L_{\cat{C}} F$ instead.
\begin{theorem}
(\cite[\S X.3, Corollary~3, p.~235]{MacLane:categories-for-the-working-mathematician})
\label{thm:strict-left-kan-agrees-with-orig}
If $I: \cat{C} \rightarrow \cat{D}$ is full and faithful, then the
natural transformation $L_I(F)\circ I \rightarrow F$ is an isomorphism.
\end{theorem}

For more information on Kan extensions and coends, see
\cite[Chapter~X]{MacLane:categories-for-the-working-mathematician}.





\section{Homotopy Invariant Left Kan Extension}

We consider a simplicial resolution of the strict left Kan extension
for two reasons: we want to guarantee that we have a homotopy functor,
and the layers of the simplicial resolution are easier to understand
than strict left Kan extension itself.
For the remainder of this paper, ``left Kan extension'' will mean the
simplicial resolution of the left Kan construction, as defined in this
section.

\begin{definition}[Left Kan extension]
\label{def:left-Kan}
Let $F$ be a functor from spaces to spaces.
The homotopy-invariant left Kan extension $L_{\cat{C}} F$ of $F$ over a
subcategory $\cat{C}$ of the category of spaces $\cat{D}$ 
is given by the realization of the
simplicial functor to spaces:
\begin{equation}
  \label{eq:left-kan-spaces}
  [n] \mapsto \bigvee_{(C_0, \ldots, C_n)} F(C_0) \wedge
\left(
    \Hom_{\cat{C}}(C_0,C_1) \times \cdots \times \Hom_{\cat{D}}(C_n, -
    )
\right)_+ .
\end{equation}
The coproduct is taken over all $(C_0, \ldots, C_n)\in\cat{C}^{\times
  n}$.  (Recall that the product of any space with an empty space is
the empty space, so this is really the (continuous) nerve of the
category in disguise.) We use the construction $F \wedge (-)_+$ rather
than $F \times (-)$ so that the construction is immediately applicable
in the case of functors from spaces to spectra as well. Recall that
spaces is a topological category, and we use the mapping space $\Hom$.
\end{definition}

\begin{lemma}
The left Kan extension given by \eqref{eq:left-kan-spaces} is a
homotopy functor if $\cat{C}$ is a full subcategory of spaces whose
objects are cofibrant. 
\end{lemma}
\begin{proof}
Realizations take levelwise weakly equivalent object to weakly
equivalent objects, so we need only show that in each dimension our
simplicial functor is a homotopy functor. This consists of tracing
through to verify that the conditions mentioned in
Section~\ref{sec:homotopy-functors} hold. A discussion of each of the
pieces of this argument occurs on
page~\pageref{page:homotopy-invariant-constructions}.

Each dimension consists of a
coproduct functors; this is homotopy invariant if after evaluation all
of the spaces involved have nondegenerate basepoints. We have made a
blanket assumption to this effect since all functors to spaces can be
(functorially) made to take values in spaces with nondegenerate
basepoints by adding a ``whisker'' if necessary. Similarly, the smash
product is a homotopy functor if all spaces involved have
nondegenerate basepoints.
On the right side of the smash product, we have a product of constant functors
$$\Hom(C_0,C_1)\times\cdots\times\Hom(C_{n-1},C_n)$$
with the functor
$\Hom(C_n,X)$. Constant functors are homotopy invariant, of course,
and $\Hom(C_n, X)$ is homotopy invariant because $C_n$ is a cofibrant
space (CW complex) by hypothesis. The product of
homotopy functors is a homotopy functor, so we are done.
\end{proof}

As with the strict left Kan extension, $L_n F$ is
equipped with a map (natural transformation) to $F$ given by mapping
 $$F(C_0)
\wedge \left( C_0 
\xrightarrow{\alpha_1} \cdots \xrightarrow{\alpha_{n}} C_n \xrightarrow{\beta} X
\right)_+$$ to $$F(\beta\alpha_{n-1}\cdots\alpha_1): F(C_0)
\rightarrow F(X).$$ Evidently, given a map $X\xrightarrow{f} Y$, we
have a map $L_\cat{C} F(X) \rightarrow L_\cat{C} F(Y)$ given by
sending the simplex $C_0 \xrightarrow{\alpha_1} C_1 \rightarrow  \cdots
\rightarrow C_n \xrightarrow{\beta} X$ to $C_0 \xrightarrow{\alpha_1}
C_1 \rightarrow  \cdots \rightarrow C_n \xrightarrow{f\circ\beta} Y$,
and this is compatible with the map to $F$, since $F(f)
F(\beta\alpha_{n}\cdots\alpha_1) = F(f\beta\alpha_{n}\cdots\alpha_1)$.

For this map to be continuous with respect to the topology on the
$\Hom$ sets requires $F$ to be a continuous functor. To emphasize the
topology, we will write $\Map$ for $\Hom$ here. (But recall that is
this work they are both the same --- both are topologized.)
Recall that a
functor $F$ is continuous if given spaces $A$ and $B$, the map of
spaces $\Map(A,B) \rightarrow \Map(FA,FB)$ sending $f$ to $Ff$ is
continuous. The map from the left Kan extension arises from the
composition of this map with the evaluation map:
$$\xymatrix{
F(C) \wedge \Map(C,X)_+
\ar[d]^{1 \wedge F}
\\
F(C) \wedge \Map(FC, FX)_+
\ar[d]^{\text{eval}}
\\
F(X)
}$$ 

The main interesting property of the left Kan extension is that it
agrees with the original functor on the category $\cat{C}$ up to
equivalence. This result is the analog of
Theorem~\ref{thm:strict-left-kan-agrees-with-orig} for the homotopy
invariant left Kan extension.
\begin{proposition}
  \label{prop:left-kan-agrees-with-orig}
  Let $F$ be a functor from spaces to spaces.
  Consider the left Kan extension $L_{\cat{C}} F$, where $\cat{C}$ is
  a full subcategory of spaces, and let
  $C\in\Obj(\cat{C})$. 
  Then the natural transformation from $L_{\cat{C}} F(C)$ to $F(C)$ is
  an equivalence.
\end{proposition}
\begin{proof}
  This is a general fact about nerves of categories with terminal
  objects. Consider the general portion of the coproduct in dimension
  $n$ given by 
  $$F(C_0) \wedge \left( C_0 \xrightarrow{\alpha_1} C_1
    \rightarrow \cdots \rightarrow C_n \rightarrow C \right) . $$
  Iterating the commutative diagram
  \begin{equation*}
    \xymatrix{
      F(C_0) \ar[d]^{F(\alpha_1)}&\wedge & 
           C_0 \ar[r]^{\alpha_1} \ar[d]^{\alpha_1} & 
           C_1 \ar[r] \ar[d] & 
           \cdots\ar[r] & 
           C_n\ar[r]^{\beta} \ar[d]^{\beta} & 
           C \ar[d]^{=} \\
     F(C_1) & \wedge& C_1 \ar[r] & C_2 \ar[r] & \cdots\ar[r]^{\beta} & C \ar[r]^{=} & C 
   }
  \end{equation*}
  gives a homotopy from $F(C_0) \wedge \left( \Hom(C_0,C_1) \times
    \cdots \times
    \Hom(C_n,C) \right)_+$ to $F(C)\wedge \left( \text{id}_{C} \times
    \cdots \times \text{id}_{C} \right)_+$.  
  The latter is $F(C) \wedge S^0 \isom F(C)$.
  (The hypothesis of being a
  full subcategory is needed to write the $\beta$ on the second line,
  since there it is required to be an element of $\Hom_{\cat{C}}(C_n,
  C)$ instead of just $\Hom_{\cat{D}}(C_n,C)$.)
\end{proof}

The particular left Kan extensions we are interested in commute with
realizations of simplicial $k$-connected spaces, for large enough $k$,
because there is a bound on the dimension of the objects in the
subcategory being extended along. 

\begin{lemma}
The functor $\Map(S^n,-)$ commutes with realizations of simplicial
$(n-1)$-connected spaces.
\end{lemma}
\begin{proof}
By adjunction, $\Map(S^n,X) \cong \Map(S^{n-1},\Map(S^1,X)) \cong
\Map(S^{n-1},\Omega X)$, and $\Omega X$ has connectivity one less than
$X$, so by induction we need only show that $\Map(S^1,-)$ commutes
with realizations of simplicial connected spaces.
Waldhausen's lemma
(Lemma~\ref{lem:realization-of-levelwise-fibration})
implies this, since it shows that if all $X_i$ are connected, then
both $\Omega \realization{X_\cdot}$ and $\realization{\Omega X_\cdot}$
are equivalent to the homotopy fiber of the map $0\rightarrow
\realization{X_\cdot}$. 
\end{proof}
\begin{corollary}
The functor $\Map(\bigvee S^n, -)$ commutes with realizations of
simplicial $(n-1)$-connected spaces.
\end{corollary}
\begin{proof}
We know $\Map(\bigvee S^n, -) \cong \prod \Map(S^n,-)$, and products
commute with realizations.
\end{proof}
\begin{corollary}
\label{cor:Map-K-commutes-with-realization}
Let $K$ be a finite CW complex of dimension $n$. 
The functor $\Map(K, -)$ commutes with realizations of
simplicial $(n-1)$-connected spaces.
\end{corollary}
\begin{proof}
The result is true for $K=\basept$. We proceed by induction, showing
that you can add one cell and the result still holds. Suppose
$\Map(K',-)$ commutes with realizations of simplicial
$(n-1)$-connected spaces. Suppose that a $(k+1)$-cell is added to $K'$
along attaching map $\alpha$ to produce $K$.
$$\xymatrix{
S^k_+ 
\ar[r]^{\alpha}
\ar[d]
&
K'
\ar[d]
\\
D^{k+1}_+
\ar[r]
&
K
}$$
Applying the functor $\Map(-,X_i)$ to this co-Cartesian square produces
a Cartesian square:
$$\xymatrix{
\Map(K,X_i)
\ar[r]
\ar[d]
&
\Map(D^{k+1}_+,X_i)
\ar[d]
\\
\Map(K',X_i)
\ar[r]^{\alpha^*}
&
\Map(S^k_+,X_i)
}$$
Hence we have a 2-cube of simplicial spaces that is levelwise
Cartesian. We would like to conclude that after realization, this is
still a Cartesian square, so both $\Map(K,\realization{X_\cdot})$ and
$\realization{\Map(K,X_\cdot)}$ are equivalent to the inverse limit
over the rest of the cube:
\begin{align*}
\realization{\Map(K,X_\cdot)}
&\simeq \holim \left( 
\realization{\Map(D^{k+1}_+,X_\cdot)}
\rightarrow
\realization{\Map(S^k_+,X_\cdot)}
\leftarrow
\realization{\Map(K',X_\cdot)}
  \right)
\\
\intertext{which, by the induction hypotheses, is}
&\simeq \holim \left( 
\Map(D^{k+1}_+,\realization{X_\cdot})
\rightarrow
\Map(S^k_+,\realization{X_\cdot})
\leftarrow
\Map(K',\realization{X_\cdot})
  \right)
\\
&\simeq \Map(K,\realization{X_\cdot}) .
\end{align*}
When $k$ is at most the connectivity of $X_i$ (that is, $k \le n-1$,
so the cell of dimension $k+1$ being attached has dimension $k+1\le
n$, which is our hypothesis on $\dim K$), both of the spaces
$\Map(S^k_+, X_i)$ and $\Map(D^k_+, X_i)$ are connected.  In this
circumstance, they trivially satisfy the $\pi_*$-Kan condition, so by
Theorem~\ref{thm:Bousfield-Friedlander}, the
square is in fact Cartesian after realization.
\end{proof}
\begin{proposition}
  \label{prop:left-kan-commutes-with-some-realizations}
  Let $F$ be a functor from spaces to spaces.
  Let $\cat{C}$ be a subcategory of CW spaces whose objects have
  dimension at most $a$. Then $L_{\cat{C}} F$ commutes with
  realizations of $(a-1)$-connected simplicial spaces.
\end{proposition}
\begin{proof}
In \eqref{eq:left-kan-spaces}, we see that everything but
$\Hom(C_n,-)$ commutes with realizations of $X$ with no conditions.
The fact that $\Hom(C_n,-)$ commutes with realizations of simplicial
$(a-1)$-connected spaces is the content of
Corollary~\ref{cor:Map-K-commutes-with-realization}. 
\end{proof}

\begin{lemma}
\label{lem:LCFX-good}
Let $\cat{C}$ be the full subcategory of pointed spaces whose objects
are finite coproducts of $S^0$: $\bigvee^k S^0$ for $k=0,1,\ldots$. 
If $X_\cdot$ is a simplicial set, and $F$ is a functor from spaces to
nondegenerately based spaces (as all of our spaces are assumed to be),
then 
the simplicial space $[k] \mapsto L_\cat{C} F(X_k)$ is good
(\ref{def:good-simplicial-space}).
\end{lemma}
\begin{proof}
From Lemma~\ref{lem:preserve-goodness}, item~$7$, we know that when
$X_\cdot$ is good, then so is the mapping space $\Hom(C_n, X_\cdot)$. 
Following the construction of the left Kan extension
(Definition~\ref{def:left-Kan}), Lemma~\ref{lem:preserve-goodness},
item~$3$, shows that the product with the constant space 
$$\Hom(C_0,C_1)\times\cdots\times\Hom(C_{n-1},C_n)$$
is still a good space. Adding a disjoint basepoint is still good
(item~$1$ in the same lemma), as smashing with a space with a
nondegenerate basepoint (item~$4$), and 
taking the coproduct over all
$n$-tuples $(C_0,\ldots,C_n)$ (by item~$6$). Finally, the realization in
the direction internal to the left Kan extension still produces a good
simplicial space by item~$7$.
\end{proof}



\section{Defining Additive Calculus From The Left Kan Extension}
\label{sec:additive-calculus}


Let $L$ denote the left Kan extension (\ref{def:left-Kan}) over the
full subcategory of pointed spaces generated by finite coproducts of
$S^0$ (including the empty coproduct, $\basept$).
Let $F_X$ denote the functor sending $Y$ to $F(X\wedge Y)$. This fixes
information about $X$ into the functor, so that the left Kan extension
$L F_X$ contains information about the value of $F$ on coproducts of
$X$, not just the value of $F$ on points. The functor $L F_X$
naturally comes equipped with a map to $F_X$, but because we are
taking the left Kan extension over a subcategory that contains 
$S^0$, the unit of the smash product, there is also a map $F(X)
\rightarrow (L F_X) (S^0)$ --- note the change from $F_X$ to $F(X)$ ---
given by sending $F(X)$ to the 0-simplex $F(X\wedge S^0) \times
1_{S^0}$. 
Applying $P_n$ to a left Kan extension $L F_X$ creates a theory that
we refer to as $n$-additivization.
$$ P_n^{d} F(X) = P_n(L F_X) (S^0) .$$
The decoration ``d'' stands for
``discrete'', since the functor is defined by a left Kan extension
over a discrete subcategory of spaces.


Using the fact that left Kan extensions commute with realizations of
appropriately highly connected spaces, we can write one of these functors in
another, perhaps more familiar, way.
Let us compute $P_1^{d} F(X)$ for a reduced functor $F$.
\begin{align*}
P_1^{d} F(X)
&= P_1(L F_X)(S^0)  .
\end{align*}
{As in Section~\ref{sec:P1F}, this is equivalent to}
$$\colim_n \Omega^n L F_X ( S^n \wedge S^0) . $$
{Then $S^0$ is the identity of the smash product, so this equals}
$$ \colim_n  \Omega^n L F_X (S^n) . $$
Proposition~\ref{prop:left-kan-commutes-with-some-realizations}
applied with $a=0$ implies that $L F_X$ commutes with realizations
of all simplicial sets, so this is equivalent to
$$ \colim_n  \Omega^n \realization{L F_X(S^n_{\cdot})}  . $$
{Since $L F$ agrees with $F$ (up to equivalence) on the category of
  finite discrete spaces, and each $S^n_i$ is a finite discrete space,
  this equivalent to}
$$ \colim_n  \Omega^n \realization{F_X(S^n_{\cdot})} .$$
Which, by the definition of $F_X$, shows that 
$$P_1^d F(X) \simeq \colim_n  \Omega^n \realization{F(X\wedge S^n_{\cdot})} .$$

\begin{example}
\label{ex:K-2-2}
To work out a particular example, let $F(X) = K(H_2(X),2)$ be the
Eilenberg-MacLane space with $\pi_2 = H_2(X)$. (This is another example
of using a topological substitute for the category of abelian groups.)
We assert that $L F_X(S^1)$ is the bar construction on $F(X)$, and
hence $\Omega L F_X(S^1) \simeq \Omega BF(X) \simeq \Omega
K(H_2(X),3) \simeq F(X)$, so $P_1^{d} F(X) = F(X)$.

Recall that in our standard simplicial set model for $S^1$, there are
$n+1$ simplices in dimension $n$. That is, the model is $[n] \mapsto
\bigvee^n S^0$. Applying $H_2(X\wedge -)$ levelwise, we get $[n]
\mapsto H_2(\bigvee^n X)$, which is $\oplus^n H_2(X)$. The face maps
are induced by the fold map $S^0 \vee S^0 \rightarrow S^0$. This is
becomes addition on $H_2$, since addition is universal as a map from
$A\oplus A \rightarrow A$, for any abelian group $A$, that restricts
to the identity on each component of the coproduct $A \oplus A$. This
allows us to identify $H_2(X \wedge S^1_{\cdot})$ with the bar
construction $B H_2(X)$ on the abelian group $H_2(X)$. Since the
functor $K(-,2)$ preserves products of abelian groups, $K(H_2(X\wedge
S^1_{\cdot})),2)$ is the bar construction on $K(H_2(X),2)$, so $L
F_X(S^1)$ is $BF(X)$, as claimed.
\end{example}


%
%
\chapter{Excisive Functors from Spaces to Spectra}
\label{chap:finite-degree}

Functors from spaces to spectra that are $n$-excisive and satisfy the
limit axiom (\ref{def:limit-axiom}) are determined
by their left Kan extensions over the full subcategory of spaces
containing as objects the discrete spaces $\Ord{k}$, where $k$ ranges
from $0$ to $n$ (the degree of the functor). To recollect: we use the
notation $\Ord{k}$ to denote the space $\bigvee^k S^0$, which has
$k+1$ points. In this section, we write $L_n F$ for the left Kan
extension of $F$ along the inclusion of the full subcategory
of spaces whose objects are
$\Set{\Ord{0},\ldots,\Ord{n}}$.

\begin{definition}
\label{def:limit-axiom}
  A homotopy functor $F$ is said to satisfy the \emph{limit axiom} if
  $F$ commutes with filtered homotopy colimits of finite
  complexes. That is, if $\hocolim F(X_\alpha) \simeq F(\hocolim
  X_\alpha)$ for all filtered systems $\Set{X_\alpha}$ of finite complexes,
  then $F$ satisfies the limit axiom. 
\end{definition}
The limit axiom is needed to relate the values of $F$ on infinite
complexes to the values of $F$ on finite complexes. For instance,
there are nontrivial functors like $\Map(-,QS^0)$ that are
contractible on all finite complexes; the methods in this section 
evidently will not be able to say anything about these functors.

For the results in this section, considering functors to spectra is critical. 
The main way in which we use spectra as the target category
is embodied in the following lemma:
\begin{lemma}[Basic Lemma for Spectra]
  \label{lem:basic-lemma-for-spectra}
  If $\cube{X}$ is a strongly co-Cartesian $S$-cube of spectra, with
  $\abs{S}=n+1$, 
  and $F$ is an $n$-excisive functor taking values in spectra, 
  then $$F\cube{X}(S) \simeq \hocolim_{U \in P_1(S)} F\cube{X}(U).$$
  That is, $F\cube{X}$ is co-Cartesian.
\end{lemma}
\begin{proof}
  Recall that $\Power_1(S)$ is the power set of $S$ with the terminal
  object removed.
  Since $F$ is $n$-excisive, it takes co-Cartesian $(n+1)$-cubes to
  Cartesian cubes. In the category of spectra, Cartesian cubes are
  also co-Cartesian, so the result follows.
\end{proof}



\section{$L_n F$ Is $n$-excisive}

Recall that the functor $L_n F(X)$ is given by the realization of a
simplicial spectrum:
\begin{equation}
\label{eq:LnF}
L_n F(X) 
=
\left\lvert
[k] \mapsto \bigvee_{(C_0,\ldots,C_k)} 
F(C_0) 
\wedge 
\left(
\Hom(C_0,C_1)\times \cdots \times \Hom(C_k,X)
\right)_+
\right\rvert
.
\end{equation}

We begin by showing that to know $L_n F$ is $n$-excisive, it is enough
to know that each simplicial dimension of $L_n F$ is $n$-excisive.

\begin{lemma}
\label{lem:levelwise-n-excisive-is-n-excisive}
   Let $F_\cdot$ be a simplicial functor from spaces to spectra. If
   each $F_i$ is $n$-excisive, then $\realization{F_\cdot}$ is $n$-excisive.
\end{lemma}
\begin{proof}
   Let $\cube{X}$ be a strongly co-Cartesian $S$-cube of spaces, with
   $\abs{S}=n+1$. If $\realization{F_\cdot \cube{X}}$ is Cartesian,
   then $\realization{F_\cdot}$ is $n$-excisive.
   Cartesian and co-Cartesian are equivalent notions for spectra, 
   so it suffices to show
   that $\realization{F_\cdot(\cube{X})}$ is co-Cartesian.

   Each $F_i$ is $n$-excisive, so $F_i \cube{X}(S) \simeq \hocolim_{U\in
     P_1(S)} F_i \cube{X}(U)$. Applying the realization functor to
   both sides, and noting that realization is a homotopy colimit and
   colimits commute, we have $\realization{F_\cdot \cube{X}(S)} \simeq
   \hocolim_{U\in P_1(S)} \realization{F_\cdot \cube{X}(U)}$. This
   shows that $\realization{F_\cdot \cube{X}}$ is co-Cartesian, as
   desired. 
\end{proof}

\begin{proposition}
  \label{prop:L-n-F-degree-n}
  If $F$ is a functor from spaces to spectra, then $L_n F$ is
  $n$-excisive. 
\end{proposition}
\begin{proof}
   Lemma~\ref{lem:levelwise-n-excisive-is-n-excisive} shows that it
   suffices to demonstrate that each level of the simplicial spectrum
   in \eqref{eq:LnF} is $n$-excisive. Since in the category of
   spectra, finite coproducts are equivalent to products, and the product of
   $n$-excisive functors is $n$-excisive, we only need to show that
   the functor
   \begin{align*}
   F_k(X) 
   &= 
   F(C_0) 
   \wedge 
   \left(
     \Hom(C_0,C_1)\times \cdots \times \Hom(C_k,X)
   \right)_+
   \\ 
   \intertext{is $n$-excisive. Now for spaces it is easy to see that
     $(A\times B)_+ \cong A_+ \wedge B_+$, so this can be rewritten as}
   F_k(X)
   &\cong
   F(C_0) 
   \wedge 
   \left(
   \Hom(C_0,C_1)_+
   \wedge \cdots \wedge \Hom(C_k,X)_+ 
   \right)
   .
   \\
   \intertext{Using the associativity of smash product (of a space
     with a spectrum), we have:}
   F_k(X) 
   &\cong
   \left( 
     F(C_0) 
     \wedge 
     \Hom(C_0,C_1)_+
     \wedge \cdots \wedge 
     \Hom(C_{k-1},C_k)_+
     \right) \wedge
     \Hom(C_k,X)_+ 
     .
   \end{align*}
   This is the smash product of a constant functor (which we will
   denote $\mathbf{C}$) to spectra with
   $\Hom(C_k,X)_+$. The category over which we have taken the left Kan
   extension consists of finite sets of cardinality at most $n$, and
   $C_k$ is one of these sets. The space of maps of a finite set into
   $X$ is just a
   product of copies of $X$; the space of pointed maps of $\Ord{n}$ into $X$ is
   isomorphic to $X^{\times n}$. In \cite[Example~3.5]{Cal2},
   Goodwillie shows that 
   $\mathbf{C} \wedge \left( X^{\times
     n}_+\right)$ is $n$-excisive for any spectrum $\mathbf{C}$. Therefore,
   $L_n F$ is $n$-excisive.
\end{proof}

\begin{lemma}
\label{lem:left-Kan-limit-axiom}
For any functor $F$ from spaces to spaces or spectra, and any
subcategory $\cat{C}$ of spaces whose objects are finite CW-complexes, the
left Kan extension $L_\cat{C} F$ satisfies the limit axiom
(\ref{def:limit-axiom}). 
\end{lemma}
\begin{proof}
We need to show that if $Y$ is equivalent to the filtered homotopy
colimit of its finite subcomplexes $\Set{Y_\alpha}$, then $L_\cat{C}
F(Y) \simeq \hocolim L_\cat{C} F(Y_\alpha)$. 
If $C_n$ is a finite complex, then its image is compact, and hence lies
inside some finite $Y_\alpha$, so $\Hom(C_n,-)$ commutes with filtered
homotopy colimits. In the definition of the homotopy left Kan extension
(Equation~\eqref{eq:left-kan-spaces}), the only term that involves $Y$
or $Y_\alpha$ is $\Hom(C_n,-)$, where $C_n \in \Obj(\cat{C})$, so
$L_\cat{C} F$ satisfies the limit axiom because $\Hom(C_n,-)$ does for all
$C_n\in \Obj(\cat{C})$. 
\end{proof}

\begin{corollary}
\label{cor:LnF-limit-axiom}
If $F$ is a functor from spaces to spectra,
then the functor $L_n F$ satisfies the limit axiom.
\end{corollary}
\begin{proof}
The sets $[n]$ are all finite CW complexes, so this is immediate from
Lemma~\ref{lem:left-Kan-limit-axiom}. 
\end{proof}

\section{Excisive Functors Are Left Kan Extensions}

In this section, we establish that any $n$-excisive functor from
spaces to spectra that satisfies the limit axiom
(\ref{def:limit-axiom}) commutes with the realization of a simplicial
spaces. That is, such a functor $F$ is equivalent to its own left Kan
extension $L_n F$.

We begin by establishing the lemma that $L_n F$ and $F$ agree on
finite sets.

\begin{lemma}
\label{lem:LnF-F-agree-finite-sets}
  If $F$ is an $n$-excisive functor from spaces to spectra, then for
  all finite sets $X$, the map $L_n F(X) \rightarrow F(X)$ is an
  equivalence. 
\end{lemma}
\begin{proof}
  Let $m = \abs{X}$ be the cardinality of $X$.
  If $m\le \abs{\Ord{n}}$, then by
  Proposition~\ref{prop:left-kan-agrees-with-orig}, the map $L_n F(X)
  \rightarrow F(X)$ is an
  equivalence. If $m>\abs{\Ord{n}}$, then we may assume by induction that the
  result is true for all smaller $m$. 

  Let $S = \Set{1,2,\ldots,n+1}$.
  Define pointed sets $W_u$ for $u\in S$ as follows:
  $$
  W_u =
  \begin{cases}
    \Set{\basept,u} & \text{if $u\not= n+1$} \\
    \Set{\basept,n+1,\ldots,m-1} & \text{if $u=n+1$}
  \end{cases}
  $$
  Let $\cube{X}$ be the strongly co-Cartesian $S$-cube given by 
  $\cube{X}(U) = \bigvee_{u\in U} W_u$.
  Note that $\cube{X}(S) = \Set{\basept,1,2,\ldots,m-1}$ has $m$
  points, and hence is isomorphic to $X$.

  For $U\subsetneq S$, we have $\abs{F\cube{X}(U)} <
  m$, so 
  by our induction hypothesis, $L_n F
  \cube{X}(U) \simeq F\cube{X}(U)$ for all $U\in\Power_1(S)$.
  Now $L_n F$ is $n$-excisive (Proposition~\ref{prop:L-n-F-degree-n}),
  and $F$ is $n$-excisive, so the  Basic Lemma for 
  Spectra (\ref{lem:basic-lemma-for-spectra}) shows that both $L_n F
  \cube{X}$ and $F \cube{X}$ are co-Cartesian, so 
  we have an
  equivalence on the terminal vertices as well. 
  That is, $L_n F(X) \simeq F(X)$.
\end{proof}


\begin{theorem}
\label{thm:finite-degree=left-kan}
  Let $F$ be an $n$-excisive functor from spaces to spectra that
  satisfies the limit axiom (\ref{def:limit-axiom}), and let $L_n F$ be
  as defined in \eqref{eq:LnF}.
  For any space $X$, the map $L_n F(X)\rightarrow F(X)$ is a weak
  equivalence. 
\end{theorem}
\begin{proof}
%
  The functors $L_n F$ and $F$ are both homotopy functors that satisfy
  the limit axiom, so it is sufficient to establish that the theorem
  is true when $X$ is the realization of a finite simplical set.
  We proceed by induction on the dimension of $X$, where by dimension
  we mean: as usual, $\dim(X_\cdot)$ is the largest $k$ such that
  $X_k$ contains nondegenerate elements, and 
  $$\dim (X) = \min \Set{ \dim(X_\cdot) 
    \suchthat X_\cdot \text{ finite and } \realization{X_\cdot}\simeq
  X} .$$



  \textbf{Base Case.} When $\dim(X) = 0$, the space $X$ is a finite
  set of points. Lemma~\ref{lem:LnF-F-agree-finite-sets} shows that
  the map $L_n F(X) \rightarrow F(X)$ is an equivalence for all
  finite~$X$.  

  \textbf{Induction Case: Adding an $(m+1)$-cell.} We now proceed by
  induction, assuming that if $\dim(X)\le m$, then $L_n F(X) \simeq
  F(X)$. To add $(m+1)$-cells to $X$, we consider a second induction
  on the minimal number of $(m+1)$-cells needed to build $X$.
  
  To form a space $Y$ by attaching an $(m+1)$-cell to $X$ along $f$, 
  one forms the pushout:
  \begin{equation}
    \label{eq:attaching-po}
\xymatrix{
    S^m
    \ar[r]
    \ar[d]^f
    &
    D^{m+1}
    \ar[d]
    \\
    X
    \ar[r]
    &
    Y
  }    
  \end{equation}
  In order to use the $n$-excisive properties of $L_n F$ and $F$, we
  need to blow up this cubical diagram to be of dimension at least
  $n+1$. We will do this by subdividing $D^{m+1}$ until it has at
  least $n+1$ simplices of dimension $(m+1)$, and then gluing them
  into its $m$-skeleton one by one.

  Let $D^{m+1}$ denote the standard $(m+1)$ simplex. Let $R(r)$ be the
  set of non-degenerate simplices of dimension $(m+1)$ in the $r$-fold
  subdivision of $D^{m+1}$, which we denote $sd_r D^{m+1}$.
  Choose $r \gg 0$ large enough that $\abs{R(r)} \ge
  n+1$, and for convenience let $R =
  R(r)$ for this $r$. Now form an $R$-cube $\cube{D}$ by gluing these
  $(m+1)$-simplices onto the $m$-skeleton of $sd_r ( D^{m+1} )$.
  Explicitly, 
  $$ \cube{D}(U) = \left(\bigcup_{r \in U} r\right) 
  \cup \Skel_{m} sd_r (D^{m+1}).$$
  Notice that $\cube{D}(R) = sd_r (D^{m+1})$, so this
  cube expresses the $(m+1)$-simplex as a pushout of dimension
  at least $n+1$.

  Instead of forming the exact analog of the pushout diagram in
  \eqref{eq:attaching-po}, we replace the space $X$ by another space
  $X'$, which is $X$ with the $m$-skeleton of $sd_r D^{m+1}$ glued on
  along the attaching map $f$.
  $$\xymatrix{
    S^m
    \ar[d]^f
    \ar[r]
    &
    \Skel_m sd_r D^{m+1}
    \ar[d]^{f'}
    \ar[r]
    &
    \cdots
    \ar[r]
    &
    sd_r D^{m+1}
    \ar[d]
    \\
    X
    \ar[r]
    &
    X'
    \ar[r]
    &
    \cdots
    \ar[r]
    &
    Y
    }$$
    While $X$ is not equivalent to $X'$, the space $X'$ still
    satisfies our induction hypothesis since we have not added any
    $(m+1)$-cells. We will not use $X$ itself further in this proof,
    except to identify the space $Y$ below as $X$ with an
    $(m+1)$-cell attached along~$f$.

  Define the $S = R\amalg \Set{\basept}$ cube $\cube{Y}$ to be the
  strongly co-Cartesian cube generated by $\cube{D}$ and the map
  $\cube{D}(\emptyset) \xrightarrow{f'} X'$.
  \begin{equation*}
    \cube{Y}(U) = 
    \begin{cases}
    \cube{D}(U) 
    & 
    \text{if $\basept \not\in U$} 
    \\
    \colim 
    \left( 
   X' \xleftarrow{f'} \cube{D}(\emptyset) \rightarrow \cube{D}(U-\Set{\basept}) 
    \right)
    & \text{if $\basept\in U$}
  \end{cases}
\end{equation*}
From its construction, it is evident that $\cube{Y}(S) = Y$, where $Y$
is the pushout $Y = \colim ( X \xleftarrow{f} S^m \rightarrow D^{m+1})$. 
 
$L_n F$ and $F$ are $n$-excisive (\ref{prop:L-n-F-degree-n}), and
$\abs{S}\ge n+1$, so
$L_n F \cube{Y}$ and $F\cube{Y}$ are co-Cartesian cubes
(\ref{lem:basic-lemma-for-spectra}). Therefore, to show $L_n F
\cube{Y}(S) \simeq F\cube{Y}(S)$ 
(that is, $L_n F(Y) \simeq
F(Y)$), we need only show $L_n F\cube{Y}(U) \simeq
F\cube{Y}(U)$ for $U \subsetneq S$.

All of the non-terminal vertices $\cube{D}(U)$ of $\cube{D}$ are
subdivisions of $D^{m+1}$ with some $(m+1)$-cells missing. All of
these retract relative to their boundary to complexes of dimension
$m$, so they all satisfy our induction hypothesis. Furthermore, this
retraction relative to the boundary also shows that $\cube{Y}(U \amalg
\Set{\basept})$ satisfies the induction hypothesis. Finally,
$\cube{D}(R) \simeq \basept$, so on all nonterminal vertices of
$\cube{Y}$, we have $L_n F\cube{Y}(U) \simeq F\cube{Y}(U)$. That is
what we needed to establish.
\end{proof}

\begin{corollary}
\label{cor:finite-degree-commutes-with-realizations}
Let $F$ be an $n$-excisive functor from spaces to spectra
satisfying the limit axiom (\ref{def:limit-axiom}). 
Then $F$ commutes with realizations of all simplicial spaces.
\end{corollary}
\begin{proof}
Theorem~\ref{thm:finite-degree=left-kan} shows that $F$
is equivalent 
to a left Kan extension over a subcategory of spaces containing only
objects of dimension~$0$. Then 
Proposition~\ref{prop:left-kan-commutes-with-some-realizations}  shows
that this left Kan extension commutes with realization of
$(-1)$-connected simplicial spaces. But all  (nonempty)  spaces are
$(-1)$-connected. 
%
\end{proof}
 



%
%

\chapter{Analyticity And Realization}
\label{chap:realization}

In this chapter we establish the result that analytic functors from
spaces to spaces
commute with realizations of highly connected spaces, and hence are
equivalent to certain left Kan extensions. In order to do this, we
also establish properties of analytic functors that show our intuition
about the behavior of the coefficient spectra is justified.

\section{Analytic Functors Have Connective Coefficient Spectra}
\label{sec:analytic-functors-connective}

In this section, we establish the following theorem, which states
that an analytic functor has coefficient spectra that are bounded
below. 
\begin{theorem}
\label{thm:analytic-bounded-below}
Let $F$ be a functor with coefficient spectra $\mathbf{C_i}$ (defined
only for $i\ge 1$). If $F$
is $r$-analytic with universal analyticity constant $c$, so $F$
satisfies $E_n(rn-c)$ for all $n$, then $\pi_j(\mathbf{C_{n+1}}) = 0$ for
$j< c-rn$. In particular, all $\mathbf{C_i}$ are bounded below.
\end{theorem}


If ``analyticity'' is to be a well-behaved concept, we need to prove
that if $F$ is analytic, then so is $P_n F$. We do this by showing
that if $F$ satisfies $E_m(c)$, then so does $T_n F$.
Goodwillie~\cite{Cal2} proves that when $m= n$, the functor $T_m F$
actually satisfies at least $E_m(c-1)$; we will reiterate his argument
as part of establishing the fact we are most interested in. We begin
by recalling a technical proposition.

\begin{proposition}
\label{prop:calc-2-prop-1-22}
(\cite[Proposition~1.22]{Cal2}) Let $\cube{X}$ be a functor from
$\Power_0(S)$ to $T$-cubes of spaces, and write $\cube{X}(U,V) =
(\cube{X}(U))(V)$. For each $U\in\Power_0(S)$, let $k_U$ be a constant
so that the $T$-cube
$\cube{X}(U)$ is $k_U$-Cartesian. Then the $T$-cube $V \mapsto
\holim(U\mapsto\cube{X}(U,V))$ is $k$-Cartesian with
$k=\min\Set{1-\abs{U}+k_U}$. 
$\qed$
\end{proposition}

Proposition~\ref{prop:calc-2-prop-1-22} is immediately applicable to
the $T_n$ construction. 
Our main interest in this is for $n=m+1$, where the $E_m(c)$ condition
satisfied is not improved by $T_{m+1}$.
\begin{corollary}
\label{cor:F-Enc-TnF-Enc}
If $F$ satisfies $E_m(c)$, then for $n\le m+1$, so does $T_n F$.
\end{corollary}
\begin{proof}
Let $\cube{Y}$ be a strongly co-Cartesian $T$-cube, with $T=[m]$ and
the map $\cube{Y}(\emptyset)\rightarrow\cube{Y}(\Set{j})$ being
$k_j$-connected.  Recall that $T_n F(X) = \holim_{U\in\Power_0(S)}
F(X*U)$, for $S=[n]$, so Proposition~\ref{prop:calc-2-prop-1-22}
applies to the consideration of the functor $\cube{X}(U,V) \mapsto
F(\cube{X}(V)*U)$ (which is only defined for $\emptyset\not=U\subset S$). 
Since $F$ satisfies $E_m(c)$ and $*U$ raises connectivity by one (for
$U\not=\emptyset$), the cube $\cube{X}(U)$ is $k_U$-Cartesian,
with $k_U = \Sigma_{j=0}^{m} (k_j+1) - c = \Sigma k_j - c + (m+1)$.
Applying Proposition~\ref{prop:calc-2-prop-1-22}, we see that 
the $T$-cube $V \mapsto T_n F(\cube{Y}(V))$ is $k$-Cartesian, with $k
= \min\Set{1-\abs{U}+k_U}$. As $\abs{U} \le n+1$, we know $k\ge
1-(n+1)+ \Sigma k_j - c + (m+1) = 1-(m-n)+ \Sigma k_j - c$. This shows
that $T_n F$ satisfies $E_m(c+m-n-1)$. Therefore, for $m\ge n-1$, if
$F$ satisfies $E_m(c)$, then so does $T_n F$. 
\end{proof}

This argument can now be used to show that $P_n F$ satisfies 
the same stable excision condition as $F$ for $n$-cubes. Of course,
applying $P_n F$ to larger cubes results in Cartesian cubes. 

\begin{example}
In general, $P_{n+1} F$ may have a better constant $E_n(c)$ than $F$
does. Consider the functor from spaces to spectra given by 
$$F(X) = \left( \mathbf{HZ} \wedge X \right) \times 
\left( \mathbf{S}^{-2}\wedge \mathbf{HZ} \wedge X \wedge X \right) .$$
This functor satisfies $E_0(2)$ since it takes $(-1)$-connected maps
to $(-3)$-connected maps, but $P_1 F(X) = \mathbf{HZ} \wedge X$
satisfies $E_0(0)$. The functor $F$ is $(-1)$-analytic with constant
$c=-2$; this example shows that the increasing of the constant
$E_n(c)$ when passing from $F$ to $P_{n+1} F$ relates to the constant,
not the analyticity or radius of convergence. 
\end{example}



\begin{corollary}
\label{cor:F-Enc-Pn+1F-Enc}
If $F$ satisfies $E_{n}(c)$, then so does $P_{n+1} F$. 
\end{corollary}
\begin{proof}
Recall that $P_{n+1} F = \colim_k (T_{n+1})^k F$. Let $\cube{X}$ be a
strongly co-Cartesian $S$-cube, with $S=[n]$, and let $k_i$ denote the
connectivity of the map $\cube{X}(\emptyset) \rightarrow \cube{X}(\Set{i})$.
Suppose $F$ satisfies $E_n(c)$.
By Corollary~\ref{cor:F-Enc-TnF-Enc}, this implies $T_{n+1} F$ satisfies
$E_n(c)$, and hence by induction all $(T_{n+1})^k F$ satisfy $E_n(c)$. We
need to establish that the colimit also satisfies the same stable
excision condition. 
This follows because homotopy
groups commute with directed colimits; that is, directed
homotopy colimits preserve injections and surjections on
homotopy groups, and hence $k$-connected maps.
\end{proof}

Recall that $D_n F$ is the homogeneous $n$-excisive functor that is
the homotopy fiber of the map from the $n$-excisive approximation $P_n
F$ to the $(n-1)$-excisive approximation $P_{n-1} F$.

\begin{lemma}
\label{lem:F-Enc-Dn+1F-Enc}
If $F$ satisfies $E_n(c)$, then $D_{n+1} F$ satisfies $E_n(c)$ as
well. 
\end{lemma}
\begin{proof}
We first show that we can reduce to considering strongly co-Cartesian
cubes with contractible initial object. This type of cube is
$k$-Cartesian if the (homotopy inverse limit of the) punctured cube is
$k$-connected, so we can determine Cartesian-ness by connectivity of a
space. We then commute homotopy inverse limits to compute the
connectivity of the punctured $D_{n+1}$ cube from that 
of $P_{n+1}$ and $P_n$. 

Consider the fibration sequence
$$ D_{n+1} F \rightarrow P_{n+1} F \rightarrow P_n F.$$
By Corollary~\ref{cor:F-Enc-Pn+1F-Enc}, the total space $P_{n+1} F$
satisfies $E_n(c)$, and of course $P_n F$ is $n$-excisive.

Let $T$ be a set of cardinality $n+1$, 
and let $\cube{X}$ be a strongly co-Cartesian $T$-cube.
Define the strongly co-Cartesian $T$-cube $\cube{Y}$  by
coning off the initial vertex $\cube{X}(\emptyset)$ of $\cube{X}$:
\begin{equation*}
\cube{Y}(U) = 
\colim \left( 
C(\cube{X}(\emptyset)) 
\leftarrow 
\cube{X}(\emptyset)
\rightarrow
\cube{X}(U)
\right)
\end{equation*}
Since the functors $D_{n+1}$, $P_{n+1}$
and $P_n$ are all $(n+1)$-excisive, they each take the $(n+2)$-cube
$\cube{X}\rightarrow\cube{Y}$ to a Cartesian cube. 
So after applying any one of these
functors, if the functored sub-cube $\cube{Y}$
is $k$-Cartesian then the functored sub-cube $\cube{X}$ is
$k$-Cartesian (by \cite[Proposition~1.6, p. 303]{Cal2}).
The sub-cube $\cube{Y}$ has contractible initial vertex, which is
the case we wanted to reduce to.
Now assume, using this reduction if necessary, that $\cube{X}$ is a
strongly co-Cartesian $T$-cube with contractible initial vertex. 
Since $D_{n+1}F$ is reduced and 
$\cube{X}(\emptyset) \simeq \basept$, the connectivity of the map 
$$ D_{n+1} F \cube{X}(\emptyset) \rightarrow 
\holim_{U \in \Power_0(T)} D_{n+1} F \cube{X}(U) $$
is determined by the connectivity of $\holim_{U \in \Power_0(T)}
D_{n+1} F \cube{X}(U)$ (since $D_{n+1} F \cube{X}(\emptyset)\simeq
\basept{}$). If our functor is not reduced, there is a fibration over 
$F(\basept)$ with fiber a reduced functor, so there
is no real difference in the arguments in this case; they are just
made relative to $F(\basept)$.

We can then compute:
\begin{align*}
\holim_{U \in \Power_0(T)} D_{n+1} F \cube{X}(U) 
&= 
\holim_{U \in \Power_0(T)} 
\hofib \left( P_{n+1} F \cube{X}(U) \rightarrow P_{n} F \cube{X}(U) \right)
\\
&\simeq
\hofib 
\left(
\holim_{U \in \Power_0(T)} P_{n+1} F \cube{X}(U)
\rightarrow
\holim_{U \in \Power_0(T)} P_{n} F \cube{X}(U)
\right)
\end{align*}
The $n$-excisiveness of $P_n F$ implies that the inverse limit of the
punctured $P_{n}$ cube is equivalent to $P_n F \cube{X}(\emptyset) \simeq
F(\basept)$, and the inverse limit of the punctured $P_{n+1}$ cube has
connectivity at least $\sum k_i - c$ relative to $F(\basept)$ because
$P_{n+1} F$ satisfies $E_n(c)$, so the 
homotopy fiber also has connectivity at least $\sum k_i - c$. Hence
$D_{n+1} F$ satisfies $E_n(c)$ as well.
\end{proof}

\begin{lemma}
\label{lem:homogeneous-Enc-bounded-spectra}
Let $\mathbf{C}$ be a spectrum with a $\Sigma_{n+1}$ action, and let
$F$ be the functor from spaces to spectra given by  
$$F(X) = \mathbf{C} \wedge_{h \Sigma_{n+1}} X^{\wedge(n+1)},$$ so $F$ 
is homogeneous of degree $n+1$, with $n\ge 0$. If 
$F$ satisfies $E_n(c)$, then $\pi_j \mathbf{C} = 0$ for $j<-c$.
\end{lemma}
\begin{proof}
We will show that $\mathbf{C}$ cannot have nonzero homotopy groups in
dimensions 
lower than is claimed by showing that if so, the result would be a
functor that does not satisfy $E_n(c)$.
The condition $E_n(c)$ gives information about the Cartesian-ness of
$F$ applied to certain cubes, so 
we will compute a bound on the Cartesian-ness by computing the
connectivity of the total fiber.  

Recall that in the
category of spectra, the total cofiber and total fiber are related by
a shift in dimension equal to the dimension of the cube.
To compute the total fiber of our functor $$F(X) =  \hocolim_{\Sigma_{n+1}} 
( \mathbf{C} \wedge X^{\wedge(n+1)} )$$
applied to the cube, we first
compute the total cofiber of the $(n+1)^{\text{st}}$ smash
power, then smash with $\mathbf{C}$, then take $\Sigma_{n+1}$ orbits,
and finally loop back $n$ times for the dimension shift. That is, we
compute 
\begin{align*}
\hocolim_{\Sigma_{n+1}} 
\left( 
\mathbf{C} \wedge 
\text{total} \cofib
\left( X^{\wedge(n+1)} \right) 
\right) 
&\simeq
\hocolim_{\Sigma_{n+1}} 
\left( 
\text{total} \cofib_{X\in\cube{X}}
\left( \mathbf{C} \wedge  X^{\wedge(n+1)} \right) 
\right) 
\\
&\simeq
 \text{total} \cofib_{X\in\cube{X}}
\left(
\hocolim_{\Sigma_{n+1}} 
\left( 
\mathbf{C} \wedge  X^{\wedge(n+1)}
\right) 
\right) 
\\
&\simeq
 \text{total} \cofib_{X\in\cube{X}} F(X)
.
\end{align*}
The second equivalence is because both the total cofiber and the
homotopy orbits are colimit constructions, and hence commute.
This shows that the cofiber we compute is actually that of $F$ applied
to the cube. 

Let $X$ be a space and consider the strongly co-Cartesian $(n+1)$-cube
$\cube{X}$ generated by $\cube{X}(\emptyset) = \basept{}$ and
$\cube{X}(\Set{i}) = X$, so $\cube{X}(U) = \bigvee_U X.$
Let $\cube{Y}$ be the cube $\cube{X}^{\wedge (n+1)}$, with
$\cube{Y}(U) = \cube{X}(U)^{\wedge (n+1)} $.

The total cofiber of the cube $\cube{Y}$ is equivalent to
the $n+1$ cross effect 
of the $n+1$ smash power, $cr_{n+1}(\bigwedge^{n+1})$. Writing $X_i
\cong X$ to make the action of $\Sigma_{n+1}$ clear, this is:
$$
cr_{n+1}(\bigwedge^{n+1})(X_1,\ldots,X_{n+1}) 
\xleftarrow{\simeq}
\bigvee_{\sigma\in\Sigma_{n+1}} X_{\sigma(1)} \wedge \cdots \wedge
X_{\sigma(n+1)} 
.$$
Fortunately, it is easy to see that 
the right hand side is a free $\Sigma_{n+1}$ space, 
so smashing
with $\mathbf{C}$ and taking  homotopy orbits gives $\mathbf{C}\wedge X^{\wedge
  (n+1)}$. Hence for this cube, the total fiber is $\Omega^n (
\mathbf{C}\wedge X^{\wedge (n+1)})$.

The Cartesian-ness of the cube $F\cube{X}$ is determined by the
connectivity of the total fiber. When the space $X$ is $m$-connected
then $X^{\wedge(n+1)}$ is $((n+1)(m+1)-1) = ((n+1)m+n)$-connected.
If $\mathbf{C}$ has its bottom nonzero homotopy group in dimension $w$,
the total fiber has connectivity
$(n+1)m+n-n+w = (n+1)m+w$. Since $F$ satisfies $E_n(c)$, we must
have $w\ge -c$.
\end{proof}

\begin{corollary}
\label{cor:homogeneous-Enc-bounded-spaces}
Let $\mathbf{C}$ be a spectrum with a $\Sigma_{n+1}$ action, and let
$F$ be the functor from spaces to spaces given by  
$$F(X) = \LoopInfty(\mathbf{C} \wedge_{h \Sigma_{n+1}} X^{\wedge(n+1)}),$$
 so $F$ 
is homogeneous of degree $n+1$, with $n\ge 0$. If 
$F$ satisfies $E_n(c)$, then $\pi_j \mathbf{C} = 0$ for $j<-c$.
\end{corollary}
\begin{proof}
Choosing $X$ to be highly enough connected, all of the maps in
the cube of Lemma~\ref{lem:homogeneous-Enc-bounded-spectra} are
connected enough that 
the Cartesian-ness of the cube is determined by the connectivity of
the total fiber. Again, if $X$ is connected enough, the total fiber will
be connective, so after the application of $\LoopInfty$  (which
preserves fibers), it will have the same connectivity as the fiber as
spectra, so we the result for spaces follows.
\end{proof}

\begin{proof}[Proof of Theorem~\ref{thm:analytic-bounded-below}]
Without loss of generality, we may assume that
$F$ is reduced. We do this by replacing $F$ with the
functor $\tilde{F}(X) = \hofib \left(F(X)\rightarrow F(0)\right)$ and
observing that they satisfy the same excision conditions $E_n(c)$ for
all $n\ge 0$.
If $F$ is $r$-analytic, then $F$
satisfies $E_n(rn-c)$, so by Lemma~\ref{lem:F-Enc-Dn+1F-Enc}, $D_{n+1}
F$ satisfies $E_n(rn-c)$, so by
Corollary~\ref{cor:homogeneous-Enc-bounded-spaces}, 
$\pi_j \mathbf{C_{n+1}} = 0$ for $j<c-rn$. 
\end{proof}


\section{$\LoopInfty$ Commutes With Certain Realizations}
\label{sec:loopinfty-commutes-with-certain-realizations}

The result of this section is that 
when all of the spectra $X_i$ in a simplicial spectrum $X_{\cdot}$ are
connective, the functor $\LoopInfty$ can be applied before or after
realization, with the same results. 
(Recall that we call a spectrum connective if all of its
negative homotopy groups are zero.)

\begin{theorem}
  \label{thm:loop-infty-commutes-with-connective-spectra}
If $X_{\cdot}$ is a simplicial connective spectrum, then the
simplicial spectrum
$\LoopInfty
\realization{[n] \mapsto X_{n}}
$
is equivalent to the simplicial spectrum
$
\realization{[n] \mapsto \LoopInfty X_{n}}.
$
\end{theorem}

The remainder of this section consists of the proof of this theorem
and subsidiary results required therein.
Let $S^1_{\cdot}$ be the standard model $\Delta^1/\partial \Delta^1$
(where $\Delta^1$ is $[n]\mapsto \Hom([n],[1])$) for the simplicial
1-sphere.  

\begin{lemma}
  \label{lem:loop-infty-commutes-with-suspension}
  If $X$ is a connective spectrum,  then 
  $$
  \realization{[n] \mapsto \LoopInfty ( S^1_{n} \wedge X )}
  $$
  is equivalent to
  $$
  \LoopInfty \realization{[n] \mapsto S^1_{n} \wedge X} 
 .$$
 Furthermore, both have the same infinite loop space structure.
\end{lemma}

A proof of this lemma different from the one that follows appears in
the literature in 
\cite{Beck:classifying-spaces-for-homotopy-everything-H-spaces}.

\begin{proof}
  Consider the levelwise cofiber sequence of simplicial sets
  $S^0_{\cdot} \rightarrow D^1_{\cdot} \rightarrow S^1_{\cdot} $.
  Applying the two functors in question, we have the following diagram:
  $$
  \xymatrix{
    \realization{\LoopInfty ( S^0_{\cdot} \wedge X ) } \ar[r]\ar[d] &
    \realization{\LoopInfty ( D^1_{\cdot} \wedge X ) } \ar[r]\ar[d] &
    \realization{\LoopInfty ( S^1_{\cdot} \wedge X ) } \ar[d] 
    \\
    \LoopInfty \realization{ S^0_{\cdot} \wedge X } \ar[r] &
    \LoopInfty \realization{ D^1_{\cdot} \wedge X } \ar[r] &
    \LoopInfty \realization{ S^1_{\cdot} \wedge X }
    }
  $$
  Note that $S^0_{\cdot} \wedge X$ is a trivial simplicial set, so the
  leftmost map is an equivalence. Also, $D^1_{\cdot}$ is simplicially
  contractible, so both of the spaces appearing in the middle are
  contractible (and hence the map is an equivalence). The bottom
  row is a fibration (up to homotopy) because cofiber sequences of
  spectra are equivalent to fiber sequences, and $\LoopInfty$
  preserves fiber sequences. 

  The top row is also a fiber sequence. This depends on the fact that
  $\LoopInfty(S^1_{\cdot} \wedge X)$ is the simplicial bar
  construction on $\LoopInfty X$, and that this produces a delooping
  of $\LoopInfty X$ (by  
  \cite[Proposition~1.4, p.~295]{Segal:categories-and-cohomology-theories}).
  Using the stated model for $S^1_{\cdot}$, we have 
  $S^1_{n} \wedge X = [n] \wedge X = \bigvee_n X \simeq \prod_n X$, 
  since finite coproducts and products
  are equivalent in spectra. Applying $\LoopInfty$ (which commutes
  with products), we have the
  simplicial object $[n] \mapsto \left( \LoopInfty X \right)^{\times
    n}$. We leave skeptical readers to verify for themselves that 
  the fold map $X \vee X \rightarrow X$
  of spectra induces the product map for the $H$-space $\LoopInfty X$
  (via the equivalence $X \times X \simeq X \vee X$).
  Using the five lemma, we immediately find that the right hand map is
  an equivalence on $\pi_i$, for $i\ge 1$. The space
  $\realization{\LoopInfty ( S^1_{\cdot} \wedge X ) }$  is always
  connected, because there are no zero simplices.
  Since $X$ is connective, we also have 
  $\LoopInfty \realization{ S^1_{\cdot} \wedge X }$ connected, so
  $\pi_0 = 0$ in both cases. This shows that the right-hand map is a
  weak equivalence. 
%
\end{proof}

\begin{corollary}
  If $X$ is a connective spectrum,  then 
  $
  \realization{\LoopInfty ( S^n_{\cdot} \wedge X )}
  \simeq 
  \LoopInfty \realization{S^n_{\cdot} \wedge X} 
  $.
\end{corollary}
\begin{proof}
  Write $S^n_{\cdot}$ as the diagonal of the bisimplicial set
  $S^1_{\cdot} \wedge S^{n-1}_{\cdot}$. Use the Eilenberg-Zilber
  theorem to replace with the whole bisimplicial set, and apply
  Lemma~\ref{lem:loop-infty-commutes-with-suspension} inductively.
\end{proof}

We will demonstrate that all grouplike $H$-spaces satisfy the
$\pi_*$-Kan condition, giving us a large class of examples.

\begin{definition}[Simple space]
A connected space $X$ is called \emph{simple} if $\pi_1 X$ is abelian
and acts trivially on the higher homotopy groups.
A general space $X$ is called \emph{simple} if each component of $X$ is a
simple space.
\end{definition}

The following lemma appears as an exercise in
\cite{Bousfield-Friedlander:Gamma-Spaces}. 

\begin{lemma}
(\cite[B.3.1, p.~120]{Bousfield-Friedlander:Gamma-Spaces})
\label{lem:identify-pistar-kan} 
Let $X$ be a simplicial space, and let $[S^t,-]$ denote the unpointed
homotopy classes of (unpointed) maps out of $S^t$. 
If each $X_m$ is a simple space, then 
\begin{equation}
\label{eq:pit-pi0-fibration}
[S^t, X_\cdot]\rightarrow \pi_0 X_\cdot
\end{equation}
 is a fibration
of simplicial sets if and only if $X_\cdot$ satisfies the $\pi_*$-Kan
condition.
$\qed$
\end{lemma}
\begin{proof}
Recall from \ref{def:pi-star-Kan} that the $\pi_*$-Kan means that if
given $t\ge 1$ and $a\in X_{m+1}$ and a coherent collection $x_i \in
\pi_t(X_m,\partial_i a)$, with $0\le i\le m+1$ and $i\not= k$, there
exists a $y \in \pi_t(X_{m+1},a)$ with $\partial_i y = x_i$ for
$i\not= k$. Also, recall that map $p: E\rightarrow B$ of simplicial sets
is a fibration if given $a\in B_{m+1}$ and a coherent collection $x_i
\in E_m$ with $p(x_i) = \partial_i a$, for $0\le i\le m+1$ and $i\not=
k$, there exists a $y\in E_{m+1}$ with $p(y) = a$ and $\partial_i y =
x_i$ for $i\not= k$. 

In a simple space, elements of $\pi_t(X_m,\partial_i a)$
are in bijective correspondence with free homotopy classes of maps of
$S^t$ to $X_m$ that land in the component of $\partial_i a$. (In
general, allowing free homotopies identifies maps that are the same
orbit under the action of $\pi_1$.) 
Now comparing the definitions of a fibration of
simplicial sets and the $\pi_*$-Kan condition to verify that if $X$
satisfies the $\pi_*$-Kan 
condition, then \eqref{eq:pit-pi0-fibration} is a fibration.

If \eqref{eq:pit-pi0-fibration} is a fibration, then we can use this
fact to produce a $[y] \in [S^t, X_{m+1}]$ that lands in the same path
component as $a$, and satisfying $\partial_i [y] = [x_i]$ for $i\not= k$. 
This $[y]$ can be realized as a map $y: S^t \rightarrow X_{m+1}$ that
takes the basepoint of $S^t$ to $a$. Now $\partial_i y: (S^t,*)
\rightarrow (X_m, \partial_i a)$, and $[\partial_i y] = [x_i]$. But
the free homotopy classes of maps $[x_i]$ landing in the path
component of $\partial_i a$ are in one-to-one correspondence with the
elements of $\pi_t (X_m,\partial_i a)$, so $\partial_i y$ must
actually be $x_i$.  This shows that $X_\cdot$ satisfies the
$\pi_*$-Kan condition.
\end{proof}

\begin{corollary}
\label{cor:grouplike-H-spaces-pi-star}
If $X_\cdot$ is a simplicial grouplike $H$-space, then $X_\cdot$
satisfies the $\pi_*$-Kan condition.
\end{corollary}
\begin{proof}
All $H$-spaces are simple, so Lemma~\ref{lem:identify-pistar-kan} can
be used. 
The map $[S^t,X_\cdot] \rightarrow \pi_0 X_\cdot$ is
obviously surjective. We will show that both the source and target are
simplicial groups; all surjections of simplicial groups are
fibrations, so this will allow us to apply
Lemma~\ref{lem:identify-pistar-kan} to conclude that $X_\cdot$
satisfies the $\pi_*$-Kan condition.

Since $X_\cdot$ is a grouplike $H$-space, the
simplicial set $\pi_0 X_\cdot$ is actually a simplicial group. 

The set $[S^t, X_m]$ is a group with multiplication induced by the
$H$-space multiplication on $X_m$. A grouplike $H$-space only
satisfies the axioms for a group up to homotopy, but we are
considering homotopy classes of maps, so that is not a problem.
\end{proof}

\begin{corollary}
  \label{cor:loop-spaces-pi-star}
  Let $X_\cdot$ be a simplical space. The 
  simplicial space $[n] \mapsto \Omega X_n$
  satisfies the $\pi_*$-Kan condition. 
  In particular, simplicial infinite loop spaces (arising from $[n]
  \mapsto \LoopInfty X_n$) satisfy this condition.
\end{corollary}
\begin{proof}
  Loop spaces are grouplike $H$-spaces, so
  Corollary~\ref{cor:grouplike-H-spaces-pi-star} applies. 
\end{proof}

We are finally ready to begin the proof of the main result of this section.
\begin{proof}[Proof of
Theorem~\ref{thm:loop-infty-commutes-with-connective-spectra}]
  Let $X_{\cdot}$ be a simplicial connective spectrum.
  If necessary, begin by functorially replacing $X_{\cdot}$ by a
  weakly equivalent simplicial connective spectrum in which the
  degeneracy maps are cofibrations.
  Let $\overbar{X_n}$ denote the quotient of $X_n$ by the union of the
  images of the
  degeneracy maps. The degeneracy maps are cofibrations, so the strict
  quotient is a homotopy invariant. 
  Since we are in the category of spectra (a stable category), we can
  split off the degenerate elements up to equivalence, giving the
  standard decomposition:
  $$ X_n \simeq \bigvee_{\Surj(n,k)} \overbar{X_k} , $$
  where $\Surj(n,k)$ denotes the surjective maps from $[n]$ to $[k]$
  in $\Delta$.
  Each degeneracy map $s_j$ has an inverse $d_j$, so if each $X_k$ is
  connective, then so is each $\overbar{X_k}$.
  This decomposition lets us identify the cokernel of the inclusion of
  the $(n-1)$ skeletion into the $n$ skeleton. 
  
  We now proceed by induction up the simplicial skeleta. Let $X_{\le n}$
  denote the simplicial $n$-skeleton of $X_{\cdot}$. The inclusion of
  the $(n-1)$-skeleton into the $n$-skeleton gives rise to a
  (levelwise) cofibration sequence 
  \begin{equation}
    \label{eq:simplicial-skeleta}
  X_{\le n-1} \rightarrow X_{\le n} \rightarrow S^n_{\cdot} \wedge
  \overbar{X_n}  .
  \end{equation}
  Since cofibration sequences and fibration sequences are equivalent
  for spectra, and $\LoopInfty$ preserves fibration sequences, 
  $$ \LoopInfty X_{\le n-1} \rightarrow \LoopInfty X_{\le n}
  \rightarrow \LoopInfty ( S^n_{\cdot} \wedge \overbar{X_n} ) $$
  is a levelwise fibration sequence of spaces. Since $\overbar{X_n}$
  is connective, and $X_{\le n-1}$ is connective by induction, 
  the fibration is necessarily surjective (levelwise) on 
  $\pi_0$ (because the long exact sequence of a fibration sequence of
  spectra continues past $\pi_0$ to $\pi_{-1}$, which is $0$ in this
  case). Furthermore, since $\pi_0$ of an infinite loop space is a
  group, this is actually a surjective map of simplicial groups, 
  which is fortunately a fibration (\cite[Exercise~8.2.5,
  p.~262]{Weibel:homological-algebra}).
  By Corollary~\ref{cor:loop-spaces-pi-star}, both $\LoopInfty X_{\le
    n}$ and $\LoopInfty ( S^n_{\cdot} \wedge \overbar{X_n} )$ satisfy the
  $\pi_*$-Kan condition. 
  A theorem of Bousfield and Friedlander
  (Theorem~\ref{thm:Bousfield-Friedlander}) 
  now shows that we have a fibration after realization as well:
  \begin{equation*}
    \realization{\LoopInfty X_{\le n-1} } 
    \rightarrow 
    \realization{\LoopInfty X_{\le n}} 
    \rightarrow 
    \realization{\LoopInfty ( S^n_{\cdot} \wedge   \overbar{X_n} ) } .
  \end{equation*}

  The realization of \eqref{eq:simplicial-skeleta}  is still a
  cofibration sequence, and hence a fibration sequence, so we also
  have a fibration sequence
  \begin{equation*}
    \LoopInfty \realization{X_{\le n-1} } 
    \rightarrow 
    \LoopInfty\realization{ X_{\le n}} 
    \rightarrow 
    \LoopInfty \realization{S^n_{\cdot} \wedge   \overbar{X_n} } .
  \end{equation*}

  Combining these two fibration sequences gives rise to the
  commutative diagram: 
  $$
  \xymatrix{
    \realization{\LoopInfty X_{\le n-1} } \ar[r]\ar[d] &
    \realization{\LoopInfty X_{\le n} } \ar[r]\ar[d] &
    \realization{\LoopInfty ( S^n_{\cdot} \wedge \overbar{X_{n}} ) } \ar[d] 
    \\
    \LoopInfty \realization{X_{\le n-1} } \ar[r] &
    \LoopInfty \realization{X_{\le n} } \ar[r] &
    \LoopInfty \realization{ S^n_{\cdot} \wedge \overbar{X_{n}} }
    }
  $$
  The map on the fibers is an equivalence by induction, and the map on
  the bases is an equivalence by
  Lemma~\ref{lem:loop-infty-commutes-with-suspension}, so the map
  on the total spaces is certainly an isomorphism on $\pi_j$ for $j
  \ge 1$ (using the five lemma). Actually, all of the $\pi_0$ are
  abelian groups, so the five lemma implies the total spaces are
  equivalent. This is obvious for the bottom row, since
  $\pi_0\LoopInfty(X) = \pi_0 X$ is an abelian group. In the top row,
  we use the fact that $\pi_0 \realization{\LoopInfty Y_{\cdot}}$ is
  the quotient of $\pi_0 \LoopInfty Y_0$ by $\pi_0 \LoopInfty Y_1$,
  and the map $\LoopInfty Y_0 \rightarrow \LoopInfty Y_1$ is an
  infinite loop map (and hence a map of groups).
%
\end{proof}
\begin{corollary}
  \label{cor:homogeneous-functors-commute-with-some-realizations}
  Let $\mathbf{C}$ be a spectrum, and let $F(X) =
  \LoopInfty ( \mathbf{C} \wedge X^{\wedge n})_{h \Sigma_n}$ be a
  (homogeneous) functor from spaces to spaces. 
  If $X_\cdot$ is a simplicial space such that $\mathbf{C}
  \wedge X_i^{\wedge n}$ is connective for all $i$, then
  $F(\realization{X_\cdot}) \simeq \realization{F(X_\cdot)}$; that is,
  $F$ commutes with the realization of $X_\cdot$.
\end{corollary}
\begin{proof}
  Corollary~\ref{cor:finite-degree-commutes-with-realizations} shows
  that functors of finite degree from spaces
  to spectra commute with realizations.
  Proposition~\ref{thm:loop-infty-commutes-with-connective-spectra}
  shows that $\LoopInfty$ commutes with realizations of all simplicial
  connective spectra. The spectrum $\mathbf{C}\wedge X_i^{\wedge n}$ is
  connective, and taking homotopy orbits does not lower connectivity, 
  so the result is an immediate corollary of combining those
  two. 
\end{proof}

\section{Analytic Functors Commute With Highly Connected Realizations}
\label{sec:analytic-realizations}

This section uses the results of the previous two sections 
(\S\ref{sec:analytic-functors-connective} and
\S\ref{sec:loopinfty-commutes-with-certain-realizations}) to show that
analytic functors with the limit axiom commute with realizations of simplicial
$k$-connected spaces, for sufficiently large $k$.

\begin{theorem}
\label{thm:analytic-functor-commutes-with-realization}
Let $F$ be a reduced analytic functor from spaces to spaces satisfying
the limit axiom (\ref{def:limit-axiom}) and the stable excision condition
$E_n(rn-c)$ for all $n$ (as defined in \S\ref{sec:excisive-functors}). 
If $X_\cdot$ is a simplicial $k$-connected
space, with $k \ge \max(r,-c)$, then $F(\realization{X_\cdot}) \simeq
\realization{F(X_\cdot)}$. That is, $F$ commutes with realizations of
simplicial $k$-connected spaces.
\end{theorem}
\begin{proof}
Recall that smashing with a space of connectivity $k$ increases
connectivity by $(k+1)$. By Theorem~\ref{thm:analytic-bounded-below},
the coefficient spectrum $\mathbf{C}_{m+1}$, for $m\ge 0$, has its
bottom nonzero homotopy group in dimension $(c-rm)$. Computing
$\mathbf{C}_{m+1} \wedge X_i^{\wedge (m+1)}$, we find that its bottom
nonzero homotopy group is in dimension 
\begin{align*}
(c-rm)+(m+1)(k+1)
&= (k+c) + (k-r)m + (m+1)
\\
&\ge (0) + (0)m + m+1 ,
\end{align*}
so in particular it is always connected. Therefore, we may apply
Corollary~\ref{cor:homogeneous-functors-commute-with-some-realizations}
to conclude that $D_n F$ commutes with the realization of $X_\cdot$.

We now induct up the Taylor tower to show that each $P_n F$ commutes
with the realization of $X_\cdot$. 
To start the induction, note that $P_1 F = D_1 F$ is connected and
commutes with the realization of $X_\cdot$.  Then suppose inductively
that $P_n F$ is connected and commutes with the realization of
$X_\cdot$. Theorem~\ref{thm:delooping-Dn} says that $P_{n+1} F(X)$ can
be computed as the homotopy fiber of a map $P_n F(X) \rightarrow \Omega^{-1}
D_{n+1} F(X)$. Our connectivity estimate from the previous paragraph
shows that under our hypotheses, $D_{n+1} F$ is a simply connected for
$n\ge 1$ when evaluated on each $X_i$, and commutes with the realization of
$X_\cdot$. Since $D_{n+1} F(X_i)$ is connected, we may apply
Lemma~\ref{lem:Waldhausen-fibration-lemma} to compute $P_{n+1} F(
\realization {X_\cdot})$ as  $\realization{P_{n+1} F(
X_\cdot)}$, so $P_{n+1} F$ also commutes with the realization of
$X_\cdot$. 

The connectivity of the map $F(X) \rightarrow P_n F(X)$ grows with $n$
and the connectivity of $X$, provided that the connectivity of $X$ is
at least $r$; that is, $X$ is within the radius of convergence. This
is the case under our hypotheses, and each $P_n F$ commutes with the
realization of $X_\cdot$, so $F$ must also commute with the
realization of $X_\cdot$.
%
%
\end{proof}



%
%

\chapter{Cotriples For Additive Functors}
\label{chap:cotriples}
The most basic notion of ``degree'' of a functor is something that is
called ``additive degree''. A functor is additive degree $n$ if
$F(\bigvee^n X)$ is determined by $F(\bigvee^k X)$ for $k<n$, in the
sense that $F(\bigvee^n X)$ is the inverse limit of a certain diagram
involving only the $F(\bigvee^k X)$, for $k<n$. The algebraic
intuition for additive degree $n$, or ``$n$-additive'', functors has
the same roots as for 
$n$-excisive functors: a polynomial of degree $n$ is determined by its
values on $n+1$ points; that is, the set $\Set{\bigvee^k S^0 \suchthat
  0\le k\le n}$. The difference
between additive and excisive functors is that if $F$ is only
additive, there need not be any relationship between 
$F(X)$ and $\Omega F(\Sigma X)$. 
Our interest in additive functors stems from the fact that in many
cases 
the difference between $F$ and its additive approximation results from
a standard construction
called a ``cotriple''. Use of this construction provides a spectral
sequence to compute the homotopy groups of the $n$-additive
approximation to a functor, and 
plays an important role in our understanding of $n$-additive functors
in general.
\section{Additivity, Homotopy Fibers, And Special Notation}
\label{sec:additivity-hofib-notation}

In this section, we make precise what we mean by an $n$-additive functor. 

Throughout this chapter, we will mainly be
concerned with cubes made up of coproducts of spaces $X_\alpha$. Let
$T$ be a set, let $\Set{X_\alpha}_{\alpha\in T}$ be a collection of
spaces, 
and define the $T$-cube $\cube{X}$ by
\begin{equation*}
  \cube{X}^{\Set{X_\alpha}}_T (U) = \bigvee_{\alpha\in T-U} X_\alpha  
\end{equation*}
with the inclusion $i: U \hookrightarrow V$ inducing the map $\cube{X}(i)$ given by 
\begin{equation*}
  \cube{X}(i)(X_\alpha) = 
  \begin{cases}
    \ast & \alpha\in V \\
    X_\alpha & \alpha\notin V
  \end{cases}
\end{equation*}
When all of the $X_\alpha$ are the same space $X$, we will write
$\cube{X}^X_T$ for the cube. 
For notational convenience, we will immediately suppress the
dependence of $\cube{X}$ on $X_\alpha$ and $T$ unless it is not clear from
context. Notice that such a cube $\cube{X}$ is equivalent to strongly
co-Cartesian cube, since one could include $X_\alpha$ in the cone over
$X_\alpha$ instead of collapsing it to a single point. The equivalent
cube would 
then be:
\begin{equation*}
  \cube{X}^{\Set{X_\alpha}}_T (U) = \bigvee_{\alpha\in T-U} X_\alpha  \vee
  \bigvee_{\alpha\in U} CX_\alpha .
\end{equation*}

\begin{definition}[$n$-additive]
\label{def:n-additive}
A functor $F$ is $n$-additive if the $(n+1)$-cube $F \cube{X}^X_{n+1}$ 
is Cartesian for all spaces $X$.
\end{definition}

\begin{remark}
We do not require that $F\cube{X}^{\Set{X_\alpha}}_T$ be Cartesian for
arbitrary collections of spaces $X_\alpha$; only those with all
$X_\alpha$ the same space $X$.
\end{remark}

We can rephrase the definition of $n$-additivity as follows: for all
$X$, an $n$-additive functor $F$ gives an equivalence 
$$ F \cube{X}^X_{n+1}(\emptyset) \xrightarrow{\simeq}
\holim_{U\in\Power_0(S)} F \cube{X}^X_{n+1}(U) . $$
Our approach to $n$-additivity will be to break the problem of
understanding this map into two parts: we show when the (homotopy)
fiber of this 
map is contractible, and understand some general conditions under which the
map is surjective on $\pi_0$. These two parts combined allow us to
understand when $F\cube{X}^X_{n+1}$ is Cartesian.
We use the term \emph{homotopy fiber} of a cube to describe
the homotopy fiber of a map like the one above. 

\begin{definition}[Homotopy fiber]
\label{def:homotopy-fiber}
(\cite[1.1]{Cal2}) Let $\cube{X}$ be an $S$-cube of pointed spaces, and
for $T\subset S$, define the topological cube $I^T$ to be the product
of $T$ copies of the unit interval $I=[0,1]$. (When $T=\emptyset$,
this is interpreted as $I^\emptyset = \Set{0}$.)
A point $\Phi \in \hofib \cube{X}$ is a
collection of continuous maps $\Phi_T: I^T \rightarrow \cube{X}(T)$,
one for each subset $T\subset S$, satisfying the two conditions below.
\begin{enumerate}
\item $\Phi$ is natural with respect to $T$. That is, for $U\subset
  T\subset S$, the following diagram commutes:
$$\xymatrix{ I^U
\ar[r]
\ar[d]^{\Phi_U}
&
I^T
\ar[d]^{\Phi_T}
\\
\cube{X}(U)
\ar[r]
&
\cube{X}(T)
}$$
where the upper arrow is the map that takes a function $U\rightarrow
I$ and extends it to a function $T \rightarrow I$ by making it zero on
$T-U$.

\item For each $T\subset S$, the function $\Phi_T$ takes the set of
  points with at least one coordinate having value one,
  $(I^T)_1 := \Set{u\in I^T : \exists_{s\in T} u_s = 1}$, to the
  basepoint in $\cube{X}(T)$.
\end{enumerate}
\end{definition}

This definition of the homotopy fiber of a cube is homeomorphic to
defining the homotopy fiber of an $S$-cube $\cube{X}$ to be the
homotopy fiber of the map 
$ \cube{X}(\emptyset) \rightarrow 
\holim_{U\in\Power_0(S)} \cube{X}(U).$
It also agrees with the construction of the homotopy fiber given
inductively by repeatedly taking homotopy fibers of the structure maps
in a single direction.

\section{Cross Effects}
\label{sec:cross-effects}

The $n$-th cross effect of a functor is a functorial comparison of
$F(\bigvee^n X)$ with $F$ on lower order coproducts of $X$. By taking the
homotopy fiber of maps to smaller coproducts of $X$, the cross effect
``kills off'' their contribution to $F(\bigvee^n X)$, leaving only the
part that does not ``come from'' lower order coproducts of $X$.

\begin{definition}[$cr_n$, $\Perp_n$]
\label{def:crn-perp}
Define the $n^{\text{th}}$ cross effect of a
functor $F$ to be the functor of $n$ variables
\begin{equation*}
  cr_nF(X_1,\ldots,X_n) = \hofib_{U\in\Power(\mathbf{n})}
  F\cube{X}^{\Set{X_i}}_{\mathbf{n}}(U) ,
\end{equation*}
where $\mathbf{n}$ denotes the set $\Set{1,\ldots,n}$.
In this notation, the subscripts of the $X_\alpha$ on the left correspond
to the $X_\alpha$ in $\cube{X}$ on the right.

Let $\Perp_n F(X) = cr_nF(X,\ldots,X)$ be the diagonal of the
$n^{\text{th}}$ cross effect of $F$ evaluated at $X$. Denote the
iteration of this functor by $\Perp_n^{(a)}F$. We will later show that
$\Perp_n$ is part of a cotriple.

\emph{Abuse of notation}: At some points in 
Section~\ref{sec:perp-F-zero}, we need to discuss
$\Perp_n F$ as a functor of $n$ variables. At those points, we
will write $\Perp_n F(X, \ldots, X)$, understanding that this is
the same as $cr_n F(X, \ldots, X)$, and hope that this causes no
confusion. 
\end{definition}

Actually, the vanishing of the cross effect $cr_n$ for all choices of inputs
is equivalent to the vanishing of $\Perp_n$.
\begin{lemma}
The functor $cr_n F(X_1,\ldots,X_n)$ is contractible for all choices
of inputs $X_i$ if and only if $\Perp_n F(X)$ is contractible for all
$X$.
\end{lemma}
\begin{proof}
Since $\Perp_n F(X) = cr_n F(X,\ldots,X)$, one implication is trivial.
Now suppose $\Perp_n F(X)$ is contractible for all $X$. Given
$\Set{X_i}_{i=1}^{i=n}$, let $X = \bigvee X_i$. Let $i_j: X_i
\rightarrow X$ denote the inclusion of the $j^{\text{th}}$ factor, and
let $p_j$ denote the projection onto the $j^{\text{th}}$ factor. 
The map 
$$
cr_n F(p_1,\ldots,p_n): 
\Perp_n F(X)
=
cr_n F(X,\ldots,X)
\rightarrow
cr_n F(X_1,\ldots,X_n)
$$
has a section $cr_n F(i_1,\ldots,i_n)$, so if $\Perp_n F(X) \simeq \basept$,
then $cr_n F(X_1,\ldots,X_n)\simeq \basept$ (\emph{e.g.}, because
$\pi_*(Id)$ factors through $0$).
\end{proof}

Since homotopy inverse limits commute, $\Perp_n^{(a)} F(X)$ may be
computed by 
\begin{equation*}
  \Perp_n^{(a)}F = \hofib_{U\in\Power(\mathbf{n})^a} F\cube{X}(U) .
\end{equation*}

\begin{example}
\label{example:perp2-cube}
  The second iterated cross effect, $\Perp_n^2 F(X)$, is
  \begin{align*}
    \Perp_n^{(2)} F(X)
    &=
    \hofib_{V_2 \in\Power(\mathbf{n})} \left( \Perp_n F \right) 
  \cube{X}^X_{\mathbf{n}}(V_2)
    \\
    &=
    \hofib_{V_2 \in \Power(\mathbf{n})}
    \left( 
      \hofib_{V_1 \in \Power(\mathbf{n})}
      F
      \cube{X}^{\cube{X}^X_{\mathbf{n}}(V_2)} (V_1)
    \right)
    \\
    &=
    \hofib_{\substack{V_1\in\Power(\mathbf{n}) \\ V_2 \in \Power(\mathbf{n})}} 
    F 
    \cube{X}^{\cube{X}^X_{\mathbf{n}}(V_2)}_{\mathbf{n}}(V_1)
  \end{align*}
  To decode the cube $\cube{X}$ that appears, recall that the
  superscript denotes the space from which the coproducts are formed,
  so we have:
  \begin{align*}
    \cube{X}^{\cube{X}^X_\mathbf{n}(V_2)}_\mathbf{n}(V_1)
    &=
    \bigvee_{v_1\not\in V_1} \cube{X}^X_\mathbf{n}(V_2)
    \\
    &=
    \bigvee_{v_1\not\in V_1}
    \bigvee_{v_2\not\in V_2}
    X
  \end{align*}
  From this example, the general form of the cubes
  used to compute $\Perp^{(a)}$ for $a>2$ should be clear.
\end{example}

In order to work with cross effects, we need to establish certain
basic properties. One of the most fundamental is that all cross
effects can be built up by iterating the second cross effect. 

In the next few results, we consider cross-effect cubes as functorial in the
spaces that are used to create them, and write $\cube{X}[X_1,\ldots]$
to indicate the cube $\cube{X}$ built using the spaces $X_1$, \emph{etc}.

Recall that
the $\mathbf{n}$-cube defining $cr_n F(Y_1,\ldots,Y_n)$ is given
by 
$$\cube{Y}[Y_1,\ldots,Y_n](U)
= 
F \left( \bigvee_{b \not\in U} Y_b \right) .$$
Consider the result of applying the functor $cr_2(-)(X_1,X_2)$ in
the first variable of this functor. That gives the $\mathbf{2}$-cube
of $\mathbf{n}$-cubes:
\begin{equation*}
    \cube{Z}[X_1,X_2](V)
    = \cube{Y}(U)[\bigvee_{v \not\in V} X_v,Y_2,\ldots,Y_n] .
\end{equation*}
Now for brevity, let $S = \mathbf{n}\amalg\mathbf{2}$, and let $T = S -
\Set{1}\amalg\emptyset$.
The cube $\cube{Z}$ can be written as an $S$-cube by 
defining $X_{k+1} = Y_k$, for
$k=2,\ldots, n$, so the cube is:
\begin{equation*} 
\cube{W}[X_1,\ldots,X_{n+1}](U\amalg V)
=
    \begin{cases}
    F\left( 
    \bigvee_{v\not\in V} X_v
    \vee
    \bigvee_{u\not\in U\cup\Set{1}} X_{u+1}
    \right)
    &
    \text{$1\not\in U$}
    \\
    F\left( 
    \bigvee_{u\not\in U} X_{u+1}
    \right)
    &
    \text{$1 \in U$}
    \end{cases}
\end{equation*}
We will immediately suppress writing the $[X_1,\ldots,X_{n+1}]$ except
where the spaces $X_i$ are relevant.

There are two things to note. 
First, the cube used to compute
$cr_{n+1}F(X_1,\ldots,X_{n+1})$ is exactly the $(n+1)$-cube 
$$\left\{ A \mapsto \cube{W}[X_1,\ldots,X_{n+1}](A) :
  A \subset T \right\} ,$$
which Goodwillie denotes $\partial^T \cube{W}[X_1,\ldots,X_{n+1}]$.
Second, when $1\in U$, the sub-cube
$\cube{W}(U\amalg -)$ is a constant cube, so the other $(n+1)$-cube,
$\partial_{\Set{1}\amalg\emptyset} \cube{W}$,
that makes up $\cube{W}$ consists of a cube of constant $2$-cubes.

\begin{lemma}
\label{lem:cr-n+1-cr-2-cr-n}
If the $(n+2)$-cube used to compute the second cross effect of $cr_n
F(Y_1,\ldots,Y_n)$ in a single variable, \emph{e.g.}, $Y_1$, is
Cartesian, then the $(n+1)$-cube 
defining $cr_{n+1} F(X_1,\ldots,X_{n+1})$ is Cartesian.
\end{lemma}
\begin{proof}
As discussed above, the cube $\cube{W}[X_1,\ldots,X_{n+1}]$ is the
$(n+2)$-cube that is used to compute the 
second cross effect with respect to $X_1$ and $X_2$ of the functor
$cr_n F(-,X_3,\ldots,X_{n+1})$, and
$\partial^T\cube{W}[X_1,\ldots,X_{n+1}]$ is the $(n+1)$-cube used to
compute $cr_n F(X_1,\ldots,X_{n+1})$, so 
we need to establish that $\cube{W}$ is Cartesian if and only if
$\partial^T \cube{W}$ is Cartesian. We will do this by showing that
the bottom arrow on the following commutative diagram is an
equivalence, and hence if either one of the two vertical maps is an
equivalence, then so is the other.
$$
\xymatrix{
\cube{W}(\emptyset)
\ar[d] \ar[dr]
&
\\
{ \holim_{\Power_0(S)} \cube{W}}
\ar[r]
&
{\holim_{\Power_0(T)} W}
}$$
We will  compare the homotopy inverse
limits of $\cube{W}$ over $\Power_0(S)$ and $\Power_0(T)$ in two stages. 
Let $\cat{R} = \Power_0(S) - \Set{1}\amalg\emptyset$ be the category
$\Power_0(S)$ without the object $\Set{1}\amalg\emptyset$. The
inclusions $\Power_0{T} \hookrightarrow \cat{R} \hookrightarrow
\Power_0{S}$ induce maps
\begin{equation}
\label{eq:holim-cr2-crn}
\holim_{\Power_0(S)} \cube{W}
\rightarrow
\holim_{\cat{R}} \cube{W}
\rightarrow
\holim_{\Power_0(T)} \cube{W}
.
\end{equation}
We will show that both of these maps are equivalences.
The right hand map in \eqref{eq:holim-cr2-crn} is an equivalence by 
\cite[\S{}XI.9,
Theorem~9.2]{Bousfield-Kan:homotopy-limits-completions-and-localizations}
because $\Power_0{T}$ is left cofinal in $\cat{R}$. To verify left
cofinality, let $U\amalg V \in\Obj(\cat{R})$ (that is, $\Set{1}\amalg
\emptyset \not= U \amalg V \subset S$).  The set $(U-\Set{1}) \amalg V$
is in $\Power_0(T)$ (because the restricton on $U\amalg V$ guarantees
that this is not $\emptyset \amalg\emptyset$), 
so the category $(\Power_0(T) \rightarrow
\cat{R})/U\amalg V$ contains the object $(U-\Set{1}) \amalg V$ with the
inclusion map $(U-\Set{1}) \amalg V \rightarrow U \amalg V$. Hence
this category is nonempty. 
Morphisms in $\Power_0(-)$ are inclusions of subsets, 
so there is at most one morphism between any two
objects; hence $U\amalg V \rightarrow U\amalg V$ is the terminal
object in this category, and it is contractible.
The left hand map in \eqref{eq:holim-cr2-crn} is an equivalence for
reasons particular to the cube $\cube{W}$, as we will now show.
Recall that a homotopy inverse limit of a functor $\cube{W}$ over a
category $\cat{D}$ is the space of maps $\Map_{\cat{D}}
(\realization{\cat{D}/-}, \cube{W}(-))$. Let $i: U\rightarrow V$ be a
morphism in $\cat{D}$, and let $\phi\in \Map_{\cat{D}}
(\realization{\cat{D}/-}, \cube{W}(-))$. Then the components $\phi_U$
and $\phi_V$ cause the following diagram to commute:
$$\xymatrix{
\realization{\cat{D}/U}
\ar[r]^{\phi_U}
\ar[d]^{\realization{\cat{D}/i}}
&
\cube{W}(U)
\ar[d]^{\cube{W}(i)}
\\
\realization{\cat{D}/V}
\ar[r]^{\phi_V}
&
\cube{W}(V)
}$$
Consider the case when $U=\Set{1}\amalg\emptyset$ and $V =
\Set{1}\amalg\Set{1}$. 
Recall that the $\mathbf{2}$-cube $\cube{W}(\Set{1}\amalg -)$ is
constant, so in particular for the inclusion $i: \Set{1}\amalg\emptyset
\hookrightarrow \Set{1}\amalg\Set{1}$, the map $\cube{W}(i)$ is the
identity.
Then the commutative square above shows that $\phi_{U} =
\cube{W}(i)^{-1} \circ \phi_V \circ \realization{\cat{D}/i}$. That is,
given a $\phi_{\Set{1}\amalg\Set{1}}$, there is a unique
$\phi_{\Set{1}\amalg\emptyset}$ that corresponds to it. This means
that the restriction map from $\holim_{\Power_0(S)}$ to
$\holim_{\cat{R}}$ is an isomorphism since the only map that is in the
former that is not in the latter is $\phi_{\Set{1}\amalg\emptyset}$.
\end{proof}

\begin{corollary}
\label{cor:crn-from-cr2}
For $n\ge 2$, the $(n+1)^{\text{st}}$ cross effect $cr_{n+1} F(X_1, \ldots, X_{n+1})$ is equivalent to
the iterated cross effect
$cr_2 (cr_n F(X_1, \ldots, X_{n-1}, -))(X_{n},X_{n+1})$.
\end{corollary}
\begin{proof}
As noted prior to Lemma~\ref{lem:cr-n+1-cr-2-cr-n}, the $(n+2)$-cube
$\cube{W}$ that computes the iterated cross effect can be written as a
$1$-cube of $(n+1)$-cubes: 
$$\partial^T \cube{W} 
\rightarrow 
\partial_{\Set{1}\amalg \emptyset} \cube{W}.$$
The cube $\partial^T \cube{W}$ is exactly the cube used to define
the $(n+1)^{\text{st}}$ cross-effect of $F$, so 
$$\hofib \partial^T \cube{W} = cr_{n+1} F(X_1,\ldots,X_{n+1}) ,$$
and the cube 
$\partial_{\Set{1}\amalg \emptyset} \cube{W}$ is a cube of constant
$2$-cubes, so
$$\hofib \partial_{\Set{1}\amalg \emptyset} \cube{W} \simeq \basept.$$
Computing $\hofib \cube{W}
= cr_2 (cr_n F(-,X_3,\ldots,X_{n+1}))(X_{1},X_{2})$ 
by taking the homotopy fiber of these
homotopy fibers gives us a natural map 
$$cr_2 (cr_n F(-,X_3,\ldots,X_{n+1}))(X_{1},X_{2})
\xrightarrow{\simeq} cr_{n+1} F(X_1,\ldots,X_{n+1}).$$
\end{proof}

\begin{corollary}
  \label{cor:perp-n+1-perp-2-perp-n}
  If the $(n+1)$-cube defining $cr_{n+1} F(X_1,\ldots,X_{n+1})$ is
  Cartesian, then
  taking the second cross effect with respect to any single $X_i$ of
  the functor $cr_{n} F(X_1,\ldots,X_n)$
  results in a Cartesian $2$-cube. 
\end{corollary}
\begin{proof}
  Apply Lemma~\ref{lem:big-cube-Cartesian-fiber-cube-Cartesian} to the
  result of the preceding Lemma~\ref{lem:cr-n+1-cr-2-cr-n} to
  reduce from an $(n+2)$-cube to a
  $2$-cube by taking fibers to compute the space $cr_n$ from the
  $n$-cube defining it.
\end{proof}

Another important feature of cross effects is that the cubes from
which they are built have section maps to every structure map in the
cube. This means that they are much nicer than arbitrary cubes; in
particular, the homotopy group $\pi_k$ of the total fiber can be
computed from $\pi_k$ of the vertices of the cube. 

\begin{hypothesis}[Compatible sections to structure maps.]
\label{hypothesis:compatible-sections}
We say that a $T$-cube $\cube{X}$ has \emph{compatible sections to all
structure maps} if for each inclusion of subsets  $i_{U,V}: U
\hookrightarrow V$, there exists a section map $s_{V,U}: \cube{X}(V)
\rightarrow \cube{X}(U)$, and furthermore these section maps compose
so that $s_{W,V} \circ s_{V,U} = s_{W,U}$.
\end{hypothesis}

\begin{lemma}
\label{lem:crn-cube-has-sections}
The cubes $\cube{X}_T^{\Set{X_i}}$ used to construct the cross effects
satisfy the compatible sections
hypothesis (\ref{hypothesis:compatible-sections}).
\end{lemma}
\begin{proof}
Explicitly, given $U\subset V$ and the induced
projection $\bigvee_{u \notin U} X_u \rightarrow \bigvee_{v \notin V}
X_v$, has a section map that is the identity on each $X_v$ for $v
\notin V$ (by hypothesis, $U \subset V$, so if $v$ is not in $V$, then
$v$ is also not in $U$).
It is easy to see that these are all compatible in the sense
of~\ref{hypothesis:compatible-sections}. 
\end{proof}

\begin{lemma}
\label{lem:pi-k-perp-is-perp-pi-k}
The group (or set) $\pi_k \Perp F(X)$ is isomorphic to the iterated
fiber of the cube of groups (or sets) $\fib \pi_k F\cube{X}$.
Extending our definition of $\Perp$ to functors to groups or sets, this
can be restated as: $\pi_k \Perp F(X) = \Perp\pi_k F(X)$; that is,
$\Perp$ commutes with $\pi_k$.
\end{lemma}
\begin{proof}
Recall that $\Perp F(X)$ is the homotopy fiber of a cube $F\cube{X}$
that has compatible sections to all structure maps.
Also, recall that the total homotopy fiber of a cube can be computed
by iterating the process of taking fibers in one direction at a time. 

The existence of sections means that the long exact sequences of the
fibrations in one direction involved actually break up into short exact
sequences for each $\pi_k$. This means that $\pi_k$ of the homotopy
fiber of each structure map $\cube{X}(i_{U,V})$ is the fiber of the map $\pi_k
\cube{X}(i_{U,V})$.  The fact that the sections are compatible means that they
pass to sections on the fibers, so this argument can be iterated until
the total fiber is reached.
\end{proof}


\section{Cotriples}
\label{sec:cotriples}

(The introduction to cotriples in this section follows Weibel
\cite[Chapter~8.6]{Weibel:homological-algebra}.)

In homological algebra, one frequently uses the technique of forming a
free resolution of an $R$-module $M$. The canonical functorial way of
doing this is to begin by applying the free functor $F$ to set of
elements of the module $M$, producing $R[M]$ whose elements are formal
sums $\Sigma r_i m_i$, and then mapping that to $M$ by sending the
formal element $r_i m_i$ to the element $r_i \cdot m_i$ given by
letting $r_i$ act on $m_i$. Iterating this construction produces an
acyclic (in dimension $>0$) chain of free $R$-modules, and hence a
free resolution of $M$. Another way of looking at this
construction is as result of iteratively applying the functor $\Perp =
FU$, the composition of the adjoint pair consisting of the 
free $R$-module functor and the forgetful
functor from $R$-modules to sets. An axiomatization of this
approach creates the objects called ``cotriples''. 

The intent of a cotriple is to create a simplicial object that
functions as a resolution of $X$. The simplicial object
$RX = \left( [n] \mapsto \Perp^{n+1} X \right)$ is equipped with a
natural map $RX\rightarrow X$ (derived from $\Perp \rightarrow
\text{Id}$) that associates the resolution $RX$ to the object $X$. (In
general, $RX$ is acyclic in positive degrees and $\pi_0 RX = X$.) In
particular, if $X = \Perp Y$ (\emph{e.g.}, $X$ is already a free
module), then this complex is homotopic to the constant simplicial
object $\Perp Y$, so iterating the construction of these resolutions
is idempotent up to homotopy.

Precisely speaking, a cotriple is a functor $\Perp$ equipped with
natural transformations $\delta: \Perp \rightarrow \Perp^2$ and
$\epsilon: \Perp \rightarrow
\text{Id}$ such that the following diagrams commute:
\begin{eqnarray*}
\xymatrix{
\Perp 
\ar[r]^{\delta} 
\ar[d]^{\delta}
&
\Perp^2 
\ar[d]^{\Perp\delta} 
\\
\Perp^2 
\ar[r]^{\delta\Perp} 
&
\Perp^3
}
&
\xymatrix{
&
\Perp
\ar[dr]^{=}
\ar[dl]^{=}
\ar[d]^{\delta}
\\
\Perp
&
\Perp^2
\ar[l]^{\epsilon \Perp}
\ar[r]_{\Perp \epsilon}
&
\Perp
}
\end{eqnarray*}
The notation of Section~\ref{sec:cross-effects} is not a coincidence;
the cross effect $\Perp_n$ is in fact a cotriple, with augmentation
map $\epsilon$ induced by the fold map $\bigvee^n X \to X$, and the
diagonal map $\delta$ induced by the diagonal inclusion of $\mathbf{n}$ into
$\mathbf{n}\times\mathbf{n}$. 
The proof is somewhat technical, so we illustrate the idea in
Section~\ref{sec:crn-cotriple-sketch} and prove it in
Section~\ref{sec:crn-cotriple-proof}. 

%
%
\section{Illustration: The Cross-Effects Form A Cotriple}
\label{sec:crn-cotriple-sketch}

This section contains an illustration of the idea of a proof that the
functor $\Perp$ is a cotriple. The purpose of this section is to
provide a plausible motivation for the somewhat technical proof
contained in Section~\ref{sec:crn-cotriple-proof}. From this
illustration, the reader can see that 
There are very few ingredients needed to prove that $\Perp$ is a
cotriple; this section gives the reader an idea what they are and how
they could be assembled to form a proof.
We only consider $\Perp=\Perp_2$ in
this section.


Recall that a cotriple requires two commuting diagrams:
\begin{equation}
  \label{eq:diagram-for-faces}
\xymatrix{
\Perp\Perp 
\ar[r]^{\Perp\epsilon}
\ar[d]^{\epsilon\Perp}
&
\Perp
\ar[d]^{\epsilon}
\\
\Perp
\ar[r]^{\epsilon}
&
1}
\end{equation}
and
\begin{equation}
  \label{eq:diagram-for-degen}
\xymatrix{
&
\Perp
\ar[dl]^{=}
\ar[d]^{\delta}
\ar[dr]^{=}
&
\\
\Perp
&
\Perp\Perp
\ar[l]^{\Perp\epsilon}
\ar[r]^{\epsilon\Perp}
&
\Perp
}  
\end{equation}

\subsection{The Notation}

This section uses some nonstandard, but very visually intuitive,
notation. To understand the functors $\Perp$ and $\Perp\Perp$, we will
consider them as the total fibers of $2$- and $4$-dimensional cubes,
respectively. (Remember that we are only working with $\Perp=\Perp_2$
to keep the argument understandable to the reader.)

The space $\Perp F(X)$ is the total homotopy fiber of the cube 
$$\xymatrix{
F(X \vee X) 
\ar[r]
\ar[d]
& 
F(X)
\ar[d]
\\
F(X)
\ar[r]
&
F(0)
}$$

The argument we make is essentially independent of $F$,
so we will omit the application of $F$ to our cubes.  (There is one
exception to the assertion that $F$ does not matter: at some points we
need to consider $0$ instead of $F(0)$.) That leaves us with the
cubes:
$$\xymatrix{
X \vee X
\ar[r]
\ar[d]
& 
X
\ar[d]
\\
X
\ar[r]
&
0
}$$

We will write subscripts on the spaces $X$ to distinguish them. This
has the effect of making it clear what the maps are: they are the
identity on $X_i$ and the zero map between spaces without the same
subscript. 
$$\xymatrix{
X_1 \vee X_2
\ar[r]
\ar[d]
& 
X_1
\ar[d]
\\
X_2
\ar[r]
&
0
}$$

In order to make it possible to write four dimensional cubes, we engage
in one more reduction of structure; we also omit the arrows entirely,
writing the cubes in the form of matrices:
$$
\begin{pmatrix}
  \begin{pmatrix}
    X_1 & X_2
  \end{pmatrix}
  &
  X_1
  \\
  X_2
  &
  0
\end{pmatrix}
$$

When we write four dimensional cubes  for $\Perp\Perp$, we will doubly
index the spaces $X$ as $X_{i,j}$.
Our convention for the meaning of the indices in $X_{i,j}$ is that the
first index, $i$, corresponds to the first application of the functor
$\Perp$ (that is, the rightmost $\Perp$). The second index, $j$,
corresponds to the second application of $\Perp$ (the left one).

\subsection{Explicit Models For $\Perp_2$ And $\Perp_2 \Perp_2$}

In order to write down the maps $\delta$,  $\epsilon\Perp$, and
$\Perp\epsilon$ explicitly, we will use explicit models for the cross
effects.  Recall that using Goodwillie's model for the total fiber,
given a cube of cubes, taking homotopy fibers twice commutes up to
natural homeomorphism. We will use this to blow up models for $\Perp
F(X)$.

As above, recall that in our notation, $\Perp F$ is the homotopy fiber
of $F$ applied to the cube:
$$
\begin{pmatrix}
  \begin{pmatrix}
    X_1 & X_2
  \end{pmatrix}
  &
  X_1
  \\
  X_2
  &
  0
\end{pmatrix}
$$

Note that the total fiber of the cube above is homeomorphic to the
total fiber of the following $4$-cube ($2$-cube of $2$-cubes): 
$$
\xymatrix{
{
\begin{pmatrix}
  \begin{pmatrix}
    X_1 & X_2
  \end{pmatrix}
  &
  X_1
  \\
  X_2
  &
  0
\end{pmatrix}
}
&
{
\begin{pmatrix}
  0&0 \\ 0&0
\end{pmatrix}
}
\\
{
\begin{pmatrix}
  0&0 \\ 0&0
\end{pmatrix}
}
&
{
\begin{pmatrix}
  0&0 \\ 0&0
\end{pmatrix}
}
}
$$
We can enlarge this to the following cube, which we
will refer to as computing $\widetilde{\Perp_2}$, by expanding some of
the sub-cubes but maintaining the property that the total fiber of all
of the cubes except that in the upper left is contractible.
$$
\xymatrix{
{
\begin{pmatrix}
  \begin{pmatrix}
    X_1 & X_2
  \end{pmatrix}
  &
  X_1
  \\
  X_2
  &
  0
\end{pmatrix}
}
&
{
\begin{pmatrix}
  X_1 & X_1 \\
  0 & 0
\end{pmatrix}
}
\\
{
\begin{pmatrix}
  X_2 & 0 \\
  X_2 & 0
\end{pmatrix}
}
&
{
\begin{pmatrix}
  0&0 \\ 0&0
\end{pmatrix}
}
}
$$
There is a map $\widetilde{\Perp}F(X) \rightarrow \Perp F(X)$ induced
by sending the vertices of all cubes except that in the upper right to
zero. (Strictly speaking, apply $F$ first, then map to $0$. This
causes the homotopy fibers of those cubes to have exactly one point,
so the total fiber is homeomorphic to $\Perp F(X)$.)

Given a $T$-cube $\cube{X}$ and a function $f: S\rightarrow T$, there
is an induced functor $\Power(f): \Power(S) \rightarrow \Power(T)$, and then
this induces a map
$$ \hofib_{\Power(T)} \cube{X} \rightarrow 
\hofib_{\Power(S)} \Power(f)^* \cube{X}.$$
This can be used to produce a map from the homotopy fiber of the
two-dimensional cube for  $\Perp F(X)$ to the homotopy fiber of the
four-dimensional cube for $\widetilde{\Perp} F(X)$; see
Section~\ref{sec:homotopy-fibers} for details. 

The functor $\Perp\Perp$ is computed by applying $\Perp$ twice; this
naturally corresponds to the total fiber of the following $4$-cube
($2$-cube of $2$-cubes):
\begin{equation}
\label{eq:perp-perp}
\xymatrix{
{
\begin{pmatrix}
  \begin{pmatrix}
    X_{11} & X_{21} \\
    X_{12} & X_{22}
  \end{pmatrix}
  &
  \begin{pmatrix}
    X_{11} & X_{21}
  \end{pmatrix}
  \\
  \begin{pmatrix}
    X_{12} & X_{22}
  \end{pmatrix}
  &
  0
\end{pmatrix}
}
&
{
\begin{pmatrix}
  \begin{pmatrix}
    X_{11}  \\
    X_{12} 
  \end{pmatrix}
  &
  \begin{pmatrix}
    X_{11} 
  \end{pmatrix}
  \\
  \begin{pmatrix}
    X_{12} 
  \end{pmatrix}
  &
  0
\end{pmatrix}
}
\\
{
\begin{pmatrix}
  \begin{pmatrix}
    X_{21}  &
    X_{22} 
  \end{pmatrix}
  &
  \begin{pmatrix}
    X_{21} 
  \end{pmatrix}
  \\
  \begin{pmatrix}
    X_{22} 
  \end{pmatrix}
  &
  0
\end{pmatrix}
}
&
{
\begin{pmatrix}
  0 & 0 \\
  0 & 0
\end{pmatrix}
}}
\end{equation}
The diagonal map $\delta: \Perp \rightarrow \Perp\Perp$ is induced by
the composition of the map $\Perp \rightarrow \widetilde{\Perp}$ with
the map from $\widetilde{\Perp}$ to  $\Perp \Perp$ is induced by
sending $X_1$ to $X_{11}$ and $X_2$ to $X_{22}$.

There are two maps $\Perp\Perp \rightarrow \Perp$. Recall that our
convention for the meaning of the indices in $X_{i,j}$ is that the
first index, $i$, corresponds to the first application of the functor
$\Perp$ (that is, the rightmost $\Perp$), and the second index, $j$,
corresponds to the second application of $\Perp$ (here, the left
one). 

With this convention, the map $\Perp\epsilon: \Perp\Perp
\rightarrow \Perp$ is induced by
sending $X_{m,n}$ to $X_{m}$. To remind the reader of which spaces map
to which, we will write the image as $X_{m,*}$ before identifying
$X_{m,*}$ with $X_m$ in the cube defining the single application of
$\Perp$. Formulating this in terms of cubical diagrams,
$\Perp\epsilon$ is induced by the map of $\Perp\Perp$
\eqref{eq:perp-perp} to the following $4$-cube, followed by taking
the total fiber (which is now easy to see is homeomorphic to $\Perp$).
$$\xymatrix{
{
  \begin{pmatrix}
    \begin{pmatrix}
      X_{1,*} & X_{2,*}
    \end{pmatrix}
    &
    0
    \\
    0 & 0
  \end{pmatrix}
}
  &
{
  \begin{pmatrix}
    X_{1,*} & 0 \\
    0 & 0 
  \end{pmatrix}
}
  \\
{
  \begin{pmatrix}
    X_{2,*} & 0 \\
    0 & 0
  \end{pmatrix}
}
  &
{
  \begin{pmatrix}
    0 & 0 \\
    0 & 0
  \end{pmatrix}
}
}
$$
Although we have not illustrated this fact in this section, 
the map $\Perp \rightarrow \widetilde{\Perp}$ is a section to
the map induced by the zero map on the spaces $X_i$ not in the upper
left sub-cube, so the composition  $(\Perp\epsilon)\delta = 1$, and hence the
left hand triangle in \eqref{eq:diagram-for-degen} commutes.

The map $\epsilon\Perp$ is similar. It is induced by mapping the
$4$-cube for $\Perp\Perp$ to the following:

$$\xymatrix{
{
  \begin{pmatrix}
    \begin{pmatrix}
      X_{*,1} \\ X_{*,2}
    \end{pmatrix}
    &
    X_{*,1}
    \\
    X_{*,2} & 0
  \end{pmatrix}
}
  &
{
  \begin{pmatrix}
    0 & 0 \\
    0 & 0
  \end{pmatrix}
}
  \\
{
  \begin{pmatrix}
    0 & 0 \\
    0 & 0
  \end{pmatrix}
}
  &
{
  \begin{pmatrix}
    0 & 0 \\
    0 & 0
  \end{pmatrix}
}
}
$$
As with the case of $\Perp\epsilon$, we have the composition
$(\Perp\epsilon) \delta = 1$, so the right triangle in
\eqref{eq:diagram-for-degen} commutes.
Finally, we verify that the square in
Equation~\eqref{eq:diagram-for-faces} commutes. The identification of
the image of $\Perp\epsilon$ with $\Perp$ comes from identifying
$X_{1,*}$ with $X_1$ and $X_{2,*}$ with $X_2$. The identification of
the image of $\epsilon\Perp$ with $\Perp$ comes from identifying
$X_{*,1}$ with $X_1$ and $X_{*,2}$ with $X_2$. This means, for
instance, that the space $X_{1,2}$ is identified with $X_{1,*}$ under
the map $\Perp\epsilon$, but identified with $X_{*,2}$ under
$\epsilon\Perp$, so these two maps are not the same as maps from
$\Perp\Perp$ to $\Perp$.
The map $\epsilon$ is induced by mapping
$$\begin{pmatrix}
  \begin{pmatrix}
    X_1 & X_2
  \end{pmatrix}
&
X_1
\\
X_2 
&
0
\end{pmatrix}
$$
to
$$
\begin{pmatrix}
X
&
0
\\
0
&
0
\end{pmatrix}
$$
Under this map, all $X_{m,n}$ in the upper left corner are identified,
equalizing the images of $\epsilon\Perp$ and $\Perp\epsilon$,
so the diagram in Equation~\eqref{eq:diagram-for-faces} commutes.

\clearpage

\section{Proof: The Cross-Effects Form A Cotriple}
\label{sec:crn-cotriple-proof}

In this section we produce a formal proof that $\Perp$ is a
cotriple. We require only a good model for the homotopy fiber of a
cube, such as the one 
given in Definition~\ref{def:homotopy-fiber}, that is functorial in
the indexing category.

We introduce the machinery of ``free cubes'' in order to have a good
method of being precise about certain maps. Actually, the ``free
cubes'' that we will use are like free modules over a ring.  The
``ring'' $U$ determines the order of the cross effect $\Perp_U$ being
used.  A $(k\times U)$-cube (meaning the indexing category is the
disjoint union of $k$ copies of $U$) is the analogy of a rank $k$ free
module over the ``ring'' $U$.  The rank $k$ determines the number of
iterations $\Perp_U^k$.

\subsection{Free Cubes}

We will begin by defining a ``free cube'' with a given ``generating
function''. This requires that we build up a bit of notation. 

\begin{notation}
%
  Given a set $U$ and a subset $A$ of $U$, let $A^c$ denote the complement of
  $A$ in $U$.
\end{notation}


%

We define a ``diagonal'' to encode the information needed to
construct a cube of coproducts and inclusion and projection maps of
the type used to define the cross effect. 

\begin{definition}[Diagonal]
  For any sets $S$ and $U$, define the ``diagonal'' $\Delta(S,U)$ to
  be the subsets of $\Power(S \times U)$ that are complements of a
  singleton in each component. That is, given a function $f:
  S\rightarrow U$, define the set $B_f \subset S \times U$ by
  $$ B_f = \bigcup_{s\in S} (s, f(s)^c) ,$$
  and then define the diagonal $\Delta(S,U)$ to be
$$
\Delta(S,U) 
=
\bigcup_{f: S\rightarrow U} 
\left\{ 
B_f
\right\}
.$$
\end{definition}

\begin{remark}
  When $S$ is the empty set, there is one function $f: \emptyset
  \rightarrow U$, resulting in the empty set as the union over $s \in
  \emptyset = S$ being the only member of $\Delta(\emptyset,U)$.
  Also note that $\Delta(S,U) \cong \Hom(S,U) \cong U^S$ via the
  correspondence $B_f \leftrightarrow f$.
\end{remark}

\begin{example}
Let $\cube{X}$ denote the $T$-cube $\cube{X}^{\Set{X_u}}_T$ 
from Section~\ref{sec:additivity-hofib-notation} that forms a
basis for defining the cross effect. Recall that $\cube{X}(U) =
\bigvee_{u\not\in U} X_u$. The ``digaonal'' sets $V \in
\Delta(\mathbf{1},T)$ are those sets for which $\cube{X}(V)$ consists
only of a single space $X_v$ for some $v\in T$.
\end{example}

\begin{lemma}
The diagonal $\Delta(S,U)$ is contravariantly functorial in $S$ and
covariantly functorial in $U$.
\end{lemma}
\begin{proof}
This is clear from the natural isomorphism $\Delta(S,U) \cong
\Hom(S,U)$, but we will give explicit maps that we will use later.

The diagonal $\Delta(S,U)$ is contravariantly functorial in $S$. Given
a map $g: S \rightarrow T$, we can define a map $\Delta(g,1):
\Delta(T,U) \rightarrow\Delta(S,U)$ sending $B \in \Delta(T,U)$ to
$(g\times 1)^{-1} (B)$. 
$$ \Delta(g,1)(B) = (g \times 1)^{-1}(B)$$
The set $(g\times 1)^{-1} (B)$ is still
the complement of a singleton in each component, because for
each $j\in S$, we have $\Set{j} \times U \cong \Set{g(j)} \times U$.

The diagonal $\Delta(S,U)$ is covariantly functorial in $U$. Given a
function $h: U \rightarrow V$, define $\Delta(1,h)$  by
``pushing the missed singletons along $h$'':
$$\Delta(1,h)(B_f) = B_{h\circ f}$$
\end{proof}

\begin{definition}[Free cube]
  \label{def:free-cube}
  Given sets $U$ and $S$ and a functor $g$ from the discrete category
  $\Delta(S,U)$ to a 
  pointed category with coproducts (for example pointed spaces or
  cubes of pointed spaces), we define
  $\Free(S,U,g)$ to be
  the $(S \times U)$-cube $\cube{X}$
  with vertices
  $$\cube{X}(A) = \bigvee_{\Set{B \in\Delta(S,U) : A \subset B}} g(B)$$
  Morphisms in $\cube{X}$ are induced by the maps $g(B) \rightarrow
  g(B')$ that are the identity if $B=B'$ and the zero map otherwise.
\end{definition}

\begin{remark}
\label{remark:maps-out-of-free-cube}
It is easy to define a map out of a free cube, or between free cubes.
Every point in a free cube is in the image of a section map of one of
the spaces on the diagonal (that is, a space $\cube{X}(B)$ with $B\in
\Delta(S,U)$), so it suffices to give a map from each space on the
diagonal. 
From this, it is clear that $\Free(S,U,g)$ is a covariant functor
with respect to natural transformations of $g$.
\end{remark}

\begin{definition}[Alternative free cube]
  \label{def:free-cube-alt}
  An alternative formulation of Definition~\ref{def:free-cube} follows.
  This form of the definition of a free cube is useful
  because it more closely resembles the definition of the cross effect.
 For $A \subset S\times U$ and for $s\in S$, define
  the projection of $A$ on the the $s$ factor, 
  $p_s(A)$, to be the intersection of $A$  with $\Set{s}\times U$,
  considered as a subset of $U$. Then $A \subset B_f$ if and only if
  $f(s) \not\in p_s(A)$ for all $s\in S$, so we can define
  $\Free(S,U,g)$ to be the cube $\cube{X}$ 
  with vertices
  $$\cube{X}(A) = \bigvee_{\Set{B_f \in\Delta(S,U) : \forall s\in S,
      f(s)\not\in p_s(A)}} g(B_f)$$
\end{definition}

\begin{example}
  The $n^{\text{th}}$ cross effect, $cr_n F(X_1,\dotsc,X_n)$, is the
  total fiber of $F$ applied to a free cube $\cube{X}$. Let
  $\mathbf{n} = \Set{1,\dotsc,n}$, and define $\cube{X} 
  =
  \Free(\mathbf{1},\mathbf{n},g)$, 
  with $g(\Set{i}^c) = X_i$. When $n=2$, we can write this down very
  explicitly. First, $\Delta(\mathbf{1},\mathbf{2}) =
  \Set{\Set{1},\Set{2}}$. Then we can compute:
  \begin{align*}
    \cube{X}(\emptyset) 
    &= 
    \bigvee_{B\in \Delta(\mathbf{1},\mathbf{n})} 
    g(B) 
    \\
    &=
    g(\Set{1}) \vee g(\Set{2})
    \\
    &=
    g(\Set{2}^c) \vee g(\Set{1}^c)
    \\
    &=
    X_2 \vee X_1
  \end{align*}
  Similarly, $\cube{X}(\Set{1}) = X_2$, $\cube{X}(\Set{2}) = X_1$, and
   $\cube{X}(\Set{1,2}) = 0$, so the cube is
  exactly the cross-effect cube we wanted to see.
\end{example}

In general, given a functor $\cube{X}$ defined on a category $\cat{D}$ and
another functor $F: \cat{C}\rightarrow \cat{D}$, we can define the
pullback functor $F^* \cube{X}$ precomposing with $F$. When dealing
with cubical diagrams, a function $f: S\rightarrow T$ induces a
functor $\Power(f): \Power(S) \rightarrow\Power(T)$; in this case,
the pullback operation is $\Power(f)^*$.

Similarly, given sets $n$ and $m$, and 
a function $f:n \rightarrow m$ between them, there is an induced
map of sets $f\times 1: n\times U \rightarrow m \times U$ that can then
be used to define a functor $\Power(f\times 1): \Power(n\times U)
\rightarrow \Power(m\times U)$. In this case, we will denote the
pullback by $\Power(f\times 1)^*$.

We now establish that ``free cubes'' are closed under the pullback
operation.
\begin{lemma}
\label{lem:pullbacks-of-free-cubes}
  Let $m$ and $n$ be sets, let the $(m\times U)$-cube $\cube{X}=
  \Free(m,U,g)$ be a free cube, and 
  let $f: n\rightarrow m$ be a function. The $(n\times U)$-cube
  $\Power(f\times 1)^*
  \cube{X}$ is isomorphic to a free cube $\cube{Y} = \Free(n,U,h)$ with 
  $$h(B) = \bigvee_{B'\in \Delta(f,1)^{-1}(B)} g(B') .$$
\end{lemma}
\begin{proof}
  We consider the question one vertex at a time. Fix a subset 
  $A \subset n\times U$.
  \begin{align*}
    \cube{Y}(A) 
    &=
    \bigvee_{\Set{B \in \Delta(n,U) : A \subset B}}
    h(B)
    \\
    \intertext{Expanding the definition of $h$ gives:}
    &{}
    \bigvee_{\Set{B \in \Delta(n,U) : A \subset B}}
    \bigvee_{\Set{B' \in \Delta(f,1)^{-1}(B)}}
    g(B')
    \\
    \intertext{Interchanging the order of quantifiers and combining
      them turns this into:}
    &{}
    \bigvee_{\Set{B'\in\Delta(m,U) : A \subset \Delta(f,1)(B') }} g(B')
  \end{align*}
  We now show that the indexing set
  $$\Set{B' \in \Delta(m,U) : A \subset \Delta(f,1)(B') }$$
  is the same as
  $$\Set{B' \in \Delta(m,U) : (f\times 1) A \subset B' }.$$
  The latter is the indexing set for $\Power(f\times 1)^*\cube{X}$, so
  this will establish that $\cube{Y} \cong \Power(f\times 1)^* \cube{X}$.
  This is elementary set theory.
  There are two directions to show containment. Recall that
  $\Delta(f,1)(B') = (f\times 1)^{-1} (B')$. 
  If $A \subset \Delta(f,1) B' $, then this means:
  \begin{align*}
    A &\subset (f\times 1)^{-1} B' 
    \\
 (f\times 1)A &\subset (f\times 1)(f \times 1)^{-1} B' \subset B'
 ,
\\
\intertext{where $(f\times 1)(f \times 1)^{-1} B'$ may be smaller than $B'$ if
  some components are not in the image (i.e., if $f$ is not
  surjective). This establishes containment in one direction.
  On the other hand,
  if $(f\times 1) A\subset B'$, then applying $(f\times 1)^{-1}$ gives
}
  A 
  \subset
  (f\times 1)^{-1}(f \times 1) A 
  &\subset
  (f\times 1)^{-1} B' 
  =
  \Delta(f,1)(B')
,
  \end{align*}
  which establishes containment in the other direction.
\end{proof}

\subsection{Homotopy Fibers}
\label{sec:homotopy-fibers}

In this section we briefly recall from Bousfield-Kan \cite[XI,
\S{}9, p.~316]{Bousfield-Kan:homotopy-limits-completions-and-localizations}
the map on homotopy fibers induced by a functor on diagram categories.
The purpose of this section is to prove the following proposition
(which follows immediately from this work of Bousfield
and Kan, as indicated below):
\begin{proposition}
\label{prop:induced-map-on-hofib}
Let $f: S\rightarrow T$ be a map of sets, and let $\cube{X}$ be a
$T$-cube of pointed spaces. Then $f$ induces a natural map
$$\hofib_{\Power(T)} \cube{X} \xrightarrow{} \hofib_{\Power(S)}
\Power(f)^*\cube{X},$$
and if $f$ is surjective, then this map is a homotopy equivalence.
\end{proposition}

Given sets $S$ and $T$ and a function $f: S \rightarrow T$, there is
an induced functor on the power set categories: $\Power(f): \Power(S)
\rightarrow \Power(T)$. Since the inverse image of the empty set is
the empty set, this functor restricts to a functor $\Power_0(f)$ from
$\Power_0(S)$ to $\Power_0(T)$.
Let $\cube{X}$ be a $T$-cube; that is, a
functor whose domain category is $\Power(T)$. One definition of the
homotopy fiber of $\cube{X}$ 
is strict fiber in the following fiber sequence:
$$
\hofib_{\Power(T)} \cube{X}
\rightarrow
\holim_{\Power(T)} \cube{X}
\rightarrow
\holim_{\Power_0(T)} \cube{X}
$$
This shows that in order to produce a map $\hofib \cube{X} \rightarrow
\hofib \Power(f)^* \cube{X}$, it suffices to produce the two right
vertical maps in the following commutative diagram:
$$\xymatrix{
\hofib_{\Power(T)} \cube{X}
\ar[r] \ar[d]
&
\holim_{\Power(T)} \cube{X}
\ar[r] \ar[d]
&
\holim_{\Power_0(T)} \cube{X}
\ar[d]
\\
\hofib_{\Power(S)} \Power(f)^*\cube{X}
\ar[r] 
&
\holim_{\Power(S)} \Power(f)^*\cube{X}
\ar[r]
&
\holim_{\Power_0(S)} \Power(f)^*\cube{X}
}$$

This is the classical situation from Bousfield and Kan. It is just as
easy to describe in the general setting, so we do so.
Given categories $\cat{C}$ and
$\cat{D}$ and a functor $\cube{X}$ on $\cat{D}$ and a functor $F:
\cat{C}\rightarrow \cat{D}$, we want to produce a map 
$$\holim_{\cat{D}} \cube{X} \rightarrow \holim_{\cat{C}} F^*\cube{X}$$
That is, produce a function
$$
\Hom_{\cat{D}} (\realization{\cat{D}/-}, \cube{X}(-))
\rightarrow
\Hom_{\cat{C}} (\realization{\cat{C}/-}, \cube{X}\circ F(-))
$$
Essentially, this follows from the contravariance of $\Hom$ and the
fact that a functor $F: \cat{C} \rightarrow \cat{D}$ induces a
simplicial map $\realization{\cat{C}/c} \rightarrow
\realization{\cat{D}/F(c)}$ for all objects $c$ in $\cat{C}$.

Specifically, the elements on the left of the diagram above are
coherent collections of maps  
$\phi_d$ sending $n$-simplices corresponding to
$d \leftarrow d_1 \leftarrow \cdots \leftarrow d_n$ to 
$\cube{X}(d)$. Given one of these, 
define a function $\phi_c$ on $\realization{\cat{C}/c}$ by sending the
$n$-simplex $c \xleftarrow{\alpha_1} \cdots \xleftarrow{\alpha_n}
c_n$ to the same place as the $n$-simplex in $\cat{D}/F(c)$ corresponding
to its image under $F$:
$$ \phi_{F(c)} ( F(c) \xleftarrow{F(\alpha_1)} \cdots 
\xleftarrow{F(\alpha_n)} F(c_n) ). $$
Note that the target of this function is $\cube{X}F(c) =
F^*\cube{X}(c)$,  just as required. Coherence of the collection
$\Set{\phi_c}$ follows from that of the collection $\Set{\phi_d}$.


If the map $f: S \rightarrow T$ is surjective, then $\Power_0(f)$ is a
``left cofinal'' functor from $\Power_0(S)$ to $\Power_0(T)$, and hence
\cite[\S{}XI.9,
Theorem~9.2]{Bousfield-Kan:homotopy-limits-completions-and-localizations}
the induced map between homotopy fibers is a homotopy equivalence.
To verify that in this case $\Power_0(f)$ is left cofinal, we need to
check that for each $V \subset T$, a certain category $\Power_0(f)/V$ is
contractible. The objects in this category are elements $(U,\mu)$,
where $U \subset S$ and $\mu: f(U) \rightarrow V$ is a morphism in
$\Power_0(T)$; that is, an inclusion of $f(U)$ into $V$. 
In the power set category, the maps are inclusions of subsets, so
there can be at most one map between objects; hence the category
$\Power_0(f)/V$  has one element for each subset of $f^{-1}(V)$, and
maps correspond to inclusions. All subsets of $f^{-1}(V)$ include in
$f^{-1}(V)$, so $\Power_0(f)/V$ is contractible whenever $f^{-1}(V)$ is
nonempty. This shows that if $f$ is surjective, then $\Power_0(f)$ is
left cofinal.

If $f^{-1}(V)$ is empty for some $V$ (that is, when $f$ is not
surjective), then the category $\Power_0(f)/V$ is empty, and the empty
space is not contractible (that is, $\emptyset$ is not equivalent to
$\basept$), so $\Power_0(f)$ is not left cofinal in that case.

This completes the proof of
Proposition~\ref{prop:induced-map-on-hofib}. $\qed$

\subsection{$\Perp$ As Composition With Free Cube Functor}
\label{sec:hofib-free-cube-functor}

The purpose of developing the machinery of ``free cubes'' was to
enable us to show that $\Perp$ is a cotriple. In this section, we
identify the functor $\Perp$ as a composition of functors with the
free cube functor. Throughout this section, we fix a set $U$ and a
functor $F$ and a space $X$, and consider only $\Perp_U F(X)$.

Given sets $U$ and $S$, let $c_X$ be the function on $\Delta(S,U)$
that has a constant value $X$. Let $C(S)$ be the contravariant functor
of sets $S$ given by 
$$ C(S) = \hofib F \circ \Free(S,U,c_X) .$$

\emph{Functoriality.}
We need to establish that this is a functor. Let $\cube{Y} =
\Free(T,U,c_X)$. Given a function $f: S\rightarrow T$, we can
construct the $(S\times U)$-cube $\Power(f\times 1)^* \cube{Y}$. From
Proposition~\ref{prop:induced-map-on-hofib}, there is a natural map 
$$ \hofib_{\Power(T)} F \cube{Y}  
\rightarrow 
\hofib_{\Power(S)} \Power(f\times 1)^* \cube{Y},$$
so it remains to construct a map 
$$
\hofib_{\Power(S)}
F
\circ
\Power(f \times 1)^* \cube{Y} 
\rightarrow 
\hofib_{\Power(S)}
F \circ
\Free(S,U,c_X) .$$

Recall from Lemma~\ref{lem:pullbacks-of-free-cubes} that $\Power(f
\times 1)^* \cube{Y} $ is a free cube with generating function 
$$h(B) = \bigvee_{\Delta(f,1)^{-1}(B)} X .$$
To map this to the constant function $c_X$, we use the fold map from 
$h(B) = \bigvee X$ to $c_X(B) = X$. If $h(B)$ is the empty coproduct,
$0$, then the fold map is the inclusion of $0$ in $X$.
As noted in Remark~\ref{remark:maps-out-of-free-cube}, since $h$ is
the ``generating function'' for $\Power(f\times 1)^* \cube{Y}$, 
specifying 
maps $h(B) \rightarrow c_X(B)$ suffices to determine a map of cubes
$$\Power(f \times 1)^* \cube{Y} \rightarrow \Free(S,U,c_X) .$$
The homotopy fiber functor is covariant with respect to maps of cubes,
so applying $F$ and $\hofib$ give the required map
$$
\hofib_{\Power(S)}
F
\circ
\Power(f \times 1)^* \cube{Y} 
\rightarrow 
\hofib_{\Power(S)}
F \circ
\Free(S,U,c_X) .$$

\emph{Identification with $\Perp$.}
Let $\mathbf{k}$ denote the set $\Set{1,\ldots,k}$. To identify
$C(\mathbf{k})$ with $\Perp^k F(X)$, we recall from
Example~\ref{example:perp2-cube} that $\Perp^k F(X)$
is 
$$
\hofib_{\Set{V_1,\ldots, V_k} \subset U} 
F
\left(
\bigvee_{i=1}^{i=k} \bigvee_{v_i \not\in V_i} X
\right)
$$
The cube involved is a $(\mathbf{k}\times U)$-cube $\cube{X}$. Given $V \subset
(\mathbf{k}\times U)$,  for each $i \in \mathbf{k}$ define $V_i$ to be
the projection of $V \cap \Set{i}\times U$ to $U$. In the notation of
Definition~\ref{def:free-cube-alt}, $V_i = p_i(V)$. 
The cube then has vertices
$$\cube{X}(V) = \bigvee_{i=1}^{i=k} \bigvee_{v_i \not\in V_i} X .$$
From the alternative construction of a ``free cube'' in
Definition~\ref{def:free-cube-alt}, we see that this is exactly a free
cube $\Free(\mathbf{k},U,c_X)$, with $c_X$ a constant function whose
value is $X$ on all $B_f \in \Delta(\mathbf{k},U)$.
This shows that $C(\mathbf{k}) \cong \Perp^k_U F(X)$.

\emph{Claim}: The map $\epsilon: \Perp \rightarrow 1$ is induced by
applying $C$ to the inclusion $i: \emptyset \rightarrow 1$. 
Let $\cube{X} = \Free(1,U,c_X)$. 
According to Section~\ref{sec:homotopy-fibers}, the homotopy fiber of
the cube $F\cube{X}$ is a coherent collection of maps
$$\hofib_{\Power(U)} F \cube{X}
=
\left\{
\phi_c : \realization{\cat{C}/c} \rightarrow F \cube{X}(c)
\right\}_{c\in\Obj\cat{C}}
$$
The map from the homotopy fiber of $F\cube{X}$ over $\Power(U)$ to
the homotopy fiber of the restriction of $F$ over $\Power(\emptyset)$
sends each map $\phi = \Set{\phi_c}$ to the map $\phi_\emptyset$ induced by
the image of the point $\emptyset \rightarrow \emptyset$ in
$\cat{C}/\emptyset$; this is just a single point $\phi(\basept)$ in
$\cube{X}$. That is, $i^*$ induces the identity on
$\cube{X}(\emptyset) = X \vee X$. 

\emph{Definition}: The map $\delta: \Perp \rightarrow
\Perp\Perp$ is induced by the applying $C$ to the fold map $\Set{1,2}
\rightarrow \Set{1}$. The map $\delta$ has not been specified any
other way elsewhere in this work.

\subsection{$\Perp$ Is A Cotriple}

We are now in a position to use this machinery to show that $\Perp$
forms a cotriple. Fix a set $U$ and a functor $F$, and consider only
iterates of $\Perp_U$ applied to $F$, and evaluated at a fixed space
$X$. That is, we only work with $\Perp_U^k F(X)$. This is sufficient
to show that $\Perp_U$ is a cotriple.

A cotriple $\Perp$ requires that the following diagram commute:
\begin{equation*}
\xymatrix{
\Perp\Perp 
\ar[r]^{\Perp\epsilon}
\ar[d]^{\epsilon\Perp}
&
\Perp
\ar[d]^{\epsilon}
\\
\Perp
\ar[r]^{\epsilon}
&
1}
\end{equation*}

Applying the functor $C$ to the diagram 
$$\xymatrix{ 
\Set{1,2}
&
\Set{1}
\ar[l]^{i_1}
\\
\Set{2}
\ar[u]^{i_2}
&
\emptyset
\ar[l]
\ar[u]
} $$
yields a commuting diagram
$$\xymatrix{
C(\Set{1,2})
\ar[r]
\ar[d]
&
C(\Set{1})
\ar[d]
\\
C(\Set{2})
\ar[r]
&
C(\emptyset)
}$$
In view of our identification $C(\mathbf{k}) \cong \Perp^k_U F(X)$,
this is the diagram above.

The other commuting diagram required for a cotriple is the following:
\begin{equation*}
\xymatrix{
&
\Perp
\ar[dl]^{=}
\ar[d]^{\delta}
\ar[dr]^{=}
&
\\
\Perp
&
\Perp\Perp
\ar[l]^{\Perp\epsilon}
\ar[r]^{\epsilon\Perp}
&
\Perp
}  
\end{equation*}

This results from applying $C$ to the diagram of sets:
$$\xymatrix{
{} 
&
\Set{1}
&
{}
\\
\Set{1}
\ar[ur]^=
\ar[r]^{i_1}
&
\Set{1,2}
\ar[u]
&
\Set{2}
\ar[ul]^=
\ar[l]^{i_2}
}$$

To summarize, we have shown:
\begin{theorem}
\label{thm:perp-is-cotriple}
The functor $\Perp$ of Definition~\ref{def:crn-perp} is a cotriple
on the category of homotopy functors from pointed spaces to pointed
spaces. $\qed$ 
\end{theorem}



%
%

\chapter{Properties Of $P_n^d$ And $\Perp_n$}
\label{chap:properties-of-Pnd}
Recall that the $n$-additive approximation functor, $P_n^d$ is given
by 
$$ P_n^d F(X) = P_n (L F_X)(S^0),$$
where $L$ is the left Kan extension over all finite sets and $F_X(Y) =
F(X\wedge Y)$.

In this chapter, we prove basic properties about the functors $P_n^d$
and $\Perp_n$. In order to be able to work effectively with $P_n^d$,
we need to restrict the functors under consideration to those that
give us some control of the behavior on $\pi_0$. To do this, we
introduce two hypotheses that a functor may satisfy in order for our
results to be applicable.

\begin{hypothesis}[Connected Values]
\label{hypothesis-1}
$F$ has connected values (on coproducts of $X$) if the functor $F$ has
the property that for spaces $X$ under consideration, $F(\bigvee X)$
is connected for all finite coproducts of $X$.
\end{hypothesis}


\begin{hypothesis}[Group Values]
\label{hypothesis-2}
In the following definition, let $\cat{T}$ denote the category of
pointed spaces, and let $\cat{C}$ denote the full 
subcategory of $\cat{T}$ generated by all finite coproducts of $S^0$.
Let $\cat{G}$ denote the category of topological groups, and let $U:
\cat{G} \rightarrow \cat{T}$ be the forgetful functor.

$F$ is group-valued (on coproducts of $X$) if there exists a functor
$F'$ so that the following diagram commutes:
$$\xymatrix{
&
\cat{G}
\ar[d]^U
\\
\cat{C}
\ar[ur]^{F'}
\ar[r]_{F(X \wedge -)}
&
\cat{T}
}$$
In this case, we will conflate $F(X\wedge -)$ with its lift to groups.
\end{hypothesis}

Functors that satisfy Hypothesis~\ref{hypothesis-1} (connected values)
or Hypothesis~\ref{hypothesis-2} (group values) on coproducts of $X$
are the subject of the first two sections in this chapter.
Section~\ref{sec:Pnd-preserves-fibrations} gives conditions under
which the functor $P_n^d$ preserves fibrations of functors.
Section~\ref{sec:Pnd-preserves-connectivity} establishes a
fundamentally important lemma for working with the approximations
$P_n^d$: they preserve connectivity of (good) natural transformations
of functors. 

The last section of this chapter
(\S\ref{sec:fiber-contractible-Cartesian}) contains a proof of a
technical result showing that under good circumstances if the total fiber
of a cube is contractible, then it is Cartesian. This
technical lemma makes it possible to deduce useful information from
the functor $\Perp_n$.

%
%
\section{$P_n^d$ Preserves (Group Or Connected) Fibrations}
\label{sec:Pnd-preserves-fibrations}
In this section, we give conditions under which $P_n^d$ preserves
fibrations. This lets us understand the effects of $P_n^d$ when we
understand the decomposition of a functor as a part of a fibration
over some other functor.

\begin{proposition}
\label{prop:LK-group-connected-base}
Given a space $X$ and functors $A$, $B$, and $C$, suppose 
$A(Y) \rightarrow B(Y) \rightarrow C(Y)$ is a fibration sequence for
all finite coproducts $Y = \bigvee X$ of $X$. 
If, on finite coproducts of $X$, either:
\begin{enumerate}
\item $C$ takes connected values
(Hypothesis~\ref{hypothesis-1});
or
\item $B$ and $C$ take group values (Hypothesis~\ref{hypothesis-2}),
and the map $B\rightarrow C$ is a surjective homomorphism of groups, 
\end{enumerate}
then
\begin{equation}
\label{eq:simple-LK-fib}
LA_X(Z) \rightarrow LB_X(Z) \rightarrow LC_X(Z)
\end{equation}
is a fibration sequence for all spaces $Z$. Furthermore, the sequence
is surjective on $\pi_0$.
\end{proposition}
\begin{proof}
Equation~\eqref{eq:simple-LK-fib} is equivalent to
\begin{equation}
\label{eq:A-B-C-realization-seq}
\realization{A(X \wedge Z_{\cdot})}
\rightarrow 
\realization{B(X \wedge Z_{\cdot})}
\rightarrow 
\realization{C(X \wedge Z_{\cdot})} ,
\end{equation}
and since $Z_\cdot$ is discrete, $C(X \wedge Z_{\cdot}) = C(\bigvee
X)$, for some coproduct of $X$. In the case of
Hypothesis~\ref{hypothesis-1}, the base space is
always connected, so Waldhausen's Lemma
(Lemma~\ref{lem:Waldhausen-fibration-lemma})
tells us that the fiber of the realization is the realization of the
fibers, so \eqref{eq:simple-LK-fib} is a fibration.
In the case of Hypothesis~\ref{hypothesis-2}, the total space and the
base space are simplicial groups, and hence satisfy the $\pi_*$-Kan
condition (\ref{def:pi-star-Kan}). Furthermore, since we have
assumed that the map $B\rightarrow C$ is surjective, the
second map in \eqref{eq:A-B-C-realization-seq} is surjective levelwise. In
particular, a surjective map of simplicial groups is a fibration, so
this map is a fibration on $\pi_0$. That allows us to apply
Theorem~\ref{thm:Bousfield-Friedlander} (Bousfield-Friedlander) 
to
conclude that \eqref{eq:simple-LK-fib} is a fibration.
Surjectivity on $\pi_0$ follows from noting that the realization of a
levelwise $0$-connected map is $0$-connected.
\end{proof}

\begin{corollary}
\label{cor:Pnd-group-connected-base-no-pi0}
Under the conditions of
Proposition~\ref{prop:LK-group-connected-base}, 
\begin{equation}
\label{eq:simple-Pnd-fib}
P_n^d A(X) \rightarrow P_n^d B(X) \rightarrow P_n^d C(X)
\end{equation}
is a fibration sequence.
\end{corollary}
\begin{proof}
By
Proposition~\ref{prop:LK-group-connected-base}, applying $L(-)_X$ to the
sequence $A\rightarrow B\rightarrow C$ yields a fibration sequence of
functors (\emph{i.e.}, a fibration when evaluated at any space). 
Applying $P_n$ preserves this fibration, as does evaluation
at $S^0$. That is the definition of $P_n^d$.
\end{proof}

\begin{remark}
Notice that this argument does not show that the resulting fibration
sequence is surjective on $\pi_0$. Even $P_1$ need not preserve
connectivity of maps (or spaces). For example, the functor $F(X)
= \Sigma\Omega X$ always produces $0$-connected spaces, but $P_1 F(X)
\simeq QX$
need only be $(-1)$-connected. ($F$ does not increase the
connectivity of $0$-connected spaces, just $(-1)$-connected spaces.)
\end{remark}

%
%
\section{$P_n^d$ Preserves Connectivity Of Natural Transformations}
\label{sec:Pnd-preserves-connectivity} 
In this section, we establish a property of fundamental importance
when working with $P_n^d$: the $n$-additive approximation preserves
the connectivity of natural transformations that satisfy some basic
good properties.

\begin{theorem}
\label{thm:Pnd-preserves-connectivity}
Let $F$ and $G$ be functors that have connected values
(Hypothesis~\ref{hypothesis-1}) 
or group values
(Hypothesis~\ref{hypothesis-2}) 
on coproducts of $X$.
Let $\eta: F \rightarrow G$ be a natural transformation of functors,
and suppose that $\eta$ lifts to a homomorphism of groups on $\pi_0$
in the case of Hypothesis~\ref{hypothesis-2}.
If $\eta$ is $k$-connected on coproducts of $X$, then $P_n^d \eta$ is
$k$-connected.
\end{theorem}

First let us recall some more of Goodwillie's calculus machinery. For
our purposes the action of $\Sigma_n$ is not particularly important,
but we describe it here for completeness. This exposition is based on
Goodwillie's lectures in Aberdeen, Scotland, June 18--23, 2001.

Before we begin, the definition of a derivative will require an action
of $\Sigma_n$ on 
$\Omega^{n-1}$. To write this, we will let $V_n$ denote the standard
representation of $\Sigma_n$ on $\mathbb{R}^n$. Let $\overbar{V_n}$
denote the representation of $\Sigma_n$ on $\mathbb{R}^{n-1}$ created
by splitting the one-dimensional trivial representation off of $V_n$.
Let $S^{V}$ be the one-point compactification of $V$. Let $\Omega^V$
denote $\Map(S^V, -)$. For $k\in\mathbb{N}$, let $k\cdot V$ denote
the product of $V$ with itself $k$ times.

\begin{definition}[Derivative of $F$]
\label{def:derivative-of-F}
The $n^{\text{th}}$ derivative of $F$ (at $\basept$), denoted
$\partial^{(n)} F(\basept)$, is the following spectrum with
$\Sigma_n$ action, which we will denote $\mathbf{Y}$.
The space $Y_k$ in the spectrum is
$\Omega^{k \cdot \overbar{V_n}} cr_n F(S^k,\ldots, S^k)$. The
structure map $Y_k \rightarrow \Omega Y_{k+1}$ arises from applying
$T_1$ in each variable of $cr_n$ (that is, suspending inside and looping
outside). Given any symmetric functor of $n$ variables
$H(A,\ldots,A)$,  this procedure produces a 
$\Sigma_n$-equivariant map $H (A,\ldots,A)
\rightarrow \Omega^n H (S^1 \wedge A, \ldots,  S^1\wedge A)$.
Here $\Omega^n$ is $\Omega^{V_n}$, so we can split off the trivial
representation and write this as a map
$$H (A,\ldots,A)
\rightarrow 
\Omega \Omega^{\overbar{V_n}} 
H (S^1 \wedge A, \ldots,  S^1\wedge A).
$$
Using $H = \Omega^{k \cdot \overbar{V_n}} cr_n F$ and $A=S^k$
produces the desired map $Y_k \rightarrow \Omega Y_{k+1}$.
\end{definition}

We will only really be interested in the nonequivariant homotopy type;
for our application we only need to understand the spectrum
$\mathbf{Y}$ with $Y_k = \Omega^{k(n-1)}\Perp_n F(S^k)$.

When $F$ satisfies the limit axiom
(\ref{def:limit-axiom}), we can express $D_n F(X)$ using the derivative:
\begin{equation}
\label{eq:def-of-deriv}
D_n F(X) \simeq \Omega^\infty \left( \partial^{(n)} F(\basept)
\wedge_{h \Sigma_n} X^{\wedge n} \right).
\end{equation}
In order to prove Theorem~\ref{thm:Pnd-preserves-connectivity}, we
first want to establish that $\partial^{(n)} LF_X(\basept)$ is a
connective spectrum for all $n$. Note that $L F_X$ always satisfies
the limit axiom (Lemma~\ref{lem:left-Kan-limit-axiom}).

\begin{lemma}
\label{lem:derivatives-connective}
If $F$ is a reduced functor that 
has connected values (Hypothesis~\ref{hypothesis-1}) or group
values (Hypothesis~\ref{hypothesis-2}) on coproducts of $X$, then for
all $n$, the derivative $\partial^{(n)} LF_X(\basept)$
of $L F_X$ is a connective spectrum.
\end{lemma}
\begin{proof}
Lemma~\ref{lem:perp-commutes-with-realization-v2} essentially shows
that $L$ and $\Perp$ commute; we have defered the proof to the end of
this section so as not to get
bogged down in technicalities.
Using Lemma~\ref{lem:perp-commutes-with-realization-v2}, we can compute
$\Perp_n L F_X(S^k) \simeq \Perp_n \realization {F (X \wedge
  S^k_\cdot)}$ by computing $\realization {\Perp_n F(X\wedge
  S^k_\cdot)}$. Now $\Perp_n F(X\wedge S^k_\cdot) = cr_n F(X\wedge
S^k_\cdot, \ldots, X \wedge S^k_\cdot)$, and $cr_n F$ is contractible
if any one of its inputs is contractible, 
so by the Eilenberg-Zilber Theorem
(\ref{thm:eilenberg-zilber}), $\realization{\Perp_n F(X \wedge
  S^k_\cdot)}$ is $(n k-1)$-connected. Subtracting $k(n-1)$ from
$nk-1$ shows that  $\Omega^{k(n-1)} \Perp_n
(LF_X)(S^k)$ is $(k-1)$-connected. 

To determine the connectivity of the spectrum $\mathbf{Y} =
\partial^{(n)} LF_X(\basept)$, we compute the colimit:
\begin{align*}
\pi_m \mathbf{Y} 
&=
\colim_k \pi_{m+k} 
Y_k
\\
&= 
\colim_k \pi_{m+k} 
\Omega^{k(n-1)} \Perp_n (LF_X)(S^k).
\end{align*}
We have just shown that the space $\Omega^{k(n-1)} \Perp_n
(LF_X)(S^k)$ is $(k-1)$-connected, so $\pi_m \mathbf{Y} = 0$ for
$m<0$. That is, $\partial^{(n)} LF_X(\basept) = \mathbf{Y}$ is a
connective spectrum. 
\end{proof}

\begin{corollary}
\label{cor:Pn+1-Pn-surjective}
If $F$ is a reduced functor that 
has connected values (Hypothesis~\ref{hypothesis-1}) or group
values (Hypothesis~\ref{hypothesis-2}) on coproducts of $X$, then for
all $n\ge 1$, the map $P_{n+1}^d F(X)\rightarrow P_n^d F(X)$ is
surjective in $\pi_0$.
\end{corollary}
\begin{proof}
Theorem~\ref{thm:delooping-Dn} (Goodwillie's delooping of $D_n$) shows
that the fibration 
$$D_{n+1} LF_X \rightarrow P_{n+1} L F_X \rightarrow P_n LF_X$$
deloops to a fibration
$$ P_{n+1} LF_X \rightarrow P_n LF_X \rightarrow \Omega^{-1} D_{n+1} LF_X.$$
The delooping of $D_{n+1} LF_X$ consists of smashing with the
suspension of $\partial^{(n+1)} LF_X(\basept)$ and taking homotopy orbits. 
By Lemma~\ref{lem:derivatives-connective}, the spectrum
$\partial^{(n+1)} LF_X(\basept)$ is connective, so its suspension is
$0$-connected; hence $\pi_0\Omega^{-1} D_{n+1} LF_X = 0$, so
evaluation at $S^0$ shows that the map 
$P_{n+1}^d F(X) \rightarrow P_{n}^d F(X)$ is surjective on $\pi_0$.
\end{proof}

\begin{lemma}
\label{lem:Dnd-preserves-conn}
Let $F$ and $G$ be reduced functors, and suppose that either 
$F$ and $G$ have connected values
(Hypothesis~\ref{hypothesis-1}) on coproducts of $X$, or
$F$ and $G$ have group values (Hypothesis~\ref{hypothesis-2}) on
coproducts of $X$. 
If $\eta: F\rightarrow G$ is a natural transformation that is
$w$-connected on coproducts of $X$, then for $n\ge 1$, the natural
transformation $D_n^d \eta$ is $w$-connected.
\end{lemma}
\begin{proof}
For any functor $H$, we know
\begin{align*}
D_n^d H(X) &= D_n (LH_X)(S^0) \\
&= \LoopInfty \left(
\partial^{(n)} LH_X (\basept) \wedge_{h\Sigma_n} (S^0)^{\wedge n} \right) . 
\end{align*}
Taking homotopy orbits and smashing with a fixed space preserves
connectivity, so this is really a question about the connectivity of
the map $\partial^{(n)} LF_X (\basept) \rightarrow \partial^{(n)} LG_X
(\basept)$. 

First, consider the case $n=1$.
Since $\eta: F\rightarrow G$ is $w$-connected on coproducts of $X$,
the map $\realization{F_X(S^k_\cdot)} \rightarrow
\realization{G_X(S^k_\cdot)}$ is $(k+w)$-connected (by
Eilenberg-Zilber, since both are contractible levelwise until
dimension $k$). The derivative spectrum $\partial LF_X(\basept)$ 
has $k^{\text{th}}$ space $\realization{F_X(S^k_\cdot)}$, and
similarly for $\partial LG_X(\basept)$, so this shows that the map 
$\partial LF_X(\basept) \rightarrow
\partial LG_X(\basept)$ is a $w$-connected map.

Similarly, for all $n\ge 1$, the map 
 $\realization{\Perp_n F_X(S^k_\cdot)}  \rightarrow 
\realization{\Perp_n G_X(S^k_\cdot)}$ is $(nk+w)$-connected. The
derivative spectrum $\partial^{(n)}LF_X(\basept)$ 
then has as its $k^{\text{th}}$ space the space
$\Omega^{k(n-1)} \realization{\Perp_n F_X(S^k_\cdot)}$. On these
spaces the map induced by $\eta$ is $(k+w)$-connected, exactly as
required to produce a $w$-connected map 
$\partial^{(n)} LF_X(\basept) \rightarrow
\partial^{(n)} LG_X(\basept)$.
\end{proof}

\begin{corollary}
\label{cor:Pnd-preserves-conn}
Let $F$ and $G$ be reduced functors, and suppose that either 
$F$ and $G$ have connected values
(Hypothesis~\ref{hypothesis-1}) on coproducts of $X$, or
$F$ and $G$ have group values (Hypothesis~\ref{hypothesis-2}) on
coproducts of $X$. 
If $\eta: F\rightarrow G$ is a natural transformation that is
$w$-connected on coproducts of $X$, then for $n\ge 1$, the natural
transformation $P_n^d \eta$ is $w$-connected.
\end{corollary}

\begin{proof}
We induct up the Goodwillie tower for $F$ and $G$.
Our hypotheses on the values of $F$ and $G$, along with
Corollary~\ref{cor:Pn+1-Pn-surjective} provide the control on $\pi_0$
needed to use the Five Lemma, making this an easy consequence of
Corollary~\ref{lem:Dnd-preserves-conn}.
\end{proof}

\begin{proof}[Proof of Theorem~\ref{thm:Pnd-preserves-connectivity}]
First we will show that we may reduce to the case of reduced functors. 
Let $F_0(X) = F(0)$ and $G_0(X) = G(0)$ be constant functors, and
consider the fibration
$$\xymatrix{
\widetilde{F}
\ar[r]
\ar[d]
&
\widetilde{G}
\ar[d]
\\
{F}
\ar[r]
\ar[d]
&
{G}
\ar[d]
\\
{F_0}
\ar[r]
&
{G_0}
\\
}$$
By Corollary~\ref{cor:Pnd-group-connected-base-no-pi0}, 
applying $P_n^d$ preserves these fibrations. The map  $F\rightarrow
F_0$ and $G\rightarrow G_0$ have sections induced by $0\rightarrow X$,
so applying $P_n^d$ also preserves surjectivity on $\pi_0$.

The approximation $P_n^d$ applied to a constant functor is just the
constant functor again, so the map on the bases is $k$-connected.
Corollary~\ref{cor:Pnd-preserves-conn} shows that the induced map
$P_n^d \widetilde{F} \rightarrow P_n^d \widetilde{G}$ is
$k$-connected.  The Five Lemma applies in this situation (connected or
group values, surjective on $\pi_0$), allowing us to conclude that the
map on the total spaces is $k$-connected.
\end{proof}

Theorem~\ref{thm:Pnd-preserves-connectivity} gives us the following
slight but essential improvement of
Corollary~\ref{cor:Pnd-group-connected-base-no-pi0}:
\begin{corollary}
\label{cor:Pnd-group-connected-base}
Given a space $X$ and functors $A$, $B$, and $C$, suppose 
$A(Y) \rightarrow B(Y) \rightarrow C(Y)$ is a fibration sequence for
all finite coproducts $Y = \bigvee X$ of $X$. 
If, on finite coproducts of $X$, either:
\begin{enumerate}
\item $C$ takes connected values
(Hypothesis~\ref{hypothesis-1});
or
\item $B$ and $C$ take group values (Hypothesis~\ref{hypothesis-2}),
and the map $B\rightarrow C$ is a surjective homomorphism of groups, 
\end{enumerate}
then
\begin{equation*}
P_n^d A(X) \rightarrow P_n^d B(X) \rightarrow P_n^d C(X)
\end{equation*}
is a fibration sequence that is surjective on $\pi_0$.
$\qed$
\end{corollary}

We have deferred the proof that $\Perp$ commutes with realizations
until this point. It is straightforward.
\begin{lemma}
\label{lem:perp-commutes-with-realization-v2}
If $F$ has connected values (Hypothesis~\ref{hypothesis-1}) or group
values (Hypothesis~\ref{hypothesis-2}) on coproducts of $X$, then
$$
cr_n (LF_X)(Y^1,\ldots,Y^n)
\simeq 
L^n (cr_n F)_{(X,\ldots,X)} (Y^1,\ldots,Y^n)
.
$$
For each $j$, let $Y^j_\cdot$ be a simplicial set whose realization is
$Y^j$. Then 
$$L^n (cr_n F)_{(X,\ldots,X)} (Y^1,\ldots,Y^n)
\simeq
\realization{(cr_n F)_{(X,\ldots,X)} (Y^1_\cdot,\ldots,Y^n_\cdot)}
,$$
where the realization is taken in each of the $n$ dimensions involved,
and the realization can be the ``strict'' realization with no adverse
effect on the homotopy type of the result.

The statement above can be abbreviated to $\Perp_n(LF_X)(Y) \simeq L
(\Perp_n F)_X(Y)$. 
\end{lemma}
\begin{proof}
%
  The space $cr_n (LF_X)(Y^1,\ldots,Y^n)$ is the homotopy fiber of a
  cube involving $LF_X\left(\bigvee_{u\in U} Y^u \right)$ for
  $U\subset\Set{1,\ldots,n}$. We will write this part of our argument
  assuming $U = \Set{1,\ldots,n}$ for simplicity.
  By
  Proposition~\ref{prop:left-kan-commutes-with-some-realizations},
  this is the realization of the simplicial space 
  $$F_X (\diag(Y^1_\cdot\vee \cdots\vee Y^n_\cdot))
  =
  \diag F_X ( Y^1_\cdot\vee\cdots\vee Y^n_\cdot)
  .$$
  Lemma~\ref{lem:LCFX-good} shows that
  this is a good space, so we can use the strict realization. Then the 
  Eilenberg-Zilber theorem (\ref{thm:eilenberg-zilber}) shows that the
  diagonal has the same homotopy type as the (multidimensional)
  realization of the $n$-dimensional simplicial space $F_X (
  Y^1_\cdot\vee\cdots\vee Y^n_\cdot)$. 

  Now if we show that we can take the fibers of the maps before
  applying the realization functor, we will be able to interpret the
  fibers of each dimension $(k_1,\ldots,k_n)$ as $(cr_n
  F)_{(X,\ldots,X)} (Y^1_{k_1},\ldots,Y^n_{k_n})$, which is the left
  Kan extension in each variable, as desired for the lemma.

  If $F$ satisfies Hypothesis~\ref{hypothesis-1}, then we can
  compute the fibers in the $\Perp$-cube levelwise using Waldhausen's
  Lemma (Lemma~\ref{lem:Waldhausen-fibration-lemma}).

  If $F$ satisfies Hypothesis~\ref{hypothesis-2}, then we will use
  Theorem~\ref{thm:Bousfield-Friedlander} (Bousfield-Friedlander) 
  to produce the same
  result.
  In this case, $F_X$ also satisfies Hypothesis~\ref{hypothesis-2} on
  coproducts of $S^0$ (i.e., on all finite sets). This shows that for
  any simplicial set $Y_\cdot$, we may regard the simplicial space
  $LF_X(Y) \simeq \realization{F_X(Y_\cdot)}$ as a simplicial group.
  Hence each corner of the
  $\Perp$-cube satisfies the $\pi_*$-Kan condition. 
  Furthermore, all of the maps in the $\Perp$-cube have compatible
  sections, so at each stage of taking iterated fibers all of the
  structure maps have sections. This gives us surjective maps of
  simplicial groups, so the induced maps on $\pi_0$ are
  fibrations. These two conditions are enough to apply
  Theorem~\ref{thm:Bousfield-Friedlander} to compute the fibers
  levelwise. 
\end{proof}


%
%
\section{Fiber Contractible Implies Cartesian (Group Or Connected)}
\label{sec:fiber-contractible-Cartesian}

We will prove the critical fact that in the
cases we consider, the cross effect vanishing is equivalent to the
cross effect cubes being Cartesian.

We generally want to use the fact that the cross effect is
contractible to conclude that the initial space in the
cross-effect cube is equivalent to the (homotopy) inverse limit of the
rest of the spaces. Unfortunately, this is not always true; the
problem is that the homotopy fiber does not detect failure to be
surjective on $\pi_0$. Fortunately, with some mild hypotheses we are
able to ensure that we stay within the realm where this issue is
avoided for one reason or another. 

This section lays the groundwork that shows that in good situations,
the cross-effect cubes are well behaved and total fiber contractible
implies Cartesian.
Section~\ref{sec:cartesian-connected-values} considers the case of a
functor to connected spaces.
Section~\ref{sec:cartesian-group-values} considers the case when the
functor is group-valued. 

If the spectral sequence developed by Bousfield and Kan in
\cite{Bousfield-Kan:homotopy-limits-completions-and-localizations}
were shown to converge to $\pi_0$ under these conditions, it could be
used to make the arguments in this section shorter. Their work would
give 
$\pi_k \holim_{\Power_0(S)} \cube{X} = \lim_{\Power_0(S)} \pi_k
\cube{X}$
with higher $\lim^i$ terms vanishing because of the structure of the
cube used to compute $\Perp$.


\subsection{Connected Values}
\label{sec:cartesian-connected-values}

Even if all of the spaces in the cross-effect cube are connected,
we need to know that the homotopy inverse limit of the punctured cube
(``the rest of the spaces''=$\Power_0(S)$) is still connected in order
to be able to conclude that cross effect zero implies the cube is
Cartesian. 

To show that the homotopy inverse limit of the punctured cube is
connected, we proceed as follows: first, we show that a pullback of a
diagram of connected spaces with section maps is connected; then we
decompose the whole homotopy inverse limit into (iterated) pullbacks
of diagrams of this form and diagrams with initial objects.

Given a cube of spaces with compatible sections, the first thing we
want to do is replace all of the maps in the cube by fibrations so
that the homotopy inverse limit is equivalent to the strict inverse
limit.
\begin{lemma}
\label{lem:fibration-cube-with-sections}
If $\cube{X}$ is a cube of spaces with compatible sections to all
structure maps (\ref{hypothesis:compatible-sections}), 
then $\cube{X}$ is equivalent to a cube $\cube{X'}$ of
spaces in which all structure maps are fibrations, and all of these
maps still have compatible sections.
\end{lemma}
\begin{proof}
From \cite[Remark~1.14, p.~305]{Cal2}, every cube of spaces is
equivalent to a fibration cube by replacing $\cube{X}(U)$ by
$\cube{Y}(U) = \holim_{U \subset V} \cube{X}(V) = \holim (\partial_U
\cube{X})$. The maps in the cube 
are then induced by the inclusion of indexing categories. The section
maps in $\cube{X}$ can then be used to give section maps in
$\cube{Y}$. 

We will now briefly sketch an example showing how to construct the
section maps.  The reader may wish to review the definition of the
precise construction of homotopy inverse limit that we are using
before proceeding. 

Given a map $f:c \rightarrow d$ in $\cat{C} = \Power(S)$, and a map $g:
\realization{d\undercat \cat{C} \overcat d} \rightarrow \cube{X}(d)$
(representing a point in $\holim ( \partial_{\Set{d}} \cube{X})$),
produce an element in $\holim ( \partial_{\Set{c}} \cube{X})$ as
follows: collapse $\realization{ c \undercat \cat{C} \overcat d }$ to the
image of $\realization{ d \undercat \cat{C} \overcat d }$, then use $g$ to
map to $\cube{X}(d)$. On $\realization{ c \undercat \cat{C} \overcat c }$,
first map a point to its image in  $\realization{ c \undercat \cat{C}
  \overcat d }$, then to $\cube{X}(d)$, and then via the section map to
$\cube{X}(c)$. 
The general case is obviously very similar, just harder to write down
explicitly.
\end{proof}

Now we establish that if we have section maps, the pullback of
connected spaces is connected. 
\begin{lemma}
\label{lem:holim-connected-two-cube}
Let $\cube{X}$ be an $\mathbf{2}$-cube with compatible sections to all
structure maps (\ref{hypothesis:compatible-sections}). 
If each space $\cube{X}(U)$ is connected, then so is
$\holim_{\Power_0(\mathbf{2})} \cube{X}$.
\end{lemma}
\begin{proof}
By Lemma~\ref{lem:fibration-cube-with-sections}, we may assume that
$\cube{X}$ is a fibration cube. Then the homotopy inverse limit is
equivalent to the strict inverse limit, so we need only show that if 
$$\cube{X} = ( X \xrightarrow{p_X} Z \xleftarrow{p_Y} Y ) $$
is a diagram with sections $s_X$ and $s_Y$ to the maps $p_X$ and
$p_Y$, and all three spaces are connected, then so is their inverse
limit. 

A map $f$ of the $0$-sphere to the inverse limit is equivalent to
compatible maps $f_X$, $f_Y$, and $f_Z$ of $S^0$ to all three spaces.
We first show that $f = (f_X, f_Z, f_Y)$ is homotopic to the map $f' =
(s_X f_Z, f_Z, s_Y f_Z)$, and then use a homotopy in $Z$ to show $f'$
is null homotopic.

Since $Z$ is
connected, the fibers of the map $p_X: X\rightarrow Z$ over every
point are equivalent.  Due to the section $s_X$, the fibers over every
point in $Z$ are connected (existence of the section map implies
surjectivity on $\pi_*$, so connectivity of the fiber cannot drop).
Let $1\in S^0$ denote the non-basepoint element of the $0$-sphere.
The points $f_X(1)$ and $s_z f_Z(1)$ are in the same fiber over the
point $f_Z(1)$, and this fiber is connected, so there is a homotopy
$H_X: f_X \simeq s_z f_Z$ that stays entirely within the fiber (so
$p_X H_X$ is the constant map $f_Z$). 

By the symmetry of $X$ and $Y$, this shows that the map $f = (f_X,
f_Z, f_Y)$ is homotopic to the map $f'$ with components $(s_X f_Z,
f_Z, s_Y f_Z)$. Now let $H: D^1 \rightarrow Z$ be a homotopy $f_Z
\simeq \basept$. Then the homotopy $(s_X H, H, s_Y H)$ is a null
homotopy of $f'$. 

This shows that $\pi_0 \holim_{\Power_0(\mathbf{2})} \cube{X} = \basept$, as
required for the lemma.
\end{proof}

We can now use  Proposition~\ref{prop:Goodwillie-decompose-holim} to
decompose the inverse limit of the 
punctured cube into pullbacks and inverse limits with initial objects. 


\begin{lemma}
\label{lem:holim-connected}
Let $\cube{X}$ be an $S$-cube with compatible sections to all
structure maps (\ref{hypothesis:compatible-sections}). 
If each space $\cube{X}(U)$ is connected, then so is
$\holim_{\Power_0(S)} \cube{X}$.
\end{lemma}
\begin{proof}
As indicated, our approach is to use
Proposition~\ref{prop:Goodwillie-decompose-holim} to write
$\holim_{\Power_0(S)} \cube{X}$ 
as pullbacks and spaces that are vertices of $\cube{X}$. As in
Lemma~\ref{lem:fibration-cube-with-sections}, when this is done in a
natural way, all of the maps between the inverse limits will have
sections, so we will be able to apply the result for the case of a
pullback with sections, Lemma~\ref{lem:holim-connected-two-cube},
repeatedly to conclude that the whole inverse limit is connected.

It is well known how to produce an inverse limit of
the type used here by iterating pullbacks, so here we will just sketch
the method used. To produce a homotopy pullback from
Proposition~\ref{prop:Goodwillie-decompose-holim}, we need to produce
a decomposition of $\Power_0(S)$. Let $\cube{A}(U)$ be the full
subcategory of $\Power_0(S)$ generated by $ \Set{ V \subset S 
\suchthat  U\subset V }$. Begin the decomposition by considering 
$\cube{A}(U)$ for $U$ a maximal element, and the union of
$\cube{A}(V)$ over all of the other maximal elements $V\not= U$. 
The sets are finite and the maximum height of the maximal elements
decreases in the intersection, so proceeding inductively we end up
with the base case of a pullback in from
Lemma~\ref{lem:holim-connected-two-cube}. 
\end{proof}

\begin{example}
As an example of the decomposition of the homotopy inverse limit in
Lemma~\ref{lem:holim-connected}, consider the $S$-cube $\cube{X}$ with
$S = \Set{1,2}$. We have two maximal elements in $\Power_0(S)$:
$\Set{1}$ and $\Set{2}$, so our decomposition is 
$A(1) = (\Set{1} \rightarrow \Set{1,2})$ and $A(2) = (\Set{2}
\rightarrow \Set{1,2})$.
Proposition~\ref{prop:Goodwillie-decompose-holim} gives
$\holim{\Power_0(S)} \cube{X}$ equal to the homotopy pullback of
$$ 
\holim_{A(1)} \cube{X} 
\rightarrow 
\holim_{A(1)\cap A(2)} \cube{X}
\leftarrow
\holim_{A(2)} \cube{X} .
$$
Now $\holim_{A(1)} \cube{X} \simeq \cube{X}(1)$, since $\Set{1}$ is
initial in $A(1)$, and similarly $\holim_{A(2)} \cube{X} \simeq
\cube{X}(2)$, and $\holim_{A(1)\cap A(2)} \cube{X} \simeq
\cube{X}(\Set{1,2})$.
\end{example}

\begin{example}
A more complicated example is the case of a $3$-cube, where we begin
with $A(1)$ and $A(2)\cup A(3)$. Then their intersection is 
$\left( A(1)\cap
A(2)\right) \cup \left( A(1)\cap A(3) \right)$. This category is 
$$ \Set{1,2} \rightarrow \Set{1,2,3} \leftarrow \Set{1,3}, $$
so the inverse limit over it is a pullback. Then one proceeds to
decompose $A(2)\cup A(3)$ into $A(2)$ and $A(3)$ in a manner similar
to the $2$-cube from the previous example. 
\end{example}

Knowing that the homotopy limit is connected is the key piece of
information to conclude that the $\Perp$-cube is Cartesian when the
total fiber is contractible.
\begin{lemma}
\label{lem:connected-cartesian}
Let $F$ be a functor satisfying Hypothesis~\ref{hypothesis-1}
(connected values) on coproducts of $X$. 
If $\Perp_n F(X) \simeq 0$, then the cube defining $\Perp_n
F(X)$ is Cartesian.
\end{lemma}
\begin{proof}
Let $\cube{C}$ be the cube defining $\Perp_n F(X)$. 
Lemma~\ref{lem:holim-connected} shows that 
$\holim_{U \in \Power_0(\mathbf{n})} \cube{C}(U)$ is connected. 
The cross effect is the homotopy fiber in the quasifibration
$$ \Perp_n F(X) 
\rightarrow 
\cube{C}(\emptyset) 
\rightarrow 
\holim_{U \in \Power_0(\mathbf{n})} \cube{C}(U) .
$$
The cube $\cube{C}$ is Cartesian if the right map is an equivalence. 
A map is an equivalence if it has a contractible homotopy fiber and is
surjective on $\pi_0$. The fiber is $\Perp_n F(X) \simeq 0$, and the
base has $\pi_0 = \basept$, so the map is an equivalence.
\end{proof}

\subsection{Group Values}
\label{sec:cartesian-group-values}
In this section, we establish that for a functor to groups, $\Perp F(X)
\simeq 0$ implies that the cube defining $\Perp F(X)$ is Cartesian.

\begin{lemma}
\label{lem:groups-surj-X0-lim}
Let $\cube{X}$ be an $n$-cube ($n\ge 1$) of discrete groups.
If $\cube{X}$ has compatible 
sections to all structure maps
(\ref{hypothesis:compatible-sections}), then the map
$$\cube{X}(\emptyset) \rightarrow \lim_{U\in\Power_0(\mathbf{n})} \cube{X}(U)$$
is surjective.
\end{lemma}
\begin{proof}
We need to show that the map above is surjective.
This is equivalent to showing that there exists an
$x_\emptyset \in \cube{X}(\emptyset)$ mapping to each coherent system of
elements $x_U \in \cube{X}(U)$, with $U \not= \emptyset$.

$\cube{X}$ is a cube of groups, and hence all of the structure maps are
group homomorphisms. This allows us to subtract an arbitrary
$w\in\cube{X}(\emptyset)$ from $x_\emptyset$, and subtract the images
$\Image_U(w)$ of $w$ in $\cube{X}(U)$ from each $x_U$, to show the
question is equivalent 
to the existence
of an $x_\emptyset - w \in \cube{X}(\emptyset)$ mapping to
each coherent system of elements $x_U - \Image_U(w)$. 

Given a coherent system of elements $x_U$ in an $n$-cube, let $w$ be
the image of $x_{\Set{n}}$ in $\cube{X}(\emptyset)$ using the section
map $\cube{X}(\Set{n}) \rightarrow \cube{X}(\emptyset)$.
Define $z_U = x_U - \Image_U(w)$, noting that when $\Set{n}\subset V$, we
have $\Image_V(w) = x_W$, so $z_W = 0$.
By the preceding paragraph, the surjectivity that we are trying to establish
is equivalent to the existence of a
$z_\emptyset$ mapping to each coherent collection $z_U$.

If $n=1$, then $\lim_{U\in\Power_0(\mathbf{1})} \cube{X}(U) = X(\Set{1})$, 
so the section map
$\cube{X}(\Set{1}) \rightarrow \cube{X}(\emptyset)$ produces a
$z_\emptyset$ mapping to $z_{\Set{1}}$, as desired.

If $n>1$, then we proceed by induction, assuming the lemma is true for
smaller $n$.
Taking the fiber of $\cube{X}$ in the direction of $\Set{n}$, we have
an $(n-1)$-cube 
$$\cube{Y}(U) := \fib \left(
\cube{X}(U) \rightarrow \cube{X}(U \cup \Set{n}) \right).$$
The cube $\cube{Y}$ satisfies the hypothesis of the lemma because
taking fibers preserves compatible sections.
Notice that for $\Set{n} \not\subset U$, the element $z_U$ passes to
the fiber, since it maps to $z_{U\cup \Set{n}} = 0$. 

Now $\cube{Y}$ is an $(n-1)$-cube, so by induction, the map from
$\cube{Y}(\emptyset)$  to $\lim \cube{Y}(U)$ is surjective. That is,
there exists a $y\in \cube{Y}(\emptyset)$ with $\Image_U(y) = z_U$. 
Mapping $y$ to $z \in \cube{X}(\emptyset)$ gives an element $z$ with
$\Image_U(z) = z_U$ for $U \subset \Set{1,\ldots,n-1}$. As above, if
$\Set{n} \subset U$, then $z_U = 0$, so $\Image_U(z) = z_U$ in this
case as well. Therefore, we have produced an element $z$ mapping to
each coherent collection of elements $z_U$, as desired.
\end{proof}

\begin{corollary}
\label{cor:discrete-groups-cartesian}
Let $\cube{X}$ be an $n$-cube ($n\ge 1$) of discrete groups with
compatible sections to all structure maps
 (\ref{hypothesis:compatible-sections}). If the total fiber of 
$\cube{X}$ is zero, then $\cube{X}$ is Cartesian.
\end{corollary}
\begin{proof}
Recall that a $T$-cube $\cube{X}$ is Cartesian when
$\cube{X}(\emptyset) \rightarrow \holim_{U \not= \emptyset}
\cube{X}(U)$ is an equivalence. The hypotheses of the existence of
section maps means that all of the structure maps are fibrations.
This means that the homotopy inverse limit is equivalent to the strict
inverse limit, so we need only show that 
$\cube{X}(\emptyset) \xrightarrow{\cong} \lim_{U \not= \emptyset} \cube{X}(U).$
Corollary~\ref{cor:discrete-groups-cartesian} shows that the map is
surjective. The total fiber of the cube is equivalent to the fiber of
this map, so if the fiber is zero then we have a surjective map of
groups with zero kernel; that is, an isomorphism.
\end{proof}

The combination of Lemma~\ref{lem:groups-surj-X0-lim} and
Lemma~\ref{lem:connected-cartesian} 
allow us to conclude that for all group-valued functors $\Perp F
\simeq 0$ means the $\Perp$-cube is Cartesian.

\begin{lemma}
\label{lem:groups-cartesian}
Let $F$ be a functor satisfying Hypothesis~\ref{hypothesis-2} (group
values) on coproducts of $X$. 
If $\Perp_n F(X) \simeq 0$, then the cube defining $\Perp_n F(X)$ is
Cartesian. 
\end{lemma}
\begin{proof}
Decompose $F$ as the fibration
$$ \widehat{F} \rightarrow F \rightarrow \pi_0 F ,$$
where $\widehat{F}$ is the connected component of the basepoint of $F$.
Let $\cube{X}$ be the $(n+1)$-cube used in
Definition~\ref{def:n-additive} to define $\Perp_n F(X)$.

In general, $X \rightarrow \pi_0 X$ has a section given by
chosing a point in each component. Using the canonical sections to the
structure maps in the cube $\cube{X}$, we may make compatible choices
for these maps $\pi_0 \cube{X}(U) \rightarrow \cube{X}(U)$ so that
they assemble to a map of cubes $\pi_0 \cube{X} \rightarrow
\cube{X}$. We will use this in two places later in the proof.

We claim that if $\Perp_n F(X)$ is contractible, then $\Perp_n
\widehat{F}(X)$ and $\Perp_n \pi_0 F(X)$ are both contractible as
well.  
The section map of cubes $\pi_0 F\cube{X} \rightarrow F\cube{X}$
produces a section $\Perp \pi_0 F(X) \rightarrow \Perp F(X)$.
This shows
that in the long exact sequence on homotopy associated to the
fibration
$$ \Perp\widehat{F}(X) \rightarrow \Perp F(X) \rightarrow \Perp \pi_0 F(X),$$
there is a surjective map
$\pi_k \Perp F(X) \rightarrow \pi_k \Perp \pi_0 F(X)$, for all $k$. 
Since $\pi_k \Perp F(X) = 0$, this shows $\Perp \pi_0 F(X)$ is contractible. 
Similarly,
$\Perp \widehat{F}(X)$ is now the fiber in a quasifibration whose
total space and base space are contractible, so it is contractible as well.

The functor $\widehat{F}$ satisfies Hypothesis~\ref{hypothesis-1} on
coproducts of $X$, so by Lemma~\ref{lem:connected-cartesian},
the cube $\widehat{F}\cube{X}$ 
defining $\Perp_n \widehat{F}(X)$ is Cartesian. 
The functor $\pi_0 F$ is a functor to discrete groups, so
Corollary~\ref{cor:discrete-groups-cartesian} shows that the cube
$\pi_0 F \cube{X}$ defining $\Perp_n \pi_0 F(X)$ is Cartesian. This
shows that the left and right vertical arrows in the following diagram
are equivalences:
$$\xymatrix{
\widehat{F} \cube{X}(\emptyset)
\ar[r]
\ar[d]^{\simeq}
&
F\cube{X}(\emptyset)
\ar[r]
\ar[d]
&
\pi_0 F \cube{X}(\emptyset)
\ar[d]^{\simeq}
\\
{\holim_{\Power_0} \widehat{F}\cube{X}}
\ar[r]
&
{\holim_{\Power_0} F\cube{X}}
\ar[r]
&
{\holim_{\Power_0} \pi_0 F\cube{X}}
}$$
The top row is a fibration by the construction of $\widehat{F}$. 
The bottom row is also a fibration, since homotopy inverse limits
commute up to natural equivalence.

Consider the long exact
sequence on homotopy groups. From the construction of $\widehat{F}$,
we know that $\pi_0 \widehat{F} = 0$, so $\pi_0 \widehat{F}
\cube{X}(\emptyset) = 0$, 
and hence $\pi_0({\holim_{\Power_0} \widehat{F}\cube{X}})$ is also
$0$. 
Furthermore, $F\rightarrow \pi_0 F$ is an isomorphism on $\pi_0$. 

Using the above data in the long exact sequences from the horizontal
fibrations gives us the following diagram with exact rows:
$$\xymatrix{
0
\ar[r]
\ar[d]
&
\pi_0 F\cube{X}(\emptyset)
\ar[r]^{\cong}
\ar[d]
&
\pi_0 \left(\pi_0 F \cube{X}(\emptyset) \right)
\ar[d]^{\cong}
\\
0
\ar[r]
&
\pi_0 {\holim_{\Power_0} F\cube{X}}
\ar[r]
&
\pi_0{\holim_{\Power_0} \pi_0 F \cube{X}}
}$$
Either by inverting the isomorphisms and using the center vertical map
or by 
applying $\holim_{\Power_0}$
to the map of cubes $\pi_0 F\cube{X}\rightarrow F\cube{X}$,
we have a section to the bottom right map in the diagram, so it is
surjective. By exactness, we have a surjective map of groups (by our
hypothesis on $F$) that has kernel zero,
so it is an isomorphism. 
\end{proof}


%
%
\chapter{The Main Theorem}
\label{chap:main-theorem}
In this chapter, we establish the fundamental theorem that makes the
use of a cotriple workable for functors from spaces to spaces.

%
%
\section{Main Theorem: Statement And Outline}
\label{sec:main-theorem-outline}
The main theorem of this paper is the following.

\begin{theorem}
\label{thm:Pnd-from-perp}
\label{thm:main-theorem}
Let $F$ be a homotopy functor from pointed spaces to pointed spaces.
If $F$ has either connected values (Hypothesis~\ref{hypothesis-1}) or
group values (Hypothesis~\ref{hypothesis-2}) on coproducts of $X$,
then the following is a fibration sequence up to homotopy:
\begin{equation}
\label{eq:perp-fibration}
\realization{\Perp_{n+1}^{*+1} F(X)}
\rightarrow 
F(X)
\rightarrow
P_n^d F(X)
.
\end{equation}
\end{theorem}

To establish this theorem, we use induction on $n$, beginning with the
case $n=1$.  We further break down the induction into the cases where
$\Perp_{n+1} F(X) \simeq 0$ and $\Perp_{n+1} F(X) \not\simeq 0$. 

In Section~\ref{sec:perp-F-zero}, we treat the case when
$\Perp_{n+1} F(X) \simeq 0$. In this case, we show directly that the
fiber of the fibration sequence we obtain from induction,
$$
\realization{\Perp_{n}^{*+1} F(X)}
\rightarrow 
F(X)
\rightarrow
P_{n-1}^d F(X)
,
$$
is a homogeneous degree $n$ functor. This implies that $F(X) \simeq
P_n^d F(X)$ in this case.

In Section~\ref{sec:perp-F-nonzero}, we treat the case when
$\Perp_{n+1} F(X) \not\simeq 0$. In this case, we consider the
auxiliary diagram: 
$$\xymatrix{
A_F(X)
\ar[r]
\ar[d]
&
\realization{\Perp_{n+1}^{*+1} F(X)}
\ar[r]^{\epsilon}
\ar[d]
&
F(X)
\ar[d]
\\
P_n^d A_F(X)
\ar[r]
&
P_n^d \left( \realization{\Perp_{n+1}^{*+1} F(X)} \right)
\ar[r]
&
P_n^d F(X)
}$$
where $A_F$ is defined as the homotopy fiber of the map $\epsilon$ in
the top row, and the bottom row is shown to be a quasifibration as
well (Propositon~\ref{prop:perp-f-nonzero-fib-conn}
and~\ref{prop:perp-f-nonzero-fib-discrete}).  
We show that $\Perp_{n+1} A_F(X)
\simeq 0$ (Lemma~\ref{lem:perp-AF-zero}), 
and hence the case $\Perp_{n+1} F \simeq 0$ shows that
there is an equivalence of the fibers, so the square on the right is
Cartesian. Then it is not hard to establish that 
$P_n^d \left( \realization{\Perp_{n+1}^{*+1} F(X)}\right) \simeq 0$
(Lemma~\ref{lem:Pnd-perp-contractible}),
so that \eqref{eq:perp-fibration} is actually a quasifibration.


%
%
\section{Case: $\Perp F\simeq 0$}
\label{sec:perp-F-zero}
In this section, the goal is to establish that when the $(n+1)$-st
cross effect of $F$ vanishes, $F$ is equivalent to its $n$-additive
approximation, $P_n^d F$.

\begin{proposition}
\label{prop:perp-F-zero}
If $F$ has either connected values
(Hypothesis~\ref{hypothesis-1}) or 
group values (Hypothesis~\ref{hypothesis-2}) on coproducts of $X$, and
$\Perp_{n+1} F(X) \simeq 0$, then $F(X) \simeq P_n^d F(X)$.
\end{proposition}


\begin{remark}
Under our hypotheses, there is no difference between
$\Perp_{n+1} F(X)$ being contractible and the cube defining
$\Perp_{n+1} F(X)$ being Cartesian. (See
\S\ref{sec:perp-cube-cartesian}.) 
\end{remark}

\begin{proof}[Outline of Proposition~\ref{prop:perp-F-zero}]
We will approach this proposition using a ladder induction, depending
on previous cases of Proposition~\ref{prop:perp-F-zero} \emph{and}
Proposition~\ref{prop:perp-F-nonzero}. We begin with the classical
case $n=1$ (Corollary~\ref{cor:perp-2-zero-F-P1dF}). 
It turns out that this case is essentially the work 
of Segal~\cite{Segal:categories-and-cohomology-theories}: when
$\Perp_2 F(X) \simeq 0$, the fold map makes $F(X)$ into a (homotopy)
monoid, and $P_1^d F(X)$ is naturally equivalent to loops on the bar
construction of $F(X)$. 

We then consider $n>1$ and apply Proposition~\ref{prop:perp-F-nonzero},
considering the $n^{\text{th}}$
cross effect, $\Perp_n F(X)$, rather than the $(n+1)^{\text{st}}$
cross effect $\Perp_{n+1} F(X)$. This gives us the fibration:
$$
\realization{\Perp_{n}^{*+1} F(X)}
\rightarrow
F(X)
\rightarrow
P_{n-1}^d F(X)
.
$$
The problem is now to show that the fiber is a homogeneous degree $n$
functor, $D_n^d F(X)$. 

Since $\Perp_{n+1} F(X)$ vanishes, we know that as a functor of $n$ variables,
$\Perp_n F(X) = cr_n F(X,\ldots, X)$ has zero 
second cross effect in each variable. Hence we may apply the $n=1$ case
to identify $\Perp_n F(X)$ as the infinite loop space of a functor to 
connective spectra, 
$\mathbf{\BoldPerp_n F} (X)$ (Corollary~\ref{cor:perp-f-spectrum}). 
Then actually
the entire 
simplicial space $\realization{\Perp_{n}^{*+1} F(X)}$ is
an infinite loop space, 
$\LoopInfty \realization{\mathbf{\BoldPerp_{n}^{*+1} F}(X)}$.
From this point (in Lemma~\ref{lem:perp-homotopy-orbits-spectrum}),
we identify the simplicial spectrum as the homotopy
orbits of $\mathbf{\BoldPerp^{*+1}_{n} F}$,
$$ 
\realization{ \mathbf{\BoldPerp^{*+1}_{n} F} (X)} 
\simeq 
\mathbf{\BoldPerp_n F}(X) _{h \Sigma_n}
,
$$
which is a homogeneous functor, $D_n \mathbf{F}(X)$.
It remains to show that $F(X) \simeq P_n^d F(X)$. Before
evaluation at $S^0$, our fibration can be written
$$ D_n (L F_X) \rightarrow LF_X \rightarrow P_{n-1} (L F_X) ,$$
and in this form, the base and the fiber are $n$-excisive. Modulo the
ever-present technical problem of $\pi_0$, the property of
$n$-excision is closed under extensions, so $LF_X$ is
$n$-excisive, so $LF_X \simeq P_n LF_X$. Evaluating at $S^0$ 
gives $F(X) \simeq P_n^d F(X)$.
\end{proof}

\subsection{$\Perp F$-cube Cartesian}
\label{sec:perp-cube-cartesian}

If $F$ has connected values (\ref{hypothesis-1})
(or group values (\ref{hypothesis-2})), on
coproducts of $X$, then Lemma~\ref{lem:connected-cartesian}
(respectively, Lemma~\ref{lem:groups-cartesian}) shows that
$\Perp_{n+1} F(X) \simeq 0$ implies that the cube defining
$\Perp_{n+1} F(X)$ is Cartesian, so henceforth we may assume that the
$\Perp$-cubes we are dealing with are Cartesian.

\subsection{Additivization And The Bar Construction}

If $F$ is nice enough and $\Perp_2 F(X) \simeq 0$, then actually
$F(X)$ is a (homotopy) monoid ($\Gamma$-space), and $F(X)$ is
equivalent to loops on the bar construction on $F(X)$.

\begin{lemma}
\label{lem:perp-2-zero-Segal}
Suppose $F$ is a reduced functor that has either connected values
(Hypothesis~\ref{hypothesis-1}) or 
group values (Hypothesis~\ref{hypothesis-2}) on coproducts of $X$. If
$\Perp_{2} F(X) \simeq 0$, then one can form the bar construction
$BF(X) = LF_X(S^1)$ on $F(X)$, and $BF(X)$ is connected,  and 
 $F(X) \simeq \Omega B F(X)$.
\end{lemma}
\begin{proof}
  Under these hypotheses, if $\Perp_2 F(X) \simeq 0$, then the cube
$$ 
\xymatrix{
F(X \vee X)
\ar[r]
\ar[d]
&
F(X)
\ar[d]
\\
F(X)
\ar[r]
& 
F(0) \simeq 0
}$$
is Cartesian, so $F(X\vee X) \simeq F(X) \times F(X)$. Furthermore,
with this identification, 
the map $\epsilon: F(X \vee X) \rightarrow F(X)$ induces a the
structure of a homotopy monoid on $F(X)$. This is the setting in which
Segal's theory of $\Gamma$-spaces applies
\cite[Proposition~1.4]{Segal:categories-and-cohomology-theories},
showing that $F(X) \simeq \Omega B F(X)$ and $BF(X) \simeq
\realization{F(X \wedge S^1_\cdot)}$ is connected.
\end{proof}

\begin{corollary}
\label{cor:perp-f-spectrum}
If $F$ has either connected values
(Hypothesis~\ref{hypothesis-1}) or 
group values (Hypothesis~\ref{hypothesis-2}) on coproducts of $X$, and
$\Perp_{n+1} F(X) \simeq 0$, then, as a symmetric functor of $n$
variables, $\Perp_n F(X, \ldots, X) $ is the infinite loop space of a
symmetric functor to connective spectra $\mathbf{\BoldPerp_n F}
(X,\ldots,X)$: 
$$\Perp_n F(X, \ldots, X) \simeq \LoopInfty \left( \mathbf{\BoldPerp_n
    F} (X,\ldots,X) \right).$$
\end{corollary}
\begin{proof}
  As usual, under these hypotheses, the cube defining $\Perp_{n+1}
  F(X)$ is Cartesian (\S\ref{sec:perp-cube-cartesian}). 
  By
  Corollary~\ref{cor:perp-n+1-perp-2-perp-n}, the functor of $n$
  variables $\Perp_n F(X,\ldots,X)$ is also additive in each variable,
  and $\Perp_n F(X,\ldots,X)$ is always reduced in each variable, 
  so Lemma~\ref{lem:perp-2-zero-Segal} gives $\Perp_n
  F(X,\ldots,X)$ as the first space of a connective $\Omega$-spectrum. 
\end{proof}

\begin{corollary}
\label{cor:perp-2-zero-F-P1dF}
If $F$ is a reduced functor that has either connected values
(Hypothesis~\ref{hypothesis-1}) or 
group values (Hypothesis~\ref{hypothesis-2}) on coproducts of $X$, and
$\Perp_{2} F(X) \simeq 0$, then $F(X) \simeq P_1^d F(X)$.
\end{corollary}
\begin{proof}
Lemma~\ref{lem:perp-2-zero-Segal} shows that $F(X) \simeq \Omega B
F(X)$. Continuing the argument from that lemma,
consider one iteration of the functor
$T_1^d F(X) \simeq \Omega \realization{ F( S^1_\cdot \wedge X ) }$ 
used to build $P_1^d = \colim_n \left( T_1^n \right)^d$. Actually,
since $(T_1^n)^d$ is not necessarily $(T_1^d)^n$, we want to show
$ T_1 LF_X \simeq LF_X$ in order to show that $P_1^d F(X) \simeq
F(X)$. 
Consider an arbitrary space $Y= \realization{Y_\cdot}$; we show
$ T_1 LF_X (Y) \simeq LF_X(Y)$.
\begin{align*}
  T_1 L F_X(Y) 
  &\simeq 
  \Omega \realization{ F( X \wedge ( S^1 \wedge Y) _\cdot ) }
  \\
  &\simeq
  \Omega \realization{  \realization{ F( X \wedge S^1_\cdot \wedge Y_\cdot ) }}
  \\ 
  &\simeq
  \Omega \realization{\realization{ [k] \mapsto F( X \wedge Y_\cdot
      \wedge S^1_k ) }}
  \\
  &\simeq
  \Omega \realization{\realization{ [k] \mapsto F(\bigvee^k X \wedge Y_\cdot
) }}
  \\
  \intertext{And, since $X \wedge Y_\cdot$ is a coproduct of copies
    of $X$, the cross effect $\Perp_2 F(X\wedge Y_\cdot) \simeq 0$, so
  this is }
  &\simeq
  \Omega \realization{\realization{ [k] \mapsto \prod^k F(X \wedge Y_\cdot) }}
  \\
  &\simeq
  \Omega \realization{ B F( X \wedge Y_\cdot ) }
  \\
  &\simeq
  \Omega B \realization{ F( X \wedge Y_\cdot ) }
  \\
  &\simeq
  \Omega B LF_X(Y)
  \\
  &\simeq
  LF_X(Y)
\end{align*}
So in fact, $T_1 LF_X(Y) \simeq LF_X(Y)$, and hence 
$P_1^d F(X) = \colim (T_1^k)^d F(X) = \colim T_1^k (L F_X)(S^0) \simeq
LF_X(S^0) \simeq F(X)$, as desired.
\end{proof}

\subsection{Iterated Cross Effects Produce Homogeneous Functors}

In this section, we write $\Perp$ instead of $\Perp_n$, for
convenience. We write $\Sigma_n^+$ for the space $\Sigma_n$ with a
disjoint basepoint added.

Let $\mathbf{H}(X_1, \ldots, X_n)$ be a symmetric functor from spaces
to spectra, and suppose $\mathbf{H}$ is additive in each variable
separately. In practice, such an $\mathbf{H}$ will arise as 
$\mathbf{\BoldPerp F}$ from Corollary~\ref{cor:perp-f-spectrum}.

Using the additivity in each variable, we can identify $\Perp_n
\mathbf{H}$ with $\Sigma_n^+ \wedge \mathbf{H}$. Here we have
$X_1=\cdots=X_n$, but label them differently to be able to see the
action of the symmetric group more clearly.
\begin{align*}
  \Perp \mathbf{H}(X_1,\ldots,X_n) 
  &\simeq
  \prod_{\alpha\in\Sigma_n} 
  \mathbf{H}(X_{\alpha(1)},\ldots, X_{\alpha(n)})
  \\
  &\simeq 
  \bigvee_{\alpha\in\Sigma_n}
  \mathbf{H}(X_{\alpha(1)},\ldots, X_{\alpha(n)})
  \\
  &\simeq
  \Sigma_n^+ \wedge \mathbf{H}(X_{1},\ldots, X_{n}).
\end{align*}
The second equivalence is the stable equivalence of finite coproducts
and products, and the third equivalence is given by the map 
$$
  \Sigma_n^+ \wedge \mathbf{H}(X_{1},\ldots, X_{n})
\rightarrow
  \bigvee_{\alpha\in\Sigma_n}
  \mathbf{H}(X_{\alpha(1)},\ldots, X_{\alpha(n)})
$$
sending
$\sigma \wedge x$ to $x$ in the coproduct indexed by $\sigma$. To
identify $x \in \mathbf{H}(X_1,\ldots,X_n)$ with 
$x \in \mathbf{H}(X_{\alpha(1)},\ldots, X_{\alpha(n)})$, we use the
map induced by $X_{\alpha(i)} \cong X_i$ in each variable.

The identification
\begin{equation}
  \label{eq:perp-is-sigma-n}
  \Perp \mathbf{H}(X_1,\ldots,X_n) 
  \simeq
  \Sigma_n^+ \wedge \mathbf{H}(X_{1},\ldots, X_{n})
\end{equation}
can be made equivariant with respect to the action of $\Sigma_n$ on
both $\mathbf{H}$ and $\Perp(-)$ in the following way. The action
induced by permuting the inputs of $\Perp \mathbf{H}$  (\emph{i.e.},
from the fact that $\Perp(-)$ is a symmetric functor) is sent to
multiplication on the $\Sigma_n$ factor. The action on $\Perp
\mathbf{H}$ induced by the $\Sigma_n$ action on $\mathbf{H}$ is sent
to the same action on $\mathbf{H}$ on the other side.

Under this model, the map $\epsilon: \Perp \mathbf{H} \rightarrow
\mathbf{H}$ is given by $\sigma \wedge x \mapsto x$.


We are now in a position to understand $\Perp^* \mathbf{\BoldPerp F}$.
Applying~\eqref{eq:perp-is-sigma-n} repeatedly at each level, we have 
$$ \Perp^k \mathbf{\BoldPerp F}(X)
\simeq
\underbrace{
\Sigma_n^+ \wedge \cdots \wedge \Sigma_n^+
}_{\text{$k$ factors}} \wedge 
\mathbf{\BoldPerp F}(X)
.$$
Recall that the face maps from dimension $n$ to $n-1$ are given by
$d_i = \Perp^i \epsilon \Perp^{n-i}$. In dimension $k$, the face map
$d_k = \epsilon \Perp^k$ just drops the first element:
$$ d_k (g_k \wedge \cdots \wedge g_1 \wedge y) = g_{k-1} \wedge \cdots
\wedge g_1 \wedge y .$$
To compute the others, note that for any $f$, the map $\Perp(f)$ is
equivariant with respect to the action of $\Sigma_n$ on $\Perp$ (by
permuting inputs), so in particular $\Perp(\epsilon): \Perp(\Perp F)
\rightarrow \Perp F$ is equivariant with respect to the action on of
$\Sigma_n$ on the leftmost $\Perp$, so 
\begin{align*}
  \Perp \epsilon (g \wedge y ) 
&=
\Perp\epsilon(g * (1 \wedge y) )
\\
&= g * \Perp\epsilon(1 \wedge y) 
\\
&= g * y,
\end{align*}
where the last follows since the degeneracy $\delta: \Perp F
\rightarrow \Perp^2 F$ given by $\delta(y) = 1\wedge y$ is a section
to the face map $\Perp\epsilon$. This argument shows that all of the
face maps $d_j$ with $0\le j<k$ are given by multiplying $g_{j+1}$ by
the next coordinate to the right (either $g_j$ if $j>0$ or $y$ if
$j=0$).

This is a standard model for 
$E \Sigma_n^+ \wedge_{\Sigma_n} \mathbf{\BoldPerp F}(X)$, 
so we have shown that the simplicial spectrum  built by iterating the
cross effects computes the homotopy orbits of $\mathbf{\BoldPerp
  F}(X)$. That is,
$$\realization{\Perp^* \mathbf{\BoldPerp F}(X)}
\simeq 
 \mathbf{\BoldPerp F}(X) \wedge_{\Sigma_n} E \Sigma_n^+
.
$$

We have just established the following lemma:
\begin{lemma}
\label{lem:perp-homotopy-orbits-spectrum}
If $F$ has either connected values
(Hypothesis~\ref{hypothesis-1}) or 
group values (Hypothesis~\ref{hypothesis-2}) on coproducts of $X$, and
$\Perp_{n+1} F(X) \simeq 0$, then
$$\realization{\Perp^* \mathbf{\BoldPerp F}(X)}
\simeq 
 \mathbf{\BoldPerp F}(X) \wedge_{\Sigma_n} E \Sigma_n^+
,
$$
where $\mathbf{\BoldPerp F}$ denotes the lift to spectra of $\Perp F$,
as in Corollary~\ref{cor:perp-f-spectrum}.
$\qed$
\end{lemma}

This ends the establishment of the results required for
Proposition~\ref{prop:perp-F-zero}. 







\subsection{Proof Of Proposition~\ref{prop:perp-F-zero}}

Recall the statement of the proposition we are to prove:
\begin{proposition}
[Proposition~\ref{prop:perp-F-zero}]
If $F$ has either connected values
(Hypothesis~\ref{hypothesis-1}) or 
group values (Hypothesis~\ref{hypothesis-2}) on coproducts of $X$, and
$\Perp_{n+1} F(X) \simeq 0$, then $F(X) \simeq P_n^d F(X)$.
\end{proposition}
\begin{proof}[Proof of Proposition~\ref{prop:perp-F-zero}]
The proof has four steps: first, settle the trivial case when $n=0$;
second, reduce to the case of a reduced functor, so 
we can assume $F(0) = 0$. Third, establish the base case $n=1$. 
Finally, finish the proof using an induction
that involves Proposition~\ref{prop:perp-F-nonzero} for lower values
of $n$. 

When $n=0$, the hypothesis of the proposition is that $\Perp_1 F(X)
\simeq 0$. But $\Perp_1 F(X)$ is the fiber of the map $F(X)\rightarrow
F(0)$, and this map must be surjective because it has a section
induced by $0\rightarrow X$. Therefore, $\Perp_1 F(X)\simeq 0$ means
$F(X)\simeq F(0)$, so $F$ is a constant functor. But this is what
$P_0^d F(X)$ is as well.

To reduce to the case of a reduced functor, consider the fibration
$$ \widetilde{F} \rightarrow F \rightarrow F_0 ,$$
where $F_0(X) = F(0)$ and $\widetilde{F}$ is defined to be the
homotopy fiber of the map of the map $F(X)\rightarrow
F(0)$. Under our hypotheses,
Corollary~\ref{cor:Pnd-group-connected-base} shows that $P_n^d$
preserves this fibration, and that it remains surjective on
$\pi_0$. If we show that $\widetilde{F}(X)  \simeq P_n^d
\widetilde{F}(X)$, then we will have two fibration sequences
$$\xymatrix{
\widetilde{F}(X)
\ar[r]
\ar[d]
&
{F}(X)
\ar[r]
\ar[d]
&
{F_0}(X)
\ar[d]
\\
P_n^d
\widetilde{F}(X)
\ar[r]
&
P_n^d
{F}(X)
\ar[r]
&
P_n^d
{F_0}(X)
}$$
with the map on fibers and bases an equivalence. Since $\pi_0 F$ is
a group (or $0$, which is a group), and the maps to the bases are
surjective on $\pi_0$, the Five Lemma applies to show that the map on
the total spaces is an equivalence.

For the rest of this proof, we will assume that $F$ is reduced. 
Given a reduced functor, Corollary~\ref{cor:perp-2-zero-F-P1dF} shows
that $F(X) \simeq P_1^d F(X)$, so that establishes the true base case
in our induction.

Finally, when $n>1$ we apply Proposition~\ref{prop:perp-F-nonzero}
with one smaller $n$ to produce a fibration sequence:
$$
\realization{\Perp_{n}^{*+1} F(X)}
\rightarrow
F(X)
\rightarrow
P_{n-1}^d F(X)
,
$$
where the map $F(X)\rightarrow P_{n-1}^d F(X)$ is surjective on
$\pi_0$. 
We now show that the fiber here is equivalent to $D_n^d F(X)$, and
then show that this allows us to deduce that the total space must be
equivalent to $P_n^d F(X)$.

Using Corollary~\ref{cor:perp-f-spectrum}, we have
$$\realization{\Perp_{n}^{*+1} F(X)}
\simeq
\realization{ \LoopInfty \Perp_{n}^{*+1} \mathbf{F}(X)}  .
$$
Using
  Theorem~\ref{thm:loop-infty-commutes-with-connective-spectra}, the
  right hand side is 
  equivalent to
$$
\LoopInfty \realization{ \Perp_{n}^{*+1} \mathbf{F}(X)} .
$$
{We can then apply
  Lemma~\ref{lem:perp-homotopy-orbits-spectrum} to deduce that this
  is equivalent to }
$$
\LoopInfty \left( \Perp_n \mathbf{F}(X) \wedge_{\Sigma_n} E\Sigma_n^+\right) .
$$
Corollary~\ref{cor:perp-f-spectrum} shows that when $\Perp_{n+1} F(X)
\simeq 0$, the functor $\Perp_n F(X)$ is actually the infinite loop
space of a spectrum with $\Sigma_n$ action, $\Perp_{n}
\mathbf{F}(X)$.
As in Lemma~\ref{lem:perp-2-zero-Segal}, the spectrum $\Perp_{n}
\mathbf{F}(X)$ arises from using the structure maps from suspending
the left Kan extension in any coordinate, for example:
$$ cr_n (L F)_{(X,\ldots,X)} (S^0, \ldots, S^0) \xrightarrow{\simeq}
\Omega cr_n(L F)_{(X,\ldots,X)} (S^1, S^0, \ldots, S^0) .
$$ 
This is exactly the spectrum defined to be the
derivative spectrum $\partial^{(n)} LF_X (\basept)$
(Definition~\ref{def:derivative-of-F}):
$$
\LoopInfty \left( \partial^{(n)} LF_X(\basept) \wedge_{\Sigma_n}
  E\Sigma_n^+ \right) .
$$
{We can identify $E\Sigma_n^+$ with $S^0$ if we change the
  strict orbits to homotopy orbits, giving:}
$$
\LoopInfty \left( \partial^{(n)} LF_X(\basept)  \wedge_{h\Sigma_n}
  S^0 \right) .
$$
{Since $S^0 = (S^0)^{\wedge n}$, we can identify this as the
  form of $D_n^d 
  F(X)$ given in Equation~\eqref{eq:def-of-deriv} in
  Section~\ref{sec:Pnd-preserves-connectivity}, so we have shown that}
$$
\realization{\Perp_{n}^{*+1} F(X)}
\simeq
D_n^d F(X)
.
$$


It remains to check that the map 
$$\realization{\Perp_{n}^{*+1} F(X)} \rightarrow F(X)$$
actually induces an isomorphism after applying $D_n^d$. 
This happens because in order to compute the derivative spectrum, one
stabilizes $\Perp_n$, but when $\Perp_n$ is applied to the map above,
it becomes an equivalence using a standard ``extra degeneracy''
argument. 

In order to apply $D_n^d(-)(X)$, we apply $D_n$ after $L(-)_X$.
Evaluation at $S^0$ would give $D_n^d F(X)$, but we will show the stronger
result that actually the natural transformation of functors $D_n
L(\epsilon_F)_X$ is an equivalence. 

With the aim of applying the ``extra degeneracy'' argument, we begin
by establishing that after applying $L(-)_X$, the map we are
considering is actually equivalent to augmentation map $\epsilon_{L
  F_X}$. This involves verifying that all of the squares in the
diagram below commute. A summary of each step follows the diagram.

To aid the reader in understanding the various transformations, we
consider $F$ as a trivial simplicial functor (this makes it possible
to distinguish between $L \realization{F}$ and $\realization{LF}$). 
$$\xymatrix{
L \realization{\Perp_{n}^{*+1} F(-)}_X
\ar[r]^{L \realization{\epsilon_F}_X}
\ar[d]^{=}
&
L \realization{F}_X
\ar[d]^{=}
\\
L \realization{\Perp_{n}^{*+1} F_X(-)}
\ar[r]^{L \realization{\epsilon_{F_X}}}
&
L \realization{F_X}
\\
\realization{L \Perp_{n}^{*+1} F_X(-)}
\ar[r]^{\realization{L \epsilon_{F_X}}}
\ar[u]^{\simeq}
\ar[d]^{\simeq}
&
\realization{L F_X}
\ar[u]^{\simeq}
\ar[d]^{=}
\\
\realization{\Perp_{n}^{*+1} L F_X(-)}
\ar[r]^{\realization{\epsilon_{L F_X}}}
&
\realization{L F_X}
}$$
The first transformation applied is that by expanding the definition,
one can check that $(\Perp^k F)_X = \Perp^k (F_X)$.
The second transformation is the map from $\realization{L(-)}
\rightarrow L\realization{-}$, which is really the commuting of
realization in two different directions.
The third transformation is the commuting of $\Perp$ and $L$, from 
Lemma~\ref{lem:perp-commutes-with-realization-v2}.

Now to compute the coefficient spectra of both sides, we apply
$\Perp_n$. Since $\Perp_n$ commutes with realizations of functors that
satisfy Hypothesis~\ref{hypothesis-1} or Hypothesis~\ref{hypothesis-2}
(Lemma~\ref{lem:perp-commutes-with-realization}),
Lemma~\ref{lem:extra-perp-contract} (the ``extra degeneracy''
argument) shows that the map 
$$
\Perp_n \realization{\Perp_{n}^{*+1} L F_X}
\xrightarrow{\simeq}
\Perp_n \realization{L F_X}
$$
is an equivalence. Stabilizing this (as in
Definition~\ref{def:derivative-of-F}) produces an equivalence on the
derivatives, so 
$$
D_n^d \realization{\Perp_{n}^{*+1} F}
\xrightarrow{\simeq}
D_n^d \realization{F}
$$
is an equivalence (using Equation~\eqref{eq:def-of-deriv} to compute
$D_n$ given the derivative), and in particular,
$$
D_n^d \realization{\Perp_{n}^{*+1} F(X)}
\simeq
D_n^d \realization{F(X)}
$$
via the augmentation map, as desired. 

Applying $P_n^d$ to our original fibration gives us a commutative diagram:
$$\xymatrix{
\realization{\Perp_{n}^{*+1} F(X)}
\ar[r]
\ar[d]^{\simeq}
&
F(X)
\ar[r]
\ar[d]
& P_{n-1}^d F(X)
\ar[d]^{=}
\\
D_n^d F(X) 
\ar[r]
&
P_n^d F(X)
\ar[r]
& P_{n-1}^d F(X)
}$$

In particular, this tells us that we may regard the top row as
the fibration:
$$ D_n^d F(X) \rightarrow F(X) \rightarrow P_{n-1}^d F(X), $$ 
and recall that this is surjective on $\pi_0$. 
Alternatively, before evaluation at $S^0$, this is:
$$ D_n (L F_X) \rightarrow LF_X \rightarrow P_{n-1} (L F_X) .$$
The base and the fiber of this fibration are $n$-excisive, and 
$F$ on coproducts of $X$ (and hence also the functor $LF_X$) 
is either connected (Hypothesis~\ref{hypothesis-1}) or
has $\pi_0$ a group  (Hypothesis~\ref{hypothesis-2}),  so
Lemma~\ref{lem:excisive-up-fibration} shows $LF_X$ is $n$-excisive.
Therefore, $L F_X \simeq P_n (L F_X)$, and then evaluation at $S^0$
gives $F(X) \simeq P_n^d F(X)$. 
\end{proof}


%
%
\section{Case: $\Perp F \not\simeq 0$}
\label{sec:perp-F-nonzero}
In this section, the goal is to establish the other side of the
``ladder induction'' for Theorem~\ref{thm:main-theorem}.

\begin{proposition}
\label{prop:perp-F-nonzero}
If $F$
 has either connected values (Hypothesis~\ref{hypothesis-1}) or
group values (Hypothesis~\ref{hypothesis-2}) on coproducts of $X$, and
$\Perp_{n+1} F(X) \not\simeq 0$,
then the following is a fibration sequence up to homotopy:
\begin{equation}
\label{eq:perp-fibration-perp-nonzero}
\realization{\Perp_{n+1}^{*+1} F(X)}
\xrightarrow{\epsilon}
F(X)
\rightarrow
P_n^d F(X)
.
\end{equation}
Furthermore, 
$$\pi_0 P_n^d F(X) \cong \coker \left( 
\pi_0 \realization{\Perp_{n+1}^{*+1} F(X)}
\rightarrow
\pi_0 F(X)
 \right),
$$
where $\coker$ is the cokernel in the category of groups.
\end{proposition}

We begin with a definition for the homotopy fiber of the map $\epsilon$.
\begin{definition}[$A_F$]
\label{def:AF}
Define the functor $A_F(X)$ to be the homotopy fiber in the
quasifibration:
\begin{equation}
  \label{eq:AF-def}
A_F(X)
\rightarrow
\realization{\Perp_{n+1}^{*+1} F(X)} 
\rightarrow 
F(X).
\end{equation}
\end{definition}

We now outline the proof of this result, 
essentially as sketched in Section~\ref{sec:main-theorem-outline}.
We
consider the 
auxiliary diagram: 
$$\xymatrix{
A_F(X)
\ar[r]
\ar[d]
&
\realization{\Perp_{n+1}^{*+1} F(X)}
\ar[r]^{\epsilon}
\ar[d]
&
F(X)
\ar[d]
\\
P_n^d A_F(X)
\ar[r]
&
P_n^d \left( \realization{\Perp_{n+1}^{*+1} F(X)} \right)
\ar[r]
&
P_n^d F(X)
}$$
We show that the bottom row is a quasifibration
(Propositions~\ref{prop:perp-f-nonzero-fib-conn}
and~\ref{prop:perp-f-nonzero-fib-discrete}).   
We further show that $\Perp_{n+1} A_F(X)
\simeq 0$ (Lemma~\ref{lem:perp-AF-zero}), 
and hence the case $\Perp_{n+1} F \simeq 0$ shows that
there is an equivalence of the fibers, so the square on the right is
Cartesian. It is not hard to establish that 
$P_n^d \left( \realization{\Perp_{n+1}^{*+1} F(X)}\right) \simeq 0$
(Lemma~\ref{lem:Pnd-perp-contractible}),
so that \eqref{eq:perp-fibration} is actually a quasifibration.
Especially in Propositions~\ref{prop:perp-f-nonzero-fib-conn}
and~\ref{prop:perp-f-nonzero-fib-discrete},
attention to path components is needed to let us make the statement about
surjectivity on $\pi_0$. Section~\ref{sec:proof-of-perp-nonzero}
assembles all of the ingredients into a proof of the result.


\subsection{Functors To Groups: $\Perp G^{ab} = 0$}

This section establishes a technical result that is needed in the
proof of Proposition~\ref{prop:perp-f-nonzero-fib-discrete}, where we
consider functors to discrete groups.

Let $G$ be a functor from spaces to groups. Generally these functors
will arise as lifts of functors from spaces to spaces. For example,
$\pi_0$ of loops on a space, $F(X) = \pi_0 \Omega X$, lifts to a
group-valued functor $G(X)$ by using concatenation of loops for the
group operation. 
In this
section, we establish that $\Perp$ preserves short exact sequences of
groups, and use this to show that the ``abelianization'' of $G$ has
vanishing ($n^{\text{th}}$) cross effect. 

Our motivation for following notation comes from
the case when the source and target category under consideration are
both the category of groups and the functor $G$ is the identity $G(H)=H$,
we have $\Perp_2 G(H) = [H*1,1*H]$, and the image of $\Perp_2 G(H)$ in
$G(H)$ is the first derived subgroup of $H$.
The cokernel of the map $\Perp_2 G(H) \rightarrow G(H)$ is the
abelianization, $H^{ab}$. See Section~\ref{sec:lower-central-series}
for a detailed explanation.

\begin{definition}
\label{def:Gprime-Gab}
Given an $n>0$ and a functor $G$ to groups,
define $G'_n := \Image (\epsilon: \Perp_{n+1} G
\rightarrow G)$ and $G^{ab}_n := \coker (\epsilon)$.  Usually, the $n$ is
clear from context, and we will abbreviate these $G'$ and $G^{ab}$.
\end{definition}
There is a short exact (fibration) sequence of groups
\begin{equation}
\label{eq:ses-Gprime-G-Gab}
G'(X) \rightarrow G(X) \rightarrow G^{ab}(X) ,
\end{equation}
and this sequence is surjective on $\pi_0$ (\emph{i.e.}, right exact). 

To ease the reader's concern about potentially modding out by a
subgroup that is not normal, we note that $G'(X)$ is normal in $G(X)$.
\begin{lemma}
$G'(X)$ is a normal subgroup of $G(X)$.
\end{lemma}
\begin{proof}
$\Perp_{n+1} G(X)$ is constructed as the kernel of a map, so it is a
normal subgroup of $G(\bigvee^{n+1} X)$. The map $G(\bigvee^{n+1} X)
\rightarrow G(X)$ is surjective, so normal subgroups correspond. That
is, the normal subgroup $\Perp_{n+1} G(X)$ of $G(\bigvee^{n+1} X)$
maps to a normal subgroup $G'(X)$ in $G(X)$.
\end{proof}

\begin{definition}
\label{def:perp-strict}
  Let $\PerpStrict$ denote the functor identical to $\Perp$,
  except with the construction made using strict inverse limits or
  fibers, rather than homotopy inverse limits or homotopy fibers.
\end{definition}
Note that there is a natural transformation
$\PerpStrict\rightarrow\Perp$ arising from the canonical map from the
strict inverse limit to the homotopy inverse limit.

We use
the functor $\PerpStrict$ in what follows because it is easier to see
that a certain functor has $\PerpStrict F(X)$ strictly $0$ than to show
that the simplicial space $\realization{\Perp^{*+1} F(X)}$ is
contractible. 

\begin{lemma}
\label{lem:Perp-F-simeq-Perpstrict-F}
  If $F$ takes values in discrete groups on coproducts of $X$, then
  $\PerpStrict F(X) \simeq \Perp F(X)$.
\end{lemma}
\begin{proof}
The cube defining $\Perp$ (and $\PerpStrict$; it is the same cube) has
compatible section maps to all structure maps. Since all vertices are
discrete, this means that all of the structure maps are fibrations.
Taking iterated fibers or homotopy fibers, this implies that the
homotopy fiber is equivalent to the strict fiber. 
\end{proof}

\begin{lemma}
\label{lem:Perpstrict-epsilon}
  If $F$ takes values in discrete groups on coproducts of $X$, then the
  image of $\epsilon^{\text{strict}}: \PerpStrict F(X) \rightarrow
  F(X)$ is the same as the image of $\epsilon: \Perp F(X) \rightarrow
  F(X)$. 
\end{lemma}
\begin{proof}
  The functor $\Perp F$ lies between $\PerpStrict F$ and the $F$, so
  we need to make sure that the image $\Perp F \rightarrow F$ is not
  larger than the image $\PerpStrict F\rightarrow F$. From
  Lemma~\ref{lem:Perp-F-simeq-Perpstrict-F}, we know that $\PerpStrict
  F(X) \simeq \Perp F(X)$. The space $F(X)$ is discrete, so $F(X)
  \cong \pi_0 F(X)$. We have the following commutative diagram:
$$\xymatrix{
\PerpStrict F(X) 
\ar[r]
\ar[d]
&
\Perp F(X) 
\ar[r]
\ar[d]
&
 F(X)
\ar[d]^{\cong}
\\
\pi_0 \PerpStrict F(X) 
\ar[r]^{\cong}
&
\pi_0 \Perp F(X) 
\ar[r]
&
\pi_0 F(X)
}$$
The bottom row shows that the images of $\pi_0 \PerpStrict F(X)$ and
$\pi_0 \Perp F(X)$ in $\pi_0 F(X)$ coincide. The fact that the right hand
vertical map is an isomorphism then implies that the images of
$\PerpStrict F(X)$ and $\Perp F(X)$ in $F(X)$ coincide.
\end{proof}
\begin{corollary}
\label{cor:Perpstrict-Gab}
If $G$ takes values in discrete groups on coproducts of $X$, then the
functors $G'(X)$ and $G^{ab}(X)$ of Definition~\ref{def:Gprime-Gab}
can be defined using $\epsilon$ or
$\epsilon^{\text{strict}}$.
\end{corollary}
\begin{proof}
Lemma~\ref{lem:Perpstrict-epsilon} shows that the images of $\epsilon$
and $\epsilon^{\text{strict}}$ are the same, and the image is all that
is used to define $G'$ and $G^{ab}$.
\end{proof}

\begin{lemma}
\label{lem:perp-preserves-connectivity}
  Suppose that a natural transformation $F\rightarrow G$ is
  $k$-connected when evaluated on coproducts of $X$. Then 
  $\Perp 
 F(X)
  \rightarrow 
   \Perp 
 G(X)$ 
  is $k$-connected.
\end{lemma}
\begin{proof}
Briefly, this follows because the cubes defining $\Perp F (X)$ and
$\Perp G (X)$ have compatible sections to all structure maps
(\ref{hypothesis:compatible-sections}), so, as in
Lemma~\ref{lem:pi-k-perp-is-perp-pi-k}, 
taking the fiber in any direction
produces split short exact sequences on homotopy. In this case, a
$k$-connected map on the total space and base of the (quasi-)fibration
results in a $k$-connected map on the fiber. The compatible sections
pass to compatible sections on the fibers, so this argument shows
that the map on total fibers is $k$-connected.
\end{proof}

\begin{corollary}
\label{cor:perpstrict-preserves-surj}
  Suppose that a natural transformation $F\rightarrow G$ is
  surjective when evaluated on coproducts of $X$. Then 
  $\PerpStrict
 F(X)
  \rightarrow 
   \PerpStrict
 G(X)$ 
  is surjective.
\end{corollary}
\begin{proof}
  When taking strict fibers, an argument almost identical to that in
  Lemma~\ref{lem:perp-preserves-connectivity} shows that surjectivity
  is preserved.
\end{proof}


\begin{lemma}
  \label{lem:PerpStrict-preserves-fib-cofib}
  Let $A\rightarrow B\rightarrow C$ be a natural short sequence
  of functors to discrete groups that is
  a fiber (cofiber) sequence on coproducts of $X$. 
  Then 
  $\PerpStrict A \rightarrow \PerpStrict B
  \rightarrow \PerpStrict C$ is a fiber (cofiber) sequence
  when evaluated at $X$.
\end{lemma}
\begin{proof}
  The construction of $\PerpStrict$ involves taking fibers,
    so it certainly preserves fiber sequences.

  A cofiber sequence of discrete groups is a fiber sequence of the
  underlying sets with the additional property that it is surjective (on
  $\pi_0$). Since $\PerpStrict$ preserves fiber sequences and
  Corollary~\ref{cor:perpstrict-preserves-surj} shows that
  $\PerpStrict$ preserves connectivity (in particular, surjectivity),
  $\PerpStrict$ preserves cofiber sequences as well.
\end{proof}

\begin{lemma}
\label{lem:perp-Gab-0}
If $F$ takes values in discrete groups on coproducts of $X$, then  
  with $G'$ and $G^{ab}$ as in Definition~\ref{def:Gprime-Gab},
  $\Perp G^{ab}(X) \simeq 0$.
\end{lemma}
\begin{proof}
By Lemma~\ref{lem:Perp-F-simeq-Perpstrict-F},  $\Perp G^{ab} \simeq \PerpStrict 
G^{ab}$, so it suffices to show $\PerpStrict G^{ab} = 0$. Recall from
Corollary~\ref{cor:Perpstrict-Gab} that we can build $G'$ and $G^{ab}$ using
$\PerpStrict$ instead of $\Perp$.

The map $\epsilon^{\text{strict}}: \PerpStrict G \rightarrow G$
factors through $G'(X)$ since the following diagram commutes and
$G'(X)$ is the image of $\epsilon$ in $G(X)$.
$$\xymatrix{
{\PerpStrict G(X)}
\ar[r]
\ar[rd]_{\epsilon^{\text{strict}}}
&
\Perp G(X)
\ar[d]^{\epsilon}
\\
&
G(X)
}$$
That gives us the following factorization of
$\epsilon^{\text{strict}}$:
$$\xymatrix{
\PerpStrict G(X)
\ar[dr]^{\epsilon^{\text{strict}}}
\ar[d]
&
\\
G'(X)
\ar[r]
&
G(X)
}$$


Applying $\PerpStrict$ to this factorization, we have
the factorization:
$$ \xymatrix{ 
\left(\PerpStrict\right)^2 G
\ar[d]
\ar[dr]
&
\\
\PerpStrict G'
\ar[r]
&
\PerpStrict G
}
$$
The map $\left(\PerpStrict\right)^2 G \rightarrow \PerpStrict G$ is
surjective, since it has a 
section map $\delta$ (from the diagonal). Therefore, the map $\PerpStrict G'
\rightarrow \PerpStrict G$ must be surjective. 

From Lemma~\ref{lem:PerpStrict-preserves-fib-cofib},
applying $\PerpStrict$ to the cofiber sequence
\eqref{eq:ses-Gprime-G-Gab} results in the short exact sequence 
$$ \PerpStrict G' \rightarrow 
\PerpStrict G
\rightarrow
\PerpStrict G^{ab}
.$$
We have just shown that the first map is surjective, so the cofiber
$\PerpStrict G^{ab}$ is zero.
%
\end{proof}


\subsection{If $m<n$, Then $P_m^d \realization{\Perp_n^{*+1} F} \simeq 0$} 

This section establishes the relatively easy fact that for $\Perp_n$, 
the part of the Goodwillie tower below degree $n$ is trivial.

\begin{lemma}
\label{lem:Pnd-perp-contractible}
  Let $R(X_1,\ldots,X_n) =  \realization{cr_n \left( \Perp_{n}^{*}
      F\right) (X_1, \ldots, X_n)}$ be a functor of $n$ variables. 
  Define the diagonal of such a functor to be the functor of one
  variable given by
  $(\diag R)(X) =
  R(X,\ldots,X)$.
  Then $P_m^d \left( \diag R \right) (X) \simeq 0$ for $0\le m<n$.
\end{lemma}
\begin{proof}
  Goodwillie's Lemma~2.1 \cite[Lemma~2.1]{Cal3} shows that if
  $H(X_1,\ldots, X_n)$ is a functor of $n$ variables that is
  contractible whenever some $X_i$ is contractible (this is called a
  ``multi-reduced'' functor), then $P_m (\diag
  H) \simeq 0$ for $0\le m<n$. 



  Recall that we write $L$ for the left Kan extension of a
  functor along the full subcategory of spaces generated by coproducts
  of zero dimensional spheres: $\bigvee^k S^0$. By analogy with the
  notation $F_X(Y) := F(X \wedge Y)$, let us define $R_X(Y_1, \ldots,
  Y_n) := R(X \wedge Y_1, \ldots, X\wedge Y_n)$.

  To use Goodwillie's lemma, we need to show that the computation of
  $P_m^d (\diag R)$ results in computing $P_m$ of the diagonal of a
  multi-reduced functor. 
  This is an easy computation:
  \begin{align*}
    P_m^d \left( \diag R \right) (X)
    &=
    P_m L [ ( \diag R )_X ] (S^0)
    \\
    &=
    P_m L [ \diag ( R_X ) ] (S^0)
    \\
    &=
    P_m \diag [ L^{(n)} ( R_X ) ] (S^0),
  \end{align*}
  where $L^{(n)} R$ indicates $L$ applied to each of the $n$ inputs to
  $R$ separately. It remains to check that $L^{(n)} R_X$ is
  contractible when any of its inputs is contractible. Since we use
  the homotopy invariant left Kan extension, if $Y$ is contractible,
  then $L^{(n)}R_X(Y, \ldots) \simeq L^{(n)}R_X(0,\ldots)$, and the latter is
  equivalent to $L^{(n-1)}R_X(0,\ldots)$ (removing the $L$ in the
  first variable), because the Kan extension is
  equivalent to the original functor on finite sets. Now 
  all $n^{\text{th}}$ cross effects have the property that they are 
  contractible if any of their inputs are contractible, so we are done.
\end{proof}

\subsection{The Functor $A_F$ Has No $n+1$ Cross Effect}

Having created the functor $A_F$ to be ``$F$ with the 
cross effect killed'', we now need to establish that $\Perp A_F \simeq
0$. The main issue is the commuting of the $\Perp$ and the
realization. 

The essence of the following lemma is that the cubes used to construct
$\Perp F$ are nice enough that we can compute the cross effects of
some particular simplicial functors levelwise. 
\begin{lemma}
\label{lem:perp-commutes-with-realization}
  Let $\Perp$ denote $\Perp_n$ for any fixed $n>0$. 
  If $F$
  satisfies Hypothesis~\ref{hypothesis-1} (connected values) or
  Hypothesis~\ref{hypothesis-2} (group values) on coproducts of $X$,
  then 
$$
\Perp  
\realization{\Perp^{*+1} F(X) }
\simeq 
\realization{\Perp^{*+2} F(X) }
.
$$
\end{lemma}
\begin{proof}
  If $F$ satisfies Hypothesis~\ref{hypothesis-1}, then $\Perp F$ also
  satisfies Hypothesis~\ref{hypothesis-1}, since $\Perp$ preserves the
  connectivity of the natural transformation from $F$ to the
  constant zero functor, by
  Lemma~\ref{lem:perp-preserves-connectivity}. Therefore we can
  compute the fibers in the $\Perp$-cube levelwise using Waldhausen's
  Lemma (Lemma~\ref{lem:Waldhausen-fibration-lemma}).
  If $F$ satisfies Hypothesis~\ref{hypothesis-2}, then we will use
  Theorem~\ref{thm:Bousfield-Friedlander} (Bousfield-Friedlander) 
  to produce the same
  result. Recall that we use the term ``$\Perp$-cube'' to denote the
  cube whose total (homotopy) fiber is $\Perp F$.
Since $F$ is a functor to groups, so is
  $\realization{\Perp^{*+1} F(X) }$, so each corner of the
  $\Perp$-cube satisfies the $\pi_*$-Kan condition. 
  Furthermore, all of the maps in the $\Perp$-cube have compatible
  sections (Lemma~\ref{lem:crn-cube-has-sections}), 
  so at each stage of taking iterated fibers all of the
  structure maps have sections. This gives us surjective maps of
  simplicial groups, so the induced maps on $\pi_0$ are
  fibrations. These two conditions are enough to apply
  Theorem~\ref{thm:Bousfield-Friedlander} to compute the fibers levelwise.
\end{proof}

\begin{lemma}
\label{lem:extra-perp-contract}
 $$
\realization{\Perp^{*+2} F(X)}
\simeq
\Perp F(X)
$$
\end{lemma}
\begin{proof}
  The degeneracy map $\delta: \Perp F\rightarrow \Perp^2 F$ shows that
  $\Perp F$ is the augmentation of $\realization{\Perp^{*+2} F(X)}$,
  so this lemma follows from 
  \cite[Exercise~8.4.6, p.~275]{Weibel:homological-algebra}.
\end{proof}

\begin{lemma}
\label{lem:perp-AF-zero}
Let $F$ be a functor satisfying Hypothesis~\ref{hypothesis-1} or
Hypothesis~\ref{hypothesis-2} on coproducts of $X$, let $\Perp$
denote $\Perp_n$ for some $n$, and let $A_F$ be the functor given in
Definition~\ref{def:AF}.
Then $A_F$ satisfies $\Perp A_F(X) \simeq 0$. 
\end{lemma}
\begin{proof}
Taking cross effects is a homotopy inverse limit construction, and
homotopy inverse limits commmute, so
\begin{align*}
\Perp A_F(X) 
&= 
\Perp \hofib \left(
\realization{\Perp^{*+1} F(X)}
\rightarrow 
F(X)
\right)
\\
&\simeq 
\hofib \left(
\Perp \realization{\Perp^{*+1} F(X)}
\rightarrow 
\Perp F(X)
\right) .
\\ \intertext{Which, by
  Lemma~\ref{lem:perp-commutes-with-realization}, is}
&\simeq
\hofib \left(
\realization{\Perp \Perp^{*+1} F(X)}
\rightarrow 
\Perp F(X)
\right),
\\
\intertext{and by Lemma~\ref{lem:extra-perp-contract}, this is}
&\simeq
\hofib \left(
\Perp F(X)
\rightarrow 
\Perp F(X)
\right)
\\
&\simeq 
0
.
\end{align*}
\end{proof}

\subsection{$P_n^d$ Preserves $A_F$ Fibration}

This section establishes that $P_n^d$ actually produces a fibration
when applied to the fibration defining $A_F$. The case of $F$ taking
values in discrete groups is the most important. Here we actually only
show that this is true for $F$ taking values in discrete groups or
connected spaces; that is all that is needed to establish the main
result that we want.

The results in this section also contain a statement about the map
from $F \rightarrow P_n^d F$, because in the case of $F$ taking
values in discrete groups, the proof that this map is surjective on
$\pi_0$ uses the same technical details that the proof that we get a 
fibration.


\begin{proposition}
\label{prop:perp-f-nonzero-fib-conn}
  If $F$ satisfies Hypothesis~\ref{hypothesis-1} (connected values) on
  coproducts of $X$,  then the following is a quasifibration:
\begin{equation*}
P_n^d A_F(X) 
\rightarrow 
P_n^d \left( \realization{\Perp_{n+1}^{*+1} F(X)}  \right)
\rightarrow 
P_n^d F(X)
,
\end{equation*}
and furthermore the map  $F(X)\rightarrow P_n^d F(X)$ is (trivially)
surjective on $\pi_0$. 
\end{proposition}
\begin{proof}
If $F$ has connected values on coproducts of $X$, then
$$ A_F(X) 
\rightarrow 
\realization{\Perp_{n+1}^{*+1} F(X)}
\rightarrow 
F(X)
$$ 
is a fibration over a connected base. Therefore, by 
Corollary~\ref{cor:Pnd-group-connected-base}, applying $P_n^d$ yields a
fibration, so 
Equation~\eqref{eq:Pnd-still-fibration} is a fibration.

To establish surjectivity of the map $\pi_0 F(X) \rightarrow \pi_0
P_n^d F(X)$, it suffices to show $\pi_0 P_n^d
F(X) = 0$. Consider the natural transformation $\eta$ from the zero functor
$0$ to $F$. Since $F$ has connected values on coproducts of $X$,
the map $\eta: 0 \rightarrow F$ is $0$-connected on coproducts
of $X$. Applying Theorem~\ref{thm:Pnd-preserves-connectivity} shows
that $0 \simeq P_n^d(0) \rightarrow P_n^d F(X)$ is $0$-connected
as well. Hence $\pi_0 P_n^d F(X) = 0$.
\end{proof}

To remind the reader that the functor takes values in
discrete groups in the next proposition, we use the letter $G$ (for
group) to denote the functor, instead of the usual $F$.

\begin{proposition}
\label{prop:perp-f-nonzero-fib-discrete}
  If $G$ takes values in discrete groups on coproducts of $X$ (so in
  particular $G$ satisfies Hypothesis~\ref{hypothesis-2}), 
  then the following is a quasifibration:
\begin{equation}
\label{eq:Pnd-still-fibration}
P_n^d A_G(X) 
\rightarrow 
P_n^d \left( \realization{\Perp_{n+1}^{*+1} G(X)}  \right)
\rightarrow 
P_n^d G(X)
.
\end{equation}
\end{proposition}
\begin{proof}
Replacing the base $G$ in the definition of $A_G$
(Equation~\eqref{eq:AF-def}) with $G'$ from 
Definition~\ref{def:Gprime-Gab}, 
we have the fibration sequence
\begin{equation}
\label{eq:fib-Gprime-base}
A_G(X) 
\rightarrow 
\realization{\Perp_{n+1}^{*+1} G(X)}
\rightarrow
G'(X),
\end{equation}
and this sequence is surjective on $\pi_0$. 
The hypotheses of Corollary~\ref{cor:Pnd-group-connected-base} 
are satisfied by the
sequences in \eqref{eq:ses-Gprime-G-Gab} and \eqref{eq:fib-Gprime-base},
so applying $P_n^d$ both are fibration sequences whose maps to the
base spaces are surjective on $\pi_0$:
\begin{gather}
\label{eq:Pnd-Gprime-G-Gab}
P_n^d G'(X) 
\rightarrow 
P_n^d G (X) 
\rightarrow 
P_n^d G^{ab}(X)
\\
\label{eq:Pnd-AG-seq}
P_n^d (A_G)(X) 
\rightarrow 
P_n^d (\realization{\Perp_{n+1}^{*+1} G(-)}) (X)
\rightarrow
P_n^d G' (X).
\end{gather}

The aim now is to show that \eqref{eq:Pnd-AG-seq} remains a fibration
when the base $P_n^d G'(X)$ is replaced by $P_n^d G(X)$.
From Lemma~\ref{lem:perp-Gab-0}, $\Perp_{n+1} G^{ab}(X) \simeq 0$, so
Proposition~\ref{prop:perp-F-zero} shows that $P_n^d G^{ab}(X) \simeq 
G^{ab}(X)$, which is a discrete space. Then, using the long exact
sequence on homotopy, the fibration in \eqref{eq:Pnd-Gprime-G-Gab} gives
$P_n^d G'(X) \xrightarrow{\simeq} P_n^d G(X)$ except possibly on
$\pi_0$, where the map is injective. This is
enough to show that  changing the base in
\eqref{eq:Pnd-AG-seq} from $P_n^d G'(X)$ to $P_n^d G(X)$ still yields
a fibration. 
That is, \eqref{eq:Pnd-still-fibration} is a
fibration (but perhaps not surjective on $\pi_0$).
%
%
%
\end{proof}

\begin{proposition}
\label{prop:perp-f-nonzero-pi0-control}
If $G$ takes values in discrete groups on coproducts of $X$ (so in
particular $G$ satisfies Hypothesis~\ref{hypothesis-2}), then 
$$
\pi_0 
P_n^d G(X) 
\cong \coker 
\left( 
\pi_0 
\realization{\Perp^{*+1}_{n+1} G(X)}
\rightarrow
\pi_0 G(X)
\right)
,$$
where $\coker$ denotes the cokernel in the category of groups.
\end{proposition}
\begin{proof}
As in the preceding
Proposition~\ref{prop:perp-f-nonzero-fib-discrete}, we have the following
fibration sequence that is surjective on $\pi_0$:
$$P_n^d (A_G)(X) 
\rightarrow 
P_n^d (\realization{\Perp_{n+1}^{*+1} G(-)}) (X)
\rightarrow
P_n^d G' (X).
$$
Lemma~\ref{lem:Pnd-perp-contractible} shows that the total space
in this fibration is contractible, and the map to
the base is surjective on $\pi_0$, so  $\pi_0 P_n^d G'(X) = 0$.

Also following Proposition~\ref{prop:perp-f-nonzero-fib-discrete}, 
we have the
following diagram in which the horizonal rows are fibrations that are
surjective on $\pi_0$:
$$\xymatrix{
G'(X)
\ar[r]\ar[d]
&
G(X)
\ar[r]\ar[d]
&
G^{ab}(X)
\ar[d]^{\simeq}
\\
P_n^d G'(X)
\ar[r]
&
P_n^d G(X)
\ar[r]
&
P_n^d G^{ab}(X)
}$$
Since $\pi_0 P_n^d G'(X) = 0$, the long exact sequence for the bottom
fibration implies that $\pi_0 P_n^d G(X) \cong \pi_0 P_n^d G^{ab}$.
The right hand vertical map is an equivalence, again as noted in the
preceding proposition, 
using  Lemma~\ref{lem:perp-Gab-0}  and Proposition~\ref{prop:perp-F-zero}.
%
Combining these, we have
\begin{align*}
\pi_0 P_n^d G(X) 
&\cong 
\pi_0 P_n^d G^{ab}(X)
\\
&\cong
\pi_0 G^{ab}(X)
,
\end{align*}
so we need to establish that $\pi_0 G^{ab}(X)$ is the cokernel of the group
map 
$$\pi_0 \epsilon: \pi_0 \realization{\Perp^{*+1}_{n+1} G(X)} \rightarrow \pi_0 G(X).$$
Both
$G^{ab}(X)$ and $G(X)$ are discrete, so it suffices to establish that
$G^{ab}(X)$ is the (strict) cokernel of the group map 
$\epsilon: \realization{\Perp^{*+1}_{n+1} G(X)} \rightarrow G(X)$. The image of this
map is $G'(X)$, by the definition of $G'(X)$ (\ref{def:Gprime-Gab})
and the fact that maps from higher iterates of $\Perp^*_{n+1} G(X)$ to
$G(X)$ factor through $\Perp_{n+1} G(X)$. By definition
(\ref{def:Gprime-Gab}), $G^{ab}(X)$ is the cokernel of the inclusion
of $G'(X)$ in $G(X)$, so $G^{ab}(X)$ is the cokernel of the map $\epsilon$, as
desired. 
\end{proof}
\subsection{Proof Of Proposition~\ref{prop:perp-F-nonzero}}
\label{sec:proof-of-perp-nonzero}

\begin{proposition}[Proposition~\ref{prop:perp-F-nonzero}]
If $F$ 
has either connected values (Hypothesis~\ref{hypothesis-1})
or  group values (Hypothesis~\ref{hypothesis-2}) on coproducts
of $X$, and $\Perp_{n+1} F(X) \not\simeq 0$,
then the following is a fibration sequence up to homotopy:
\begin{equation}
\label{eq:perp-fibration-perp-nonzero-v2}
\realization{\Perp_{n+1}^{*+1} F(X)}
\xrightarrow{\epsilon}
F(X)
\rightarrow
P_n^d F(X)
.
\end{equation}
Furthermore, 
$$\pi_0 P_n^d F(X) \cong \coker \left( 
\pi_0 \realization{\Perp_{n+1}^{*+1} F(X)}
\rightarrow
\pi_0 F(X)
 \right),
$$
where $\coker$ is the cokernel in the category of groups.
\end{proposition}
\begin{proof}
First, suppose that $F(X)$ takes either connected values or discrete group
values on coproducts of $X$. 
Consider the auxiliary diagram created by applying $P_n^d$ to the
fibration sequence defining $A_F(X)$:
$$\xymatrix{
A_F(X)
\ar[r]
\ar[d]
&
\realization{\Perp_{n+1}^{*+1} F(X)}
\ar[r]^{\epsilon}
\ar[d]
&
F(X)
\ar[d]
\\
P_n^d A_F(X)
\ar[r]
&
P_n^d \left( \realization{\Perp_{n+1}^{*+1} F(X)} \right)
\ar[r]
&
P_n^d F(X)
}$$
Proposition~\ref{prop:perp-f-nonzero-fib-conn} (in the case of
connected values) or 
Proposition~\ref{prop:perp-f-nonzero-fib-discrete} (in the case of
discrete group values)
shows that the bottom row is
a quasifibration. 
Proposition~\ref{prop:perp-f-nonzero-fib-conn} (connected values)
or 
Proposition~\ref{prop:perp-f-nonzero-pi0-control} (discrete group values)
imply that the map $F(X) \rightarrow P_n^d F(X)$ surjective
on $\pi_0$. Lemma~\ref{lem:perp-AF-zero} shows that 
$\Perp_{n+1} A_F(X) \simeq 0$, so that
Proposition~\ref{prop:perp-F-zero} gives $A_F(X) \simeq P_n^d
A_F(X)$, and hence the square on the right is homotopy Cartesian. 
Lemma~\ref{lem:Pnd-perp-contractible} shows that 
$P_n^d \left( \realization{\Perp_{n+1}^{*+1} F(X)}\right) \simeq 0$,
so this square being Cartesian is equivalent to
\eqref{eq:perp-fibration-perp-nonzero-v2} 
being a quasifibration, as we wanted to establish.

We can reduce the general problem when $F$ satisfies Hypothesis~\ref{hypothesis-2}
 to the cases of connected and
discrete group values that we have already considered by examining
the fibration
$$ \widehat{F}(X) \rightarrow F (X) \rightarrow \pi_0 F(X),$$
where $\widehat{F}(X)$ is the component of the basepoint in $F(X)$.
This gives rise to the following square:
$$\xymatrix{
\realization{\Perp_{n+1}^{*+1} \widehat{F}(X)}
\ar[r]
\ar[d]
&
\widehat{F}
\ar[r]
\ar[d]
&
P_n^d \widehat{F}
\ar[d]
\\
\realization{\Perp_{n+1}^{*+1} {F}(X)}
\ar[r]
\ar[d]
&
{F}
\ar[r]
\ar[d]
&
P_n^d {F}
\ar[d]
\\
\realization{\Perp_{n+1}^{*+1} \pi_0{F}(X)}
\ar[r]
&
\pi_0{F}
\ar[r]
&
P_n^d \pi_0 {F}
}$$
\begin{itemize}
\item The middle column is a fibration by construction, and surjective
on $\pi_0$ for the same reason.

\item The functors $F$ and $\pi_0 F$ are group-valued, and $F$
  surjects onto $\pi_0 F$, so
  Corollary~\ref{cor:Pnd-group-connected-base} shows that 
  the right column is a fibration and surjective on $\pi_0$.

\item The left column is a realization of a levelwise fibration. 
  Since $F$ satisfies Hypothesis~\ref{hypothesis-2},
  ${\Perp_{n+1}^{*+1} {F}(X)}$ and 
  ${\Perp_{n+1}^{*+1} \pi_0{F}(X)}$ are simplicial groups, and hence
  satisfy the the $\pi_*$-Kan
  condition (\ref{def:pi-star-Kan}). 

  From Lemma~\ref{lem:pi-k-perp-is-perp-pi-k}, we know that $\pi_k
  \Perp^{*+1}_{n+1} F(X) \cong  \Perp^{*+1}_{n+1} \pi_k F(X)$, so the map $$\pi_0
  \Perp^{*+1}_{n+1} F(X) \rightarrow \pi_0 \Perp^{*+1}_{n+1} \pi_0 F(X)$$ is an
  isomorphism of simplicial sets, and hence a fibration.
  This is the necessary data to apply 
  Theorem~\ref{thm:Bousfield-Friedlander} (Bousfield-Friedlander)
  to conclude that the realization is a fibration.

  The realization of a levelwise $0$-connected map is $0$-connected, 
  so the left column is also a surjection on $\pi_0$.

\item The functor $\widehat{F}$ has connected values, so the
  top row is a fibration and surjective on $\pi_0$ by
  Proposition~\ref{prop:perp-f-nonzero-fib-conn}. 

\item The composition of the maps in the middle row is null homotopic,
  since the composition factors through 
  $P_n^d \left(\realization{\Perp_{n+1}^{*+1} F(X)} \right)$, which is
  contractible by Lemma~\ref{lem:Pnd-perp-contractible}.

\item The functor $\pi_0 F$ takes values in discrete groups, so the
bottom row is a fibration and surjective on $\pi_0$ by
Proposition~\ref{prop:perp-f-nonzero-fib-discrete}. 
\end{itemize}
We can then use the $3 \times 3$ lemma for fibrations to show that the
middle row (\emph{i.e.}, \eqref{eq:Pnd-still-fibration}) is a
fibration and surjective on $\pi_0$.

The statement about $\pi_0$ is trivial in the connected case;
$\pi_0$ of every space in the top row is zero (which is trivially a
group). This implies that the vertical arrows connecting the second
and third rows are $\pi_0$-isomorphisms, so the statement about
$\pi_0$ follows from Proposition~\ref{prop:perp-f-nonzero-pi0-control}.
\end{proof}


%
%
\chapter{Consequences of the Main Theorem}
\label{chap:main-consequences}
In this chapter, we explore two of the main consequences of the Main
Theorem (\ref{thm:main-theorem}). The first is the existence of a
spectral sequence that can be used to compute $P_n^d F$, and the
second is a way of understanding the $n$-excisive approximation to the
identity functor as a derived functor of the $n^{\text{th}}$ derived
subgroup functor. 

\section{Spectral Sequence}

Bousfield and Friedlander show that under suitable conditions, given a
simplicial space $X_\cdot$, there is a spectral sequence for
calculating the homotopy groups of the realization of a simplicial
space $\pi_* \realization{X_\cdot}$ from $\pi_* X_\cdot$. We will use
their notation: let $\pi_n^v X$ (homotopy in the vertical direction)
denote the simplicial set $[k] \mapsto \pi_n X_k$, and 
let $\pi_m^h \pi_n^v X$ (homotopy in the horizontal direction) denote
$\pi_m \realization{[k] \mapsto \pi_n X_k}$.

\begin{theorem}[{\cite[Theorem~B.5]{Bousfield-Friedlander:Gamma-Spaces}}]
\label{thm:BF-spectral-sequence}
If $X_\cdot$ is a simplicial space that satisfies the $\pi_*$-Kan
condition (\ref{def:pi-star-Kan}), then there is a spectral sequence
with $E^2_{p,q} = \pi_p^h \pi_q^v X$ 
converging to $\pi_{p+q} \realization{X_\cdot}$.
\end{theorem}

We want to use this result to produce a spectral sequence to calculate
the homotopy groups of the homotopy fiber of a map $\realization{X_\cdot}
\rightarrow Y$.

\begin{corollary}
\label{cor:offset-BF-ss}
Let $X_\cdot$ be a simplicial space that satisfies the $\pi_*$-Kan
condition (\ref{def:pi-star-Kan}), and suppose there is a simplicial
map to a space $Y$ regarded as a trivial simplicial space. Then there
is a spectral sequence converging to $\pi_{s+t} \hofib \left(
  \realization{X} \rightarrow Y \right)$ whose $E^2$ term is:
\begin{align*}
  E^2_{s,t} &= \pi_s^h \pi_t^v X  && \text{$s\ge 1$} \\
  E^2_{0,t} &= 
  \ker \left( \pi_0^h \pi_t^v X \rightarrow \pi_t Y \right)
  && \text{$s=0$}
  \\
  E^2_{-1,t} &= 
  \coker \left( \pi_0^h \pi_t^v X \rightarrow \pi_t Y \right)
  && \text{$s=-1$, $t\ge 1$}
\end{align*}
\end{corollary}
\begin{proof}
  As in the proof of Theorem~\ref{thm:BF-spectral-sequence}, let $P_t$
  be the Postnikov tower functor with 
  $$ \pi_j P_t Y \cong 
  \begin{cases}
    \pi_j Y & \text{$j\le t$} \\
    0 &  \text{$j > t$}
  \end{cases}
  $$
  Let $F_t$ denote the homotopy fiber of the canonical map $P_t
  \rightarrow P_{t-1}$, so $F_t Y \simeq K(\pi_t Y,t)$ is an
  Eilenberg-MacLane space. Create the following diagram, letting $A_t$
  be the homotopy fiber of 
  $$\realization{[k] \mapsto P_t X_k}
  \rightarrow \realization{[k] \mapsto P_t Y} = P_t Y,$$
 and $W_t = \hofib \left(A_t \rightarrow A_{t-1}\right)$:
$$\xymatrix{
  W_t
  \ar[r] \ar[d]
  &
  \realization{F_t X_\cdot}
  \ar[r] \ar[d]
  &
  F_t Y
  \ar[d]
  \\
  A_t
  \ar[r] \ar[d]
  &
  \realization{P_t X_\cdot}
  \ar[r] \ar[d]
  &
  P_t Y
  \ar[d]
  \\
  A_{t-1}
  \ar[r]
  &
  \realization{P_{t-1} X_\cdot}
  \ar[r]
  &
  P_{t-1} Y
  }$$
  Note that the middle column is a fibration because $P_t$ preserves
  the $\pi_*$-Kan condition, so
  Theorem~\ref{thm:Bousfield-Friedlander} applies.
  Let $Z$ denote the homotopy fiber of the map $\realization{X_\cdot}
  \rightarrow Y$. 
  The long exact sequence on homotopy, combined with the fact that
  the natural transformation from the identity to $P_t$ is an
  isomorphism on homotopy groups in dimensions $\le t$, shows that
  $\pi_j A_t \cong  \pi_j P_t Z$ for $j< t$. Hence the map $Z
  \rightarrow A_t$ is at least $(t-1)$-connected. Since the
  connectivity of this map increases with $t$, the spectral sequence
  derived from the tower of fibrations $\Set{A_t}$ converges to the
  $\pi_* Z$.

  To form this spectral sequence, we need to identify 
  $\pi_{s+t} W_t$. Using the long exact sequence from the top row of
  the diagram, we have:
  $$
  \pi_{s+t} W_t 
  \cong
  \begin{cases}
    \pi_{s+t} \realization{F_t X_\cdot} & \text{$s \ge 1$} 
    \\
    \ker\left( \pi_{t} \realization{F_t X_\cdot}  \rightarrow \pi_t Y \right)
    & \text{$s = 0$} 
    \\
    \coker\left( \pi_{t} \realization{F_t X_\cdot}  \rightarrow \pi_t Y \right)
    & \text{$s = -1$, $t\ge 1$} 
  \end{cases}
  $$
  The appearance of $s=-1$ occurs because if the map
  $\realization{X_\cdot} \rightarrow Y$ is not surjective on $\pi_t$,
  the fiber has a homotopy group in one dimension lower. In this case,
  $t\ge 1$, since failure to be surjective on $\pi_0$ is not visible
  in the fiber.  There are no terms with $s\le -2$ since both
  $\realization{F_t X_\cdot}$ and $F_t Y$ have no homotopy below
  dimension $t$.
  Then, exactly as in 
  \cite[p.~123]{Bousfield-Friedlander:Gamma-Spaces}, we can identify
  $\pi_{s+t} \realization{F_t X_\cdot}$ with $\pi_s^h \pi_t^v X$. This
  gives us the desired result.
\end{proof}


We can apply this corollary in the setting of
Theorem~\ref{thm:main-theorem} to produce the following result:

\begin{theorem}
\label{thm:additive-spectral-sequence}
Let $F$ be a homotopy functor from spaces to spaces that takes 
connected values (Hypothesis~\ref{hypothesis-1}) or group values
(Hypothesis~\ref{hypothesis-2}) on coproducts of $X$. Then there is a
spectral sequence beginning with the $E^2$ page given below and
converging to $\pi_{s+t} P_n^d F(X)$. 
\begin{align*}
  E^2_{s,t} &= \pi_{s-1}^h \pi_t^v \Perp^{*+1}_{n+1} F(X) && \text{$s\ge 2$} \\
  E^2_{1,t} &= 
  \ker \left( \pi_0^h \pi_t^v \Perp^{*+1}_{n+1} F(X)\rightarrow \pi_t F(X) \right)
  && \text{$s=1$}
  \\
  E^2_{0,t} &= 
  \coker \left( \pi_0^h \pi_t^v \Perp^{*+1}_{n+1} F(X) \rightarrow \pi_t F(X) \right)
  && \text{$s=0$}
\end{align*}
\end{theorem}
\begin{proof}
%
The main theorem (\ref{thm:main-theorem}) gives us a quasifibration:
\begin{equation*}
\realization{\Perp_{n+1}^{*+1} F(X)}
\rightarrow 
F(X)
\rightarrow
P_n^d F(X),
\end{equation*}
with $\pi_0 P_n^d$ the group cokernel of the map on $\pi_0$ from the
fiber to the total space.

The simplicial space $\Perp_{n+1}^{*+1} F(X)$ satisfies the
$\pi_*$-Kan condition (\ref{def:pi-star-Kan}) 
when $F(X)$ is
connected or has group values; that is, under
Hypothesis~\ref{hypothesis-1} or Hypothesis~\ref{hypothesis-2}, so 
the spectral sequence of Corollary~\ref{cor:offset-BF-ss}
can used and converges to $\pi_* \Omega P_n^d F(X)$.
Shifting the index $s$ by one gives a spectral sequence converging to
$\pi_{s+t-1} \Omega P_n^d F(X) \cong \pi_{s+t} P_n^d F(X)$, for
$s+t\ge 1$. The fact that $\pi_0$ is the cokernel of the map as
claimed (i.e., $\pi_0 P_n^d F(X) \cong E^2_{0,0}$) 
is established separately in the main theorem (\ref{thm:main-theorem}).
\end{proof}

We will use this spectral sequence extensively later to understand the
functors $P_n^d F$ from computations of the iterated cross effects.

\section{Lower Central Series}
\label{sec:lower-central-series}

In this section, we explain how the main theorem
(\ref{thm:main-theorem}) demonstrates a relationship between the
$n$-excisive approximations to the identity functor and the lower
central series of a simplicial group. We first recall the related
classical results of Curtis.

Recall that if $G$ is a group, the $r^{\text{th}}$ group in the lower
central series of $G$ is denoted $\Gamma_r G$. The group $\Gamma_r G$
is defined recursively, with $\Gamma_2 G = [G,G]$ being the 
derived subgroup of $G$, and $\Gamma_{r} G = [G,\Gamma_{r-1} G]$ for
$r > 2$. When $G$ is a simplicial group, $\Gamma_r G$ is defined to be
$\Gamma_r$ applied levelwise to $G$.

\begin{theorem}[{\cite[Theorem~1.4]{Curtis:some-relations-between-homotopy-and-homology}}]
If $G$ is a free simplicial group that is $n$-connected, $n\ge
0$, then for $r\ge 2$, the map $G \rightarrow G/\Gamma_r G$ is
$\lfloor{n+\log_2 r}\rfloor$-connected.
\end{theorem}
Actually, because $G$ is a free group, the map is an isomorphism on
homotopy in dimension $\lfloor{n+\log_2 r}\rfloor$ as well, but this
is not so important. 
Another way of stating this theorem is that the cofibration
(=fibration) sequence of (simplicial) groups
\begin{equation}
\label{eq:Curtis}
\Gamma_r G \rightarrow G \rightarrow G/\Gamma_r G
\end{equation}
is surjective on $\pi_0$ and the connectivity of the fiber is at least
$\lfloor{n+\log_2 r}\rfloor$.

The main consequence of this theorem is that Curtis is able to apply
it to Kan's loop group functor $G(X_\cdot)$, which is a simplicial
group that in dimension $k$ is the reduced free group on the
$(k+1)$-simplices of $X$ (reduced meaning that the generator
corresponding to the basepoint is identified with the identity). 
Let $X$ be a simply connected simplicial 
set, and let $G = G(X)$ be the free simplicial group resulting from
applying Kan's loop group functor.  The theorem then shows that the
lower central series filtration ``converges'', in the sense that $G
\simeq \lim_r G/\Gamma_r G$ because the spaces are connected and the
connectivity of the fiber of the map grows (slowly) to infinity. This
in turn means that $\Omega X$ can be analyzed by looking at the
quotients in the lower central series.


To analyze the functor $G(X)$ in our setting, let
$G(X)=\strictrealization{G(\Sing(X))}$,
where $\Sing(X)$ is the standard singular simplicial set of $X$ with
$\Sing(X)_k = \Hom(\Delta^k, X)$. For convenience, use the notation
$X_k=\Sing(X)_k$. To compute $\Perp_n G(X)$, we need to compute the
total homotopy fiber of a cube whose vertices are $G(\bigvee^k X)$,
for $0\le k\le n$. 
In general, the cube for any $\Perp_r G$ has compatible sections to every
structure map, and all of the vertices are simplicial groups (and
hence satisfy the $\pi_*$-Kan condition (\ref{def:pi-star-Kan})), so
by Theorem~\ref{thm:Bousfield-Friedlander} (Bousfield-Friedlander), we
can compute the homotopy fiber by taking fibers of the simplicial sets
(groups) levelwise.  
Furthermore, since everything is fibrant, we can use strict fibers
rather than homotopy fibers.
When working with groups, the ``strict fiber'' of a map is commonly
called the kernel, so we use that term henceforth.

In dimension $m$, the simplicial group $G(\bigvee^k X)_m$ is the
reduced free group generated by $\bigvee^k X_{m+1}$. The reduced
free group functor distributes over coproducts, so this is the free
product of $k$ copies of the free group on $X_{m+1}$.

The second cross effect $\Perp_2 G(X)$ is the (realization of the
levelwise) kernel of the map $G(X\vee X) \rightarrow G(X)\times G(X)$,
which by the preceding paragraph is the same as the kernel of the map
$p: G(X)*G(X) \rightarrow G(X)\times G(X)$. To distinguish between the
two copies of $G(X)$, we will use $g_i \in G(X)*1$ and $h_j \in
1*G(X)$. 

Define $\Gamma_2^{\text{ext}} G(X)$ to be the normal closure of the subgroup
$[G(X)*1, 1*G(X)]$ of $G(X)*G(X)$ generated by commutators of the form
$[g_i,h_j]$. We will to show that  $\Gamma_2^{\text{ext}} G(X)$ is the
kernel of the map $p$, which is $\Perp_2 G(X)$.
An element $g_1h_1g_2h_2\cdots g_wh_w \in G(X)*G(X)$ is in the
kernel of the map $p$ if the products $g_1\cdots g_w=e$ and $h_1\cdots
h_w = e$. 
All commutators that generate $\Gamma_2^{\text{ext}} G(X)$ are of this
form, and the kernel of a map is normal, so $\Gamma_2^{\text{ext}}
G(X)$ is contained in $\Perp_2 G(X)$. 
Now, commuting the $g$ and $h$ elements to place all of the $g$'s
adjacent and all of the $h$'s adjacent produces the product of
$(g_1\cdots g_w)(h_1\cdots h_w) = e$ with a bunch of commutators of
$g_i$ and $h_j$, including higher iterated commutators. We claim that
these higher iterated commutators are contained in
$\Gamma_2^{\text{ext}} G(X)$. Let $y\in
\Gamma_2^{\text{ext}} G(X)$ and let $g\in G*G$: the element $[y,g] =
y^{-1} g^{-1} y g$ is then in $\Gamma_2^{\text{ext}} G(X)$ because it
is normal. This shows that 
the kernel of the map $p$ is also contained in the subgroup
$\Gamma_2^{\text{ext}} G(X)$, so $\Perp_2 G(X) =
\Gamma_2^{\text{ext}} G(X)$.

We use the notation $\Gamma^{\text{ext}}_2 G(X)$ to denote a group in
the ``exterior lower central series'' since its image under the fold
map $\epsilon$ is the standard subgroup $\Gamma_2 G(X)$ in the lower
central series of $G(X)$. 
The higher cross effects are produced by iterating the second cross
effect (\ref{cor:crn-from-cr2}), so the same analysis shows that
$\Perp_r G(X)$ is the $r^{\text{th}}$ ``exterior'' derived subgroup
$\Gamma_r^{\text{ext}} G(X)$. 

We can now interpret the fibration from
Theorem~\ref{thm:main-theorem},
\begin{equation*}
\realization{\Perp_{r+1}^{*+1} F(X)}
\rightarrow 
F(X)
\rightarrow
P_r^d F(X)
,
\end{equation*}
in the case when the functor $F(X)$ is our functor $G(X)$, which is
essentially loops on a space. The functor $G$ requires $X_0 = \basept$
(so $X$ is connected), 
and if $X$ is connected, then the functor $G(X \wedge -) \simeq
\Omega(X \wedge -)$
commutes with realizations 
(using Lemma~\ref{lem:Waldhausen-fibration-lemma}).
Lemma~\ref{lem:F-commutes-with-real-Pnd-is-Pn} then implies  that
$P_r^d G(X) \simeq P_r G(X)$. This lets us translate the statement of
the theorem to the fibration:
\begin{equation*}
\left\lvert{\left( \Gamma_{r+1}^{\text{ext}} \right)^{*+1} G(X)}\right\rvert
\rightarrow 
G (X)
\rightarrow
P_r G(X)
.
\end{equation*}
One way to produce a direct comparison of this spectral sequence with
Curtis's result is to realize that in the setting of bisimplicial
groups, it becomes somewhat easier to understand the spectral sequence
we are using. 
Recall that a simplicial group $G_\cdot$ may be tranformed into a
nonabelian chain complex $NG$ by the process of normalization (for
details, see \cite[\S8.3, pp.~264--266]{Weibel:homological-algebra}).
To produce the normalized 
complex, one takes the intersection of the kernels of 
all but the last face map $d_0$, and uses $d_0$ for the
differential: $(NG)_n = \cap_{i>0} \ker d_i$ and $d_n:
(NG)_n\rightarrow (NG)_{n-1}$ is the same $d_n$ from the simplicial
group $G_\cdot$.
The homology of the resulting chain complex $H_*(NG)$
isomorphic to the homotopy of the original simplicial group, $\pi_*
G$. Using this process, we see that the spectral sequence of
Corollary~\ref{cor:offset-BF-ss} arises from taking the homology of
the bicomplex of nonabelian groups:
$$ 
G(X) \leftarrow 
\Gamma_{r+1}^{\text{ext}} G(X)
\leftarrow
N \left( \Gamma_{r+1}^{\text{ext}}\right) ^2 G(X)
\leftarrow 
\cdots
$$
This bicomplex maps to its ``good (horizontal) truncation'':
$$ G(X) \xleftarrow{\epsilon} \Image(\epsilon) \leftarrow 0 \leftarrow
\cdots ,$$
where $\Image(\epsilon) \cong \Gamma_{r+1} G(X)$, as noted above.
Now taking homology horizontally first, then vertically, causes the
spectral sequence to collapse to $\pi_* \left( G(X) /\Gamma_{r+1} G(X)
\right)$, which is the approximation Curtis uses.
Taking homology vertically first produces:
$$ \pi_* G(X) \leftarrow \pi_* \Gamma_{r+1} G(X) \leftarrow 0 ;$$
then when we take homology horizontally, we can identify the resulting
$E^2$ page as a quotient the $E^2$ page from
Theorem~\ref{cor:offset-BF-ss}, with an isomorphism in the first
column and a surjection in the second column.
This shows that the approximation Curtis uses can be regarded as
the good truncation of an approximation arising from Goodwillie's
calculus. The good truncation we used produces $H_0$, so in this
sense, Curtis's approximation $G/\Gamma_{r+1} G$ is $H_0$ of the
approximation used by Goodwillie calculus. Looking at
this another way, the Goodwillie tower can be thought of as the
derived functor of the quotient by the lower central series.
From the Blakers-Massey Theorem (\ref{thm:blakers-massey}), the
functor $G(X) \simeq \Omega (X)$ satisfies the stable excision
condition $E_n(n-1)$. This implies that for $X$ simply connected, 
the map $G(X)\rightarrow P_r G(X)$ is $(r-1)$-connected. Comparing
this with Curtis's result, we see that the higher columns in the
spectal sequence make a tremendous difference; they raise the 
connectivity from roughly $\log_2 r$ up to roughly $r$.






%
%
\chapter{Different Theories Of Calculus}
\label{chap:different-theories}
In this chapter, we present a way of combining Goodwillie's excisive
calculus and the additive calculus, and show there is a filtration of
distinct ``theories'' beginning with the additive calculus and ending
in the excisive calculus. The additive calculus is in many cases
easier to work with than the excisive calculus, since it involves no
suspensions. In one respect, this filtration shows how many
suspensions are necessary before the additive and excisive
approximations become the same.

The building block for the excisive calculus is the $T_n$ functor that
increases $n$-excisiveness. For additive calculus
(recall Section~\ref{sec:additive-calculus}), the ``building block''
appears to be $T_n (L_0 F_X)$; that is, $T_n$ applied to the left Kan
extension of $F$ along coproducts of $X$. In the case $n=1$, we have
$T_1 F(X) = \Omega F(S^1 \wedge X) $ and $T_1 (L_0 F_X) (S^0) = \Omega
\realization{F ( S^1_\cdot \wedge X) }$.  However, if the latter functor is
iterated, it produces $\Omega\realization{ \Omega\realization{
    F(S^1_\cdot \wedge S^1_\cdot \wedge X)}}$, which need not have a
good behavior with respect to coproducts. 
To produce the desired additive approximation, one must
apply all of the iterations of  $T_n$ to the same left Kan extension;
\emph{i.e.}, $T_n^k (L_0 F_X) (S^0)$.

This limits the number of serious candidates for interpolating between
additive and excisive calculus to just a few. Beginning with the
sequence of natural transformations
$$ \text{Id} \rightarrow T_n \rightarrow T_n^2 \rightarrow \cdots
\rightarrow P_n ,$$
we can either apply $\Pnd$ to this sequence, or apply each
functor in the sequence to $\Pnd$.
Applying $\Pnd$ to the stabilization sequence is a good approach
because it isolates complicated behavior that spreads across
dimensions. 
The stabilization of $\Pnd$ (that is, applying this sequence of
functors to $\Pnd$) is also interesting, and we will say some things
about it. Actually, the two are very closely related; for functors
whose target category is spectra, they are equivalent. However, for
functors whose target category is spaces, they are not.

\section{Background And Basic Facts}

\begin{lemma}
\label{lem:F-commutes-with-real-Pnd-is-Pn}
If $F$ commutes with realizations and satisfies the limit axiom, then
$\Pnd F(X) \simeq P_n F(X)$. 
\end{lemma}
\begin{proof}
We begin with $\Pnd F(X) = P_n (L_0 F_X)(S^0)$. Now if $F$ commutes
with realizations, then $L_0 F_X \simeq F_X$ because they agree on
finite sets, and hence all discrete sets by the limit axiom. This
means that they agree levelwise on all simplicial sets, and hence also
after realization (Lemma~\ref{lem:realization-lemma}).  The right hand
side is therefore equivalent to $P_n F(X)$.
\end{proof}

\begin{lemma}
\label{lem:loop-commutes-with-PoneD}
If $\pi_0 F$ is additive and $F(X)$ is connected or a grouplike
$H$-space, then $P_1^d (\Omega F(X)) \simeq \Omega P_1^d F(X)$.
\end{lemma}
\begin{proof}
Under these conditions, the spectral sequence of
Theorem~\ref{thm:additive-spectral-sequence} converges. There is a map
$P_1^d (\Omega F(X))\rightarrow \Omega P_1^d F(X)$ comparing the two
functors. Comparing the spectral sequence for $P_1^d (\Omega F(X))$
and $P_1^d F(X)$, we see that they are the same except for a
dimension shift and the difference in $\pi_0$ that is lost when
$\Omega$ is applied.
\end{proof}

\begin{definition}
For convenience, let us define $P_n^{(a)}$ to denote the functor $P_n^d
T_n^a$. 
\end{definition}

\begin{proposition}
Let $F$ be a reduced $r$-analytic functor from spaces to spaces
satisfying the limit axiom (\ref{def:limit-axiom}), 
with universal analyticity constant $c$ as in
Theorem~\ref{thm:analytic-functor-commutes-with-realization}. 
Let $M= \max(r,-c)$. Then:
\begin{enumerate}
\item $P_n^d F(X) \simeq P_n F(X)$ if $X$ is at least $M$-connected.
\item $P_n^{(m)} F \simeq P_n F$ for all for all $m > M$; and
\end{enumerate}
\end{proposition}
\begin{proof}
Theorem~\ref{thm:analytic-functor-commutes-with-realization} shows
that $F$ commutes with realizations of simplicial $M$-connected
spaces. Lemma~\ref{lem:F-commutes-with-real-Pnd-is-Pn} then shows that
$P_n^d F(X) \simeq P_n F(X)$ for $X$ at least
$M$-connected. 
The space $T_n F(X)$ is computed by evaluating a homotopy inverse limit
of a diagram involving suspensions of $X$. If $X$ is $k$-connected,
then $T_n F(X)$ depends only on $F$ evaluated on $(k+1)$-connected
spaces. Hence $T_n^m F(X)$ depends only on $F$ evaluated on
$(k+m)$-connected spaces. The minimum connectivity of a pointed space
is $-1$, so $T_n^m F$ depends only on $F$ on $(m-1)$-connected
spaces, and hence under the condition that $m>M$, it commutes with
realizations. Then the preceding paragraph shows that $P_n^{(m)} F =
P_n^d T_n^m F \simeq P_n T_n^m F \simeq P_n F$, as desired.
\end{proof}

\section{The Hilton-Milnor Theorem And Whitehead Products}

In order to make computations, we will need to know a bit about
Whitehead products and make use of the Hilton-Milnor theorem.

Given $\alpha\in \pi_n X$ and $\beta\in\pi_m X$, there is a product
$[\alpha,\beta]$ called the Whitehead product that behaves much like a
Lie bracket. For details, see
\cite[\S IX.7]{Whitehead:elements-of-homotopy-theory}.

To define the Whitehead product, consider $\alpha$ as a map of pairs
$(D^n, \partial D^n) \rightarrow (X, \basept{})$, and $\beta$
similarly. If $X$ were an $H$-space, we could define a map on the
product $D^n \times D^m$ by $(x,y) \mapsto \mu(\alpha(x),\beta(y))$.
Since that may not be the case, we can consider the map from $S^{n+m-1}
\cong \partial (D^n
\times D^m) = (\partial D^n)\times D^m \cup D^n \times (\partial D^m)$
given by $\alpha$ on $D^n$ or $\beta$ on $D^m$. This definition makes
sense since $\alpha
\rvert_{\partial D^n} = \basept{}$, and similarly $\beta$, so either
one or the other is $\basept{}$ for every point in the domain.
This defines a map $[\alpha,\beta]: S^{n+m-1} \rightarrow
X$. This element of $\pi_{n+m-1} X$ is called the Whitehead product of
$\alpha$ and $\beta$.

The main fact about Whitehead products that we will make use of is
that they are graded commutative.
\begin{lemma}
(\cite[\S X.7.5, p.~474]{Whitehead:elements-of-homotopy-theory})
If $\alpha\in \pi_n X$ and $\beta\in\pi_m X$, then $[\beta,\alpha] =
(-1)^{n m} [\alpha,\beta]$.
\end{lemma}

For our purposes, very little of the full strength of the
Hilton-Milnor theorem will be needed, so we will give as few details
as possible.
\begin{theorem}[Hilton-Milnor]
(\cite[\S XI.6, Theorem~6.6]{Whitehead:elements-of-homotopy-theory})
The space $\bigvee^k S^{n}$, for $n>1$, has the
same homotopy type as the weak product $\prod S^{w_j}$, where $w_j$ is
a sequence of integers beginning with $w_j = n$ for $j=1, \ldots, k$,
then increasing in steps of $n-1$. The equivalence sends $S^{w_j}$ to
a certain iterated Whitehead product, with the weight $m$ Whitehead
product corresponding to a sphere of dimension $S^{n+(m-1) (n - 1)}$.
\end{theorem}

\begin{lemma}
\label{lem:degree-of-pi-m-s-n}
  The functor $\pi_{m}(S^{n} \wedge -)$ is degree 
  $k = \lfloor{(m-1)/(n-1)} \rfloor$.
\end{lemma}
\begin{proof}
  The Hilton-Milnor Theorem implies that the homotopy $\pi_{m}$ of
  $\bigvee S^{n}$ is determined by the basic products of
  weight at most $k$. Verification involves checking that the
  a product of weight $\le k$ corresponds to a sphere of dimension $k
  (n-1)+1 \le m$, and that a product of weight $k+1$ corresponds to a
  sphere of dimension $(k+1)(n-1)+1 = k(n-1) + n > m$.
\end{proof}
\begin{example}
\label{ex:pi-2n-1-s-n}
In particular, Lemma~\ref{lem:degree-of-pi-m-s-n} shows that
$\pi_{2n-1} (S^n \wedge -)$ is the lowest homotopy group of $(S^n
\wedge -)$ that is degree $2$. Computing $\pi_3 (S^2 \vee S^2)$, for
example, we find two copies of $\pi_3(S^2) \cong \mathbb{Z}$ and one
copy of $\pi_3 S^3 \cong \mathbb{Z}$. Therefore, $\Perp_2 \pi_3(S^2
\wedge -) \cong \mathbb{Z}$. Actually, for every $n$, we have the
same: $\Perp_2 \pi_{2n-1} (S^n \wedge -) \cong \mathbb{Z}$, since the
first new sphere in $S^n \vee S^n$ corresponds to the Whitehead
product $[i_1, i_2]$ of the two inclusions $S^n \rightarrow S^n \vee
S^n$, and is a copy of $S^{2n-1}$. We will use this soon in our
computations. 
\end{example}

\section{Examples Showing Theories Are Distinct}

Notice that Lemma~\ref{lem:F-commutes-with-real-Pnd-is-Pn} shows that
under at least some circumstances, $\Pnd F$ and $P_n F$ coincide, so
we need to provide some evidence that they can differ.

\begin{example}
One example that is both trivial and fundamental is the functor $F(X)
= K(H_2(X),2)$, as in Example~\ref{ex:K-2-2}.  This is a functor whose
values are always simply connected, but $P_1 F(X) = 0$ (in fact, the
whole excisive Taylor tower is zero).  We can compute $P_1^d F(S^2)
\cong K(\mathbb{Z},2)$, since $H_2 (\bigvee^k S^2) \cong \oplus^k
\mathbb{Z}$. This shows that $\Pnd$ and $P_n$ can be very different,
and also illustrates that the connectivity of the values of $F$ have
nothing to do with that difference. As an aside, the radius of
convergence of $F$ is $2$; that is, $F$ is equivalent to the inverse
limit of its Taylor tower on $2$-connected spaces (because both are
contractible there). 
\end{example}

While important, the preceding example is not such a satisfying way of
demonstrating a difference between additive and excisive. Notice,
though, that with slight modifications, it does produce one family of
examples for which all of the $P_1^{(a)}$ functors are
distinct. For the family of functors $F_b (X) = K(H_b(X),b)$, we can
compute that $\pi_b P_n^{(a)} F_b (X) \cong H_{b-a} X$. In particular,
for a given $X$ and large enough $a$, this is zero, whereas when $a$
is small relative to the dimension of $X$, there are many spaces $X$
for which it is not zero (\emph{e.g.}, $X = \bigvee_{k=0}^{k=a} S^k$).

Using Theorem~\ref{thm:additive-spectral-sequence}, we will show that
there is another way in which the $P_n^{(a)}$ functors can be
different.  To make the computations tractable, we will work only with
$n=1$.

\begin{example}
Fix an $a$ and consider the functor $F(X) = \Omega^{3a} (S^{a}
\wedge X)$. We will establish that $P_1^{(a)} F \xrightarrow{\not \simeq}
P_1^{(a+1)} F$ by evaluating both at $S^0$.

First, we will show that $P_1^{(a+1)} F \simeq P_1 F$.
\begin{lemma}
\label{lem:example-of-difference-1}
$P_1^{(a+1)} F(S^0) \simeq P_1 F(S^0)$
\end{lemma}
\begin{proof}
  We have:
  \begin{equation*}
    T^{a+1} F = \Omega^{4a+1} (S^{2a+1} \wedge - ) ,
  \end{equation*}
  and $\pi_j (S^{2a+1} \wedge -)$ is linear for $j\le 4a+1$ (by
  Lemma~\ref{lem:degree-of-pi-m-s-n}), so by
  Lemma~\ref{lem:loop-commutes-with-PoneD}, 
  \begin{align*}
    P_1^{(a+1)} F &= P_1^d \Omega^{4a+1} (S^{2a+1} \wedge - )
    \\
    &\simeq  \Omega^{4a+1} P_1^d (S^{2a+1} \wedge - ) , 
           \intertext{and since $(S^{2a+1}\wedge -)$ commutes with
             realizations, by
             Lemma~\ref{lem:F-commutes-with-real-Pnd-is-Pn} this is} 
\\
    &\simeq  \Omega^{4a+1} P_1  (S^{2a+1} \wedge - )  \\
    &\simeq  P_1 F ,
  \end{align*}
where the last equivalence follows because $P_1 F$ is 1-excisive, so
$\Omega^{a+1} P_1 F (S^{a+1} \wedge X) \simeq P_1 F(X)$.
\end{proof}

Next, we will show that $P_1^{(a)} F$ is not equivalent to $P_1 F$. We
will do this by mapping to another functor that \emph{is} equivalent
to $P_1 F$, for exactly the same reason as the functor in
Lemma~\ref{lem:example-of-difference-1}.

The functor $T_1^a F$ is $\Omega^{4a} ( S^{2a} \wedge -)$, which for
our purposes we will view as $\Omega G$, for $G(-) = \Omega^{4a-1}
(S^{2a} \wedge -)$.
Using Lemma~\ref{lem:degree-of-pi-m-s-n}, we see that
$G$ has $\pi_0$ quadratic. In fact,
putting $v=2a$, we see that this functor is the functor $\pi_{2v-1}
(S^v \wedge -)$ from Example~\ref{ex:pi-2n-1-s-n}, so $\pi_0$ of
$\Perp_2 G$ is actually $\mathbb{Z}$. 
Our assertion is that
the map $P_1^d\Omega G \rightarrow \Omega P_1^d G$ is not even $1$-connected,
whereas the latter functor is equivalent to $P_1 F$. To show that this
map is not $1$-connected, we will use the spectral sequence of
Theorem~\ref{thm:additive-spectral-sequence}. For computing with that
spectral sequence, it will be helpful to note that cross effects
commute with taking homotopy groups
(Lemma~\ref{lem:pi-k-perp-is-perp-pi-k}); we will use this fact
without further comment.

Recall from Lemma~\ref{lem:perp-homotopy-orbits-spectrum}
that when $\Perp_{n+1} F \simeq 0$,
 the complex $\Perp^{*+1}_n \mathbf{F}$ has the realization 
$(\Perp_n \mathbf{F})_{h\Sigma_n}$, 
and homotopy orbits can also be expressed as the group
homology of $\Sigma_n$. In our case, $n=2$, and the complex is, as
usual for $H_*(\Sigma_2; -)$, one with only one nondegenerate cell in
each dimension. By coincidence, we actually have the complex for
$H_*(\Sigma_2; \mathbb{Z})$, since $\pi_0 \Perp_2 \pi_{2v-1}(S^v
\wedge -) \cong \mathbb{Z}$, but this would not be the case if, for
example, we increased the dimension of the sphere involved by
one.  

The bottom row of the spectral sequence of
Theorem~\ref{thm:additive-spectral-sequence} arises from the augmented
complex $F \leftarrow \Perp^{*+1}_n F$, so $E_{0,0}$ is the group
$\pi_{2v-1} S^v$. We will not need to know anything about it, just
that the complex is still exact with the augmentation map (which it
is), so the only differences between the (shifted) $H_{*-1} (\Sigma_2;
\mathbb{Z})$ and $E^2_{*,0}$ are in dimension $0$ (obviously) and $1$. 

In the particular case of our functor, we have a good amount of
information about the bottom row. On the $E^1$ page, we do not know
$\pi_{2v-1} S^v$, but we have $v$ even, so we know by Serre's work
that it contains an infinite cyclic factor (\emph{e.g.},
\cite[Theorem~18.22, p.~254]{Bott-Tu}). The rest of the row consists
of factors of $\mathbb{Z}$. As usual when computing the
homology of a cyclic group, there is only one non-degenerate copy of
$\mathbb{Z}$ in each dimension. Furthermore, the differentials
$d^1_{i,0}$ for $i>1$ are either multiplication by $2$ or $0$, as we
will now explain. 
Each copy of $\mathbb{Z}$ corresponds to a Whitehead product of two of
the inclusion maps $S^v \rightarrow \bigvee S^v$, and we have $v$
even, so the graded commutativity of the Whitehead product means that
$[\alpha,\beta] = [\beta,\alpha]$. Specifically, $d^1_{2,0} = 0$ since
the non-degenerate copy of $\mathbf{Z}$ is represented by the product
$[i_0 i_1 , i_1 i_0 ]$, which is sent under the differential $d
= \partial_0 - \partial_1 $ to 
$$ [i_1, i_0] - [i_0, i_1] = [ i_1,i_0] - [i_1, i_0] = 0 .$$
Similarly, $d^1_{3,0} = 2$. 

The group $E^1_{0,0}$ contains an infinite cyclic factor as mentioned,
but the spectral sequence converges to stable homotopy, so 
$E^\infty_{0,0} = \pi_{2v-1}^S S^v$ is known to be torsion. Hence
the differential $d^1_{1,0}$ must be injective.

This gives us enough information about the spectral sequence to
determine that the map $P_1^d\Omega G \rightarrow \Omega P_1^d G$ is
not $1$-connected. In low dimensions, the $E^2$ page of the spectral
sequence for $P_1^d G$ is the following:
  $$
\xymatrix{
E^{2}_{0,2}  \\
E^{2}_{0,1} & E^{2}_{1,1}   \\
E^{\infty}_{0,0} & 0 & 
\mathbb{Z}/2 \ar[ull]^{d^{2}_{2,0}} 
& 0 \ar[ull] \ar[ull]^{d^{2}_{3,0}} \ar@{-->}[-2,-3]_{d^{3}_{3,0}} 
} .
  $$
If $d^2_{2,0}$ is nonzero, then the spectral sequence for $P_1^d
\Omega G$ (which is just a shifted version of the one for $P_1^d G$)
converges to $\pi_0 P_1^d\Omega G = E^2_{0,1}$, whereas $\pi_0 \Omega
P_1^d G = E^{\infty}_{0,1}$. 

Now consider the case in which $d^2_{2,0}$ is zero. The group
$E^2_{3,0} = 0$, and hence supports no differentials, so the spectral
sequence for $P_1^d \Omega G$ produces the same terms corresponding to
$E^2_{0,2}$ and $E^2_{1,1}$ as that for $P_1^d G$ (these groups live
to $E^\infty$ because they are not the target of any more
differentials).  However, $\pi_1 \Omega P_1^d G$ is an extension of
$E^2_{2,0}\cong \mathbb{Z}/2$ by these groups, whereas $\pi_1 P_1^d
\Omega G$ has no such factor, so they differ in the abutment.
\end{example}

This example actually shows the difference between the two
alternative constructions mentioned at the start of the chapter, since
$P_n$ commutes with $\Omega$ always, and $\Pnd$ was just shown not to. 
\begin{lemma}
The functor 
$\Pnd F$ need not be equivalent to either $T_n \Pnd F$ or $\Pnd T_n
F$.
$\qed$
\end{lemma}


\backmatter

\bibliographystyle{amsplain}

%
%
\providecommand{\bysame}{\leavevmode\hbox to3em{\hrulefill}\thinspace}
\providecommand{\MR}{\relax\ifhmode\unskip\space\fi MR }
\providecommand{\MRhref}[2]{%
  \href{http://www.ams.org/mathscinet-getitem?mr=#1}{#2}
}
\providecommand{\href}[2]{#2}

%
%
\chapter{Vita}

Andrew Mauer-Oats was supported for several years
by Department of Education ``Graduate Assistance in Areas of National
Need'' (GAANN) fellowships and as a teaching assistant.

\end{document}